\documentclass[a4paper,12pt]{amsart}
\usepackage[top=3cm,bottom=3cm,outer=3cm,inner=2cm,marginpar=2.45cm]{geometry}

\makeatletter
\@ifundefined{ifannalen}{
  \newif\ifannalen
  \annalenfalse
}

\@ifundefined{ifcompositio}{
  \newif\ifcompositio
  \compositiofalse
}
\makeatother

\usepackage[frenchb,ngerman,english]{babel}
\usepackage[utf8]{inputenc}
\usepackage[T1]{fontenc}
\usepackage{enumitem}
\usepackage{tabu}
\usepackage[all]{xy}
\usepackage[abbrev]{amsrefs}
\usepackage{mathtools} 
\usepackage[usenames,dvipsnames]{xcolor}

\babeltags{de = ngerman}
\babeltags{fr = frenchb}

\IfFileExists{marionotations.sty}{
  \usepackage[no-mario]{marionotations}
}{
  \usepackage[no-mario]{mynotations}
}

\usepackage[final]{hyperref}

\newcommand{\tmin}{{\text{min}}}

\newcommand{\numax}{\nu_{\text{max}}}
\NewDocumentCommand{\sm}{s m}{{#2}\IfBooleanTF{#1}{_}{^}\text{sm}}

\newcommand{\tlog}{{\text{log}}}

\NewDocumentCommand{\dep}{t{_} d<> O{k} m}{#4\IfBooleanTF{#1}{_}{^}{\IfNoValueF{#2}{#2\:}(#3)}}

\newcommand{\obstidlSHM}{\ideal{I}_{\Xi}^S}

\newcommand{\lineq}{\sim} 

\NewDocumentCommand{\lelong}{m O{x}}{\nu\paren{#1,#2}}

\newcommand{\mtidlof}[2][]{\multidl_{#1}\paren{#2}} 
\newcommand{\sBase}[2][]{\mathbf{B}_{#1}\paren{#2}} 
\DeclareMathOperator*{\reglim}{reg-lim}  
\DeclareMathOperator{\Ann}{Ann}  
\DeclareMathOperator{\lc}{lc} 

\NewDocumentCommand{\Ohvol}{ 
  D<>{\dep\vphi_L} 
  D<>{S} 
  D<>{\omega} 
  O{\psi_{S}}
}{\:d\vol_{#2,#3,#1}[#4]} 
\NewDocumentCommand{\gOhvol}{ 
  O{m_1}          
  D<>{\dep\vphi_L}  
  D<>{S}            
  D<>{\omega}       
  O{\dep\psi}     
  m
}{\abs{J^{#1}#6}_{#4}^2 \Ohvol<#2><#3><#4>[#5]} 

\NewDocumentCommand{\lcV}{ 
  D||{\sigma}      
  !D<>{\vphi_L}   
  !D<>{\omega}    
  O{\psi}        
}{\:d\operatorname{lcv}^{#1}_{#3,#2}\left[#4\right]}

\NewDocumentCommand{\idxup}{ 
  m          
  O{\omega}  
}{\paren{#1}^{\mathrlap{\!#2}}}

\newcommand{\rs}[1]{\widetilde{#1}} 
\newcommand{\currInt}[1]{\left[#1\right]} 
\newcommand{\Berg}{\mathcal{B}} 

\newcommand{\orgu}{u}
\newcommand{\extu}[1][\orgu]{\widetilde{#1}}
\newcommand{\holU}{U}
\newcommand{\Umin}{\holU^\tmin}
\newcommand{\wphi}{\widetilde{\vphi}}
\newcommand{\bphi}{\boldsymbol{\vphi}}

\newcommand{\setE}{\mathcal{E}} 
\NewDocumentCommand{\lcc}{ 
  O{\sigma}                
  D<>{X}                   
  D(){S}
}{\lc_{#2}^{#1}\paren{#3}}

\newtheorem{prop}{Proposition}[subsection]
\newtheorem{thm}[prop]{Theorem}

\newtheorem{lemma}[prop]{Lemma}

\theoremstyle{remark}
\newtheorem{remark}[prop]{Remark}

\theoremstyle{definition}
\newtheorem{example}[prop]{Example}
\newtheorem{notation}[prop]{Notation}
\newtheorem{SNCassumption}[prop]{Snc assumption}
\newtheorem{definition}[prop]{Definition}

\numberwithin{equation}{subsection}

\allowdisplaybreaks  

\begin{document}

\ifcompositio
  \newcommand{\thmparen}[1]{(#1)}
\else
  \newcommand{\thmparen}[1]{#1}  
\fi

\newcommand{\titlestr}{Extension with log-canonical measures and an
  improvement to the plt extension of Demailly--Hacon--P\u aun}

\newcommand{\shorttitlestr}{Extension with lc-measures and improvement
  to the plt extension}

\newcommand{\MCname}{Tsz On Mario Chan}
\newcommand{\MCnameshort}{Mario Chan}
\newcommand{\MCemail}{mariochan@pusan.ac.kr}

\newcommand{\YJname}{Young-Jun Choi}
\newcommand{\YJnameshort}{Young-Jun Choi}
\newcommand{\YJemail}{youngjun.choi@pusan.ac.kr}

\newcommand{\addressstr}{Dept.~of Mathematics, Pusan National
  University, Busan 46241, South Korea}

\newcommand{\subjclassstr}[1][,]{32J25 (primary)#1 32Q15#1 14E30 (secondary)}

\newcommand{\keywordstr}[1][,]{$L^2$ extension#1 Ohsawa--Takegoshi
  extension#1 purely log-terminal#1 lc centres}

\newcommand{\thankstr}{%
  Most part of this project was done when the first author was under
  the support of the National Center for Theoretical Sciences (NCTS)
  in Taiwan and while the second author had a short term visit there.
  The first author is thankful for the trust from Director Jungkai
  Chen of the NCTS and the second author would like to thank Director
  Chen for his hospitality.
  Both authors are in debt to Jean-Pierre Demailly who shared his
  ideas around the subject of extension and a lot more unreservedly
  with the authors, especially during their stay at the Institut Fourier
  which was supported by the ERC grant ALKAGE (No.~670846).
  The authors are grateful to Bo Berndtsson and Mihai P\u aun for their
  comments on the lc-measures and the potential extension theorem in
  the higher codimensional cases.
  Many thanks also to Chen-Yu Chi for drawing the first author's attention
  to the study of analytic adjoint ideal sheaves as well as pointing
  out several mistakes in the first draft of this paper, and to Dano Kim for
  the advice on the development of $L^2$ extension theorems.
  This work was supported by the National Research Foundation 
  (NRF) of Korea grant funded by the Korea government
  (No.~2018R1C1B3005963). 
}

\ifannalen
  \title{\titlestr\thanks{\thankstr}}
  \titlerunning{\shorttitlestr}        
  \author{\MCname   \and
    \YJname
  }
  \authorrunning{\MCnameshort \and \YJnameshort} 
  \institute{\MCname \at
    \email{\MCemail}           
    \and
    \YJname \at
    \email{\YJemail} \\
    \addressstr
  } 
\else
  \title[\shorttitlestr]{\titlestr}
  \author[\MCnameshort]{\MCname}
  \email{\MCemail}
  \address{}
  %
  \author{\YJname}
  \email{\YJemail}
  \address{\addressstr}
  \thanks{\thankstr}
  \subjclass[2010]{\subjclassstr}

  \ifcompositio
  \fi 


  \dedicatory{Dedicated to Prof.~Fabrizio Catanese on the occasion of
    his 71st birthday}

  \begin{abstract}

With a view to proving the conjecture of ``dlt extension'' related to
the abundance conjecture, a sequence of potential candidates for
replacing the Ohsawa measure in the Ohsawa--Takegoshi $L^2$ extension
theorem, called the ``lc-measures'', which hopefully could provide the
$L^2$ estimate of a holomorphic extension of any suitable holomorphic
section on a subvariety with singular locus, are introduced in the
first half of the paper.
Based on the version of $L^2$ extension theorem proved by Demailly, a
proof is provided to show that the lc-measure can replace the Ohsawa
measure in the case where the classical Ohsawa--Takegoshi $L^2$
extension works, with some improvements on the assumptions on the
metrics involved.
The second half of the paper provides a simplified proof of the result
of Demailly--Hacon--P\u aun on the ``plt extension'' with the
superfluous assumption ``$\supp D \subset \supp\paren{S+B}$'' in their
result removed.
Most arguments in the proof are readily adopted to the ``dlt
extension'' once the $L^2$ estimates with respect to the lc-measures
of holomorphic extensions of sections on subvarieties with singular
locus are ready.


  \end{abstract} 

\fi 

\date{\today} 

\maketitle

\ifannalen 
  \begin{abstract}

With a view to proving the conjecture of ``dlt extension'' related to
the abundance conjecture, a sequence of potential candidates for
replacing the Ohsawa measure in the Ohsawa--Takegoshi $L^2$ extension
theorem, called the ``lc-measures'', which hopefully could provide the
$L^2$ estimate of a holomorphic extension of any suitable holomorphic
section on a subvariety with singular locus, are introduced in the
first half of the paper.
Based on the version of $L^2$ extension theorem proved by Demailly, a
proof is provided to show that the lc-measure can replace the Ohsawa
measure in the case where the classical Ohsawa--Takegoshi $L^2$
extension works, with some improvements on the assumptions on the
metrics involved.
The second half of the paper provides a simplified proof of the result
of Demailly--Hacon--P\u aun on the ``plt extension'' with the
superfluous assumption ``$\supp D \subset \supp\paren{S+B}$'' in their
result removed.
Most arguments in the proof are readily adopted to the ``dlt
extension'' once the $L^2$ estimates with respect to the lc-measures
of holomorphic extensions of sections on subvarieties with singular
locus are ready.


    \subclass{\subjclassstr[\and]}
  \end{abstract} 
\fi
  

\section{Introduction}
\label{sec:introduction}

This work is the first step towards generalising the result in
\cite{DHP}, namely the extension theorem on \emph{purely log-terminal (plt)}
pairs, to an extension theorem on \emph{divisorially log-terminal
  (dlt)} pairs.
The latter extension theorem is essential in proving the Abundance
Conjecture in algebraic geometry (see, for example,
\cite{Fujino_abundance-3fold}, \cite{DHP},
\cite{Fujino&Gongyo_abundance} and \cite{Gongyo&Matsumura}).

There are two main results in this paper.
The first one is an \emph{Ohsawa--Takegoshi-type $L^2$ extension theorem
which replaces the Ohsawa measure in the estimate by a measure
supported on the log-canonical (lc) centres of a given subvariety}
(Theorems \ref{thm:ext-from-lc-with-estimate-codim-1-case} and
\ref{thm:extension-with-seq-of-potentials}).
Such measure (called ``lc-measure'', see Definition
\ref{def:lc-measure}) seems to be well-suited to the use in 
birational geometry and can possibly provide the best possible
estimates for minimal holomorphic extension with universal constant
(see Example \ref{eg:Berndtsson-example}).
The current result, following the line of thought and formulation
given in \cite{Demailly_extension}, essentially recovers the classical
Ohsawa--Takegoshi $L^2$ extension theorem from codimension-$1$
subvarieties (in which the section to be extended should vanish on the
singular locus), with some relaxation on the assumptions on the given
metrics and auxiliary functions in the formulation and improvement in
the estimate.

The second result is an \emph{improvement to the result of the plt extension
of Demailly--Hacon--P\u aun in \cite{DHP}} in the sense that a
superfluous assumption is removed in their theorem with a simplified
proof (see Theorem \ref{thm:improved-plt-intro} or
\ref{thm:proof-improved-plt}).
Although the proof presented in this paper makes use of the $L^2$
extension with respect to the lc-measure, one may also use the version of
$L^2$ extension with respect to the Ohsawa measure (for example, the
version in \cite{Demailly_extension}) together with the suitable
slight improvement to the setup as presented in Section
\ref{sec:setup}.
However, it is the authors' belief that the lc-measures will play a
role in the future proof of the ``dlt extension''.
That's why the two independent results are presented together to
emphasise their linkage.

\subsection{Background}
\label{sec:background}

Readers are referred to the survey by Varolin
(\cite{Varolin_lecture-note-on-L2}) for a quick outlook of the
development of the celebrated Ohsawa--Takegoshi extension theorem
since the paper \cite{Ohsawa&Takegoshi-I} by Ohsawa and Takegoshi.
They are also referred to \cite{Blocki_Suita-conj},
\cite{Guan&Zhou_optimal-L2-estimate} and \cite{Berndtsson&Lempert} for
some later development on the extension with optimal estimates.

For the background on the abundance conjecture, and the relevant
non-vanishing conjecture as well as the conjecture on dlt extension,
readers are referred to \cite{Fujino_abundance-3fold}, \cite{DHP},
\cite{Fujino&Gongyo_abundance} and \cite{Gongyo&Matsumura}.
Among them, the work of Gongyo and Matsumura in
\cite{Gongyo&Matsumura} provides a proof to the dlt extension with a
strong assumption via the $L^2$ injectivity theorem
(\cite{Matsumura_injectivity}), while that of Demailly, Hacon and P\u
aun in \cite{DHP} proves the extension theorem for plt pairs via the
Ohsawa--Takegoshi $L^2$ extension theorem, whose technique is followed
closely in this paper.

By the time when this paper is finished, the authors are notified that
Mihai P\u aun and Junyan Cao are finishing their version of $L^2$
extension theorem (\cite{Cao&Paun_OT-ext}) which is aiming for the dlt extension
(\cite{DHP}*{Conj.~1.3}).
The authors are also aware of the work of Chen-Yu Chi on the quantitative
extension of holomorphic sections from unions of strata of divisors
(\cite{Chi_extension}).


The present work takes off from the work of Demailly in
\cite{Demailly_extension}.
Let $X$ be a weakly pseudoconvex \textde{Kähler} manifold, $K_X$ its
canonical bundle and $(L, e^{-\vphi_L})$ a hermitian line bundle on $X$
equipped with a hermitian metric $e^{-\vphi_L}$ which is possibly
singular.
In \cite{Demailly_extension}, Demailly proves a new version of the
Ohsawa--Takegoshi $L^2$ extension theorem applicable to the questions
on extending $\:\parres{K_X\otimes L}_S$-valued holomorphic sections on
possibly non-reduced subvarieties $S$ defined via multiplier ideal
sheaves (a feature which can be considered as a far-reaching
generalisation to the result in the work of Dano Kim,
\cite{KimDano_Thesis} and \cite{KimDano_lc-extension}, in which an
$L^2$ extension theorem for extending holomorphic sections on maximal
log-canonical centres of some log-canonical pairs $(X,D)$ is proved).
An interesting new input of this version of $L^2$ extension theorem is that,
if the ambient manifold $X$ is \emph{compact} (or if it is
holomorphically convex, see \cite{Cao&Demailly&Matsumura}), via a
brilliant use of the Hausdorff-ness of the topology on the relevant cohomology
groups, a holomorphic extension of a $\:\parres{K_X \otimes L}_S$-valued
section $f$ on a subvariety $S$ can be assured without the need
of any $L^2$ assumption on the section $f$ with respect to the Ohsawa
measure (provided that the suitable weak positivity assumption
involving $\vphi_L$ and $\psi$ still holds true), although $f$ is
still required to be locally extendible to some holomorphic section
in some multiplier ideal sheaf constructed from $\vphi_L$ (namely, the
multiplier ideal sheaf of $\vphi_L+m_0\psi$ in the notation in Section
\ref{sec:setup}).


Note that the Ohsawa measure in the estimates given by all different
versions of the Ohsawa--Takegoshi extension theorem diverges to infinity
around the singular points of the subvariety from which the given
holomorphic section is extended.
That's why,
%
although the extension theorem of Demailly without the $L^2$
assumption loses the estimate on the extended section, it was
considered advantageous since the sections to be extended would not
have to vanish along the singular locus of the subvariety. 
It was hoped that, with this feature in the new version of the
extension theorem, one could follow the arguments as in \cite{DHP} to
construct a suitable psh potential in order to prove the so called
``dlt extension'' (see \cite{DHP}*{Conj.~1.3}).

Unfortunately, in the course of proving the dlt extension,
in order to show that the given
$\:\parres{\mu\paren{K_X+S+B}}_S$-valued section on the subvariety $S$
(see Theorem \ref{thm:improved-plt-intro} or Section
\ref{sec:basic-setup-plt} for the notations $\mu$, $S$ and $B$) has
local extensions lying in the suitable multiplier ideal sheaf, one has
to prove via an approximation of the metric on $K_X +S+B$ (where each
approximating metric is constructed from some ``algebraic'' metric on
$K_X+S+B+\frac 1k A$ for some ample divisor $A$ and positive integer
$k$) and make use of the estimates provided by the Ohsawa--Takegoshi
$L^2$ extension theorem to prove convergence as in \cite{DHP} (see also
Theorem \ref{thm:bphi-bdd-below-by-u-on-S} for a relevant statement).
It follows that, in order to prove the dlt extension via the argument
in \cite{DHP}, the estimate in the $L^2$ extension theorem is
indispensable.

In view of this, the goal of the present work is to resume
the estimate of the Ohsawa--Takegoshi $L^2$ extension theorem under
Demailly's setting by replacing the (generalised) Ohsawa measure by 
the ``measure on log-canonical centres'', or the ``lc-measure'' for
short, which is defined in Definition \ref{def:lc-measure}.
The latter measure, instead of diverging to infinity around the
singular locus of the subvariety $S$, can indeed be supported in the
singular locus of $S$ (or on some lc centres of $(X,S)$ if $S$ is a
divisor).
This provides the means to get some sort of control over the $L^2$
norm of the holomorphic extensions, and eventually can be useful in
proving the dlt extension.

In the remaining of Section \ref{sec:introduction}, the main results
(Theorems \ref{thm:ext-from-lc-with-estimate-codim-1-case} and
\ref{thm:improved-plt-intro}) are presented.
A discussion on the lc-measures can be found in Section
\ref{sec:lc-measures}.
Example \ref{eg:Ohsawa-example} shows that the lc-measures
can filter out Ohsawa's example (\cite{Ohsawa-example}*{after
  Prop.~5.4}), while Example \ref{eg:Berndtsson-example} is a
computation by Bo Berndtsson of a concrete estimate of the minimal
holomorphic extensions in a fundamental example which, it turns out,
can be expressed in terms of lc-measures.
These give evidence that the lc-measures could be used for
providing $L^2$ estimates of holomorphic extensions in the general
situations.
Section \ref{sec:extension-from-mlc} is devoted to the proof of the
$L^2$ extension theorem with respect to the lc-measure supported on lc
centres of codimension $1$, while Section \ref{sec:improvement-DHP} is
devoted to the proof of the improvement to the plt extension of
\cite{DHP}.

\subsection{Notation}
\label{sec:notation}

In this paper, the following notations are used throughout.

\begin{notation}
  Set $\ibar := \ibardefn \;$.\ibarfootnote
\end{notation}

\begin{notation}
  Each potential $\vphi$ (of the curvature of
  a metric) on a holomorphic line bundle $L$ in the following
  represents a collection of local functions
  $\set{\vphi_\gamma}_\gamma$ with respect to some fixed local
  coordinates and trivialisation of $L$ on each open set $V_\gamma$ in
  a fixed open cover $\set{V_\gamma}_\gamma$ of $X$.  The functions
  are related by the rule
  $\vphi_\gamma = \vphi_{\gamma'} + 2\Re h_{\gamma \gamma'}$ on
  $V_\gamma \cap V_{\gamma'}$ where $e^{h_{\gamma \gamma'}}$ is a
  (holomorphic) transition function of $L$ on
  $V_\gamma \cap V_{\gamma'}$ (such that
  $s_\gamma = s_{\gamma'}e^{h_{\gamma \gamma'}}$, where $s_\gamma$ and
  $s_{\gamma'}$ are the local representatives of a section $s$ of $L$
  under the trivialisations on $V_\gamma$ and $V_{\gamma'}$
  respectively).
  Inequalities between potentials is meant to be the inequalities
  under the chosen trivialisations over open sets in the fixed open
  cover $\set{V_\gamma}_\gamma$.
\end{notation}

\begin{notation} \label{notation:potentials}
  For any prime (Cartier) divisor $E$, let
  \begin{itemize}
  \item $\phi_E := \log\abs{s_E}^2$, representing the collection
    $\set{\log\abs{s_{E,\gamma}}^2}_{\gamma}$, denote a potential (of
    the curvature of the metric) on the line bundle associated to $E$
    given by the collection of local representations
    $\set{s_{E,\gamma}}_{\gamma}$ of some canonical section $s_E$
    (thus $\phi_E$ is uniquely defined up to an additive constant);
    
  \item $\sm\vphi_E$ denote a smooth potential on the line
    bundle associated to $E$;
    
    
  \item $\psi_E := \phi_E - \sm\vphi_E$, which is a global function
    on $X$, when both $\phi_E$ and $\sm\vphi_E$ are fixed.
  \end{itemize}
  All the above definitions are extended to any $\fieldR$-divisor $E$
  by linearity.
  For notational convenience, the notations for a $\fieldR$-divisor
  and its associated $\fieldR$-line bundle are used interchangeably.
\end{notation}

\begin{notation}
  For any $(n,0)$-form (or $K_X$-valued section) $f$, define $\abs f^2
  := c_n f \wedge \conj f$, where $c_n :=
  (-1)^{\frac{n(n-1)}{2}}\paren{\pi\ibar}^n$.
  For any \textde{Kähler} metric $\omega =\pi\ibar \sum_{1\leq j,k\leq
    n} h_{j\conj k} \:dz^j \wedge d\conj{z^k}$ on $X$, set $d\vol_{X,\omega} :=
  \frac{\omega^{\wedge n}}{n!}$.
  Set also $\abs f_\omega^2 d\vol_{X,\omega} = \abs f^2$.
\end{notation}

\begin{notation}
  For any two non-negative functions $u$ and $v$,
  write $u \lesssim v$ (equivalently, $v \gtrsim u$) to mean that there
  exists some constant $C > 0$ such that $u \leq C v$, and $u
  \sim v$ to mean that both $u \lesssim v$ and $u \gtrsim v$ hold
  true.
  For any functions $\eta$ and $\phi$, write $\eta \lesssim_\tlog \phi$
  if $e^\eta \lesssim e^\phi$.
  Define $\gtrsim_\tlog$ and $\sim_\tlog$ accordingly.
\end{notation}

\subsection{Basic setup}
\label{sec:setup}


Let $(X,\omega)$ be a \emph{compact} \textde{Kähler} manifold of
complex dimension $n$, and let $\mtidlof{\vphi} := \mtidlof[X]{\vphi}$
be the multiplier ideal sheaf of the potential $\vphi$ on $X$ given at
each $x \in X$ by
\begin{equation*}
  \mtidlof{\vphi}_x := \mtidlof[X]{\vphi}_x
  :=\setd{f \in \holo_{X,x}}{
    \begin{aligned}
      &f \text{ is defined on a coord.~neighbourhood } V_x \ni x \vphantom{f^{f^f}} \\
      &\text{and }\int_{V_x} \abs f^2 e^{-\vphi} d\lambda_{V_x} < +\infty
    \end{aligned}
  } \; ,
\end{equation*}
where $d\lambda_{V_x}$ is the Lebesgue measure on $V_x$.
Throughout this paper, the following are assumed on $X$:

\begin{enumerate}[itemsep=8pt]

\item \label{item:diff-of-q-psh}
  $(L, e^{-\vphi_L})$ is a hermitian line bundle with an
  analytically singular metrics $e^{-\vphi_L}$, where
  $\vphi_L$ is locally equal
  to $\vphi_1 - \vphi_2$, where each of the
  $\vphi_i$'s is a quasi-psh local function with \emph{neat analytic
  singularities}, i.e.~locally
  \begin{equation*}
    \vphi_i \equiv c_i \log\paren{\sum_j \abs{g_{ij}}^2} \mod
    \smooth \; ,
  \end{equation*}
  where $c_i \in \fieldR_{\geq 0}$ and $g_{ij} \in \holo_X \;$;

\item 
  $\psi$ is a global function on $X$ such that it can also be
  expressed locally as a difference of two quasi-psh functions with
  neat analytic singularities; 

\item \label{item:psi-bounded}
  $\sup_X \psi \leq 0$ (which implies that $\psi$ is quasi-psh after
  some blow-ups as it has only neat analytic singularities);

\item 
  there exist numbers $ m_0,  m_1 \in \fieldR_{\geq 0}$ with $m_0 <  m_1$ such
  that 
  \begin{equation*}
    \mtidlof{\vphi_L + m_0 {\psi}}
    = \mtidlof{\vphi_L + m {\psi}}
    \supsetneq \mtidlof{\vphi_L + m_1 \psi}
    \quad \text{ for all }m \in [ m_0 ,  m_1) \; , 
  \end{equation*}
  i.e.~$m_1$ is a jumping number of the family $\set{\mtidlof{\vphi_L
      + m {\psi}}}_{m \in \fieldR_{\geq 0}} \;$ (such numbers exist on compact
  $X$ as $\psi$ is quasi-psh after suitable blow-ups and thus it
  follows from the openness property of multiplier ideal sheaves and
  \eqref{eq:mult-ideal-sheaves-log-resoln}); 
  

\item $S := S^{(m_1)}$ is a \emph{reduced} subvariety defined by the
  annihilator
  \begin{equation*}
    \Ann_{\holo_X} \paren{ \dfrac{\multidl\paren{\vphi_L +  m_0 \psi}}
      {\multidl\paren{\vphi_L +  m_1 \psi}} } \; ;
  \end{equation*}
  in particular, $S \subset \paren{\psi}^{-1}\paren{-\infty}$.

\end{enumerate}

\begin{remark}
  If $\vphi_L + (m_1+\beta) \psi$ is a quasi-psh potential for
  all $\beta \in [0, \delta]$ for some $\delta > 0$ (which holds true in
  all the main theorems in this paper),
  openness property guarantees that, on every compact subset
  $K\subset X$, there exists $\eps >0$ such that
  $\res{\mtidlof{\vphi_L + (m_1+\eps) \psi}}_{K} =
  \res{\mtidlof{\vphi_L + m_1 \psi}}_{K}$
  (see \cite{Hiep_openness}*{Main Thm.~(ii)} or
  \cite{Lempert_openness}*{Thm.~1.1}; see also
  \cite{Guan&Zhou_effective_openness}).
  This then implies that the analytic subspace defined by
  $\Ann_{\holo_X} \paren{\frac{\multidl\paren{\vphi_L +m_0 \psi}}
    {\multidl\paren{\vphi_L +m_1 \psi}}}$ is automatically reduced,
  following the arguments in \cite{Demailly_extension}*{Lemma 4.2}.
\end{remark}

\begin{remark}
  Most of the arguments in Sections \ref{sec:lc-measures} and
  \ref{sec:extension-from-mlc} of this paper can be adapted to the
  case when $X$ is just a weakly pseudoconvex (non-compact)
  \textde{Kähler} manifold (after passing to a relatively compact
  exhaustion) provided that the upper-boundedness on $\psi$ in
  (\ref{item:psi-bounded}) still holds true on $X$.
  As this is not automatic on a non-compact manifold, and the most
  interesting applications which the authors concern about are on
  compact manifolds, the background manifold is assumed to be compact
  in this paper for the sake of clarity.
\end{remark}

\begin{definition} \label{def:polar-ideal-sheaves}
  Suppose that $\vphi$ is a potential or a global function on $X$ such
  that it is locally a difference $\vphi_1 - \vphi_2$ of quasi-psh
  local functions with neat analytic singularities as in
  (\ref{item:diff-of-q-psh}) above.
  The \emph{polar ideal sheaf $\sheaf P_\vphi$ of $\vphi$} is defined
  to be the ideal sheaf generated by the local holomorphic functions
  $g_{ij}$ for all $j$'s and $i = 1,2$.
\end{definition}

\begin{notation}
  Given a set $V \subset X$, a section $f$ of
  $\frac{\mtidlof{\vphi_L +m_0\psi}} {\mtidlof{\vphi_L +m_1\psi}}$ on
  $V$ (which is supported in $S\cap V$), and a section $F$ of
  $\mtidlof{\vphi_L +m_0\psi}$ on $V$, the notation
  \begin{equation*}
    F \equiv f \mod \mtidlof{\vphi_L +m_1\psi} \quad\text{on }V
  \end{equation*}
  is set to mean that, for all $x \in V$, if $(F)_x$ and $(f)_x$
  denote the germs of $F$ and $f$ at $x$ respectively, one has
  \begin{equation*}
    \paren{(F)_x \bmod \mtidlof{\vphi_L +m_1\psi}_x} = (f)_x \; .
  \end{equation*}
  If such a relation between $F$ and $f$ holds, $F$ is said to be an
  \emph{extension} of $f$ on $V$.
  If the set $V$ is not specified, it is assumed to be the whole space
  $X$.
  Such notation is also applied to cases with a slight variation of
  the sheaf $\mtidlof{\vphi_L +m_1\psi}$ (for example, with
  $\mtidlof{\vphi_L +m_1\psi}$ replaced by $\smooth_X \otimes
  \mtidlof{\vphi_L +m_1\psi}$%
  ).
\end{notation}


\subsection{Lc-measure and extension theorem}
\label{sec:lc-measure_main-thm}


As explained above, the first goal of this paper is to replace the
generalised Ohsawa measure $\gOhvol<\vphi_L>[\psi]{f}$ in the previous
versions of $L^2$ extension theorem (as in \cite{Demailly_extension})
by the measure on log-canonical (lc) centres given as follows.
\begin{definition}\label{def:union-of-sigma-lc-centres}
  If $S$ given in Section \ref{sec:setup} is a reduced divisor with
  snc on $X$, define $\lc_X^\sigma (S)$ to be the \emph{union of all
    lc centres of $(X,S)$ of codimension $\sigma$ in $X$} (see
  \cite{Kollar_Sing-of-MMP}*{Def.~4.15} for the definition of lc
  centres when $S$ is a divisor).
  For a general reduced subvariety $S$ in $X$ given in Section
  \ref{sec:setup}, define $\lc_X^\sigma (S)$ as
  \begin{equation*}
    \lcc := \pi\paren{\lcc<\rs X>(\rs S)} \; ,
  \end{equation*}
  where $\pi \colon \rs X \to X$ is a log-resolution of
  $(X,\vphi_L,\psi)$ and $\rs S$ is the reduced divisor with snc
  described in Section \ref{sec:blow-up} (which satisfies
  $\pi(\rs S) = S$).
  Moreover, an \emph{lc centre of $(X,S)$ \emph{(or, more precisely, lc
      centre of $\paren{X, \frac{\mtidlof{\vphi_L+m_0\psi}}
        {\mtidlof{\vphi_L+m_1\psi}}}$ or $(X,\vphi_L, \psi, m_1)$)} of
    codimension $\sigma$} is meant to be the image under $\pi$ of an
  lc centre of $(\rs X, \rs S)$ of codimension $\sigma$ in $\rs X$.
\end{definition}

\begin{remark}
  Admittedly, it is confusing to talk about the ``codimension'' of an
  lc centre of ``$(X,S)$'' when $S$ is not a divisor.
  For example, with a suitable choice of $\vphi_L$ and $\psi$ such
  that $S = \set p \subset X$ (a point), the lc centre of $(X,\set p)$ has
  codimension $1$ (see Example \ref{eg:extension-from-a-pt}).
  The choice of language here is just to favour the case when $S$ is
  an snc divisor.
\end{remark}

\begin{definition} \label{def:lc-measure}
  The \emph{lc-measure supported on the lc centres of $(X,S)$ of codimension
  $\sigma$ in $X$} (or \emph{$\sigma$-lc-measure} for short) \emph{with respect to
  $f \in \cohgp0[S]{K_X \otimes L \otimes \frac{\mtidlof{\vphi_L
          +m_0 \psi}} {\mtidlof{\vphi_L+m_1 \psi}}}$}, denoted as
  $\abs f_\omega^2 \lcV|\sigma,(m_1)|$, is defined by 
  \begin{equation*} 
    \smooth_0\paren{S} \ni g \mapsto \int_{\lc_X^\sigma (S)}
    g \abs f_\omega^2 \lcV|\sigma,(m_1)|
    := \lim_{\eps \tendsto 0^+} \eps \int_X \extu[g] \abs{\extu[f]}_\omega^2 
    \frac{e^{-\vphi_L -m_1\psi}}{\abs{\psi}^{{\sigma} + \eps}}
    d\vol_{X,\omega} \; ,\footnotemark
  \end{equation*}%
  \footnotetext{~``$\operatorname{lcv}$'' is used in the lc-measure to
    suggest ``\underline{l}c-\underline{c}entre-\underline{v}olume''.
    It also looks like the mirror image of ``$\operatorname{vol}$''.}%
  where
  \begin{itemize}
  \item
    $\extu[f]$ is
    a smooth extension of $f$ to a section on $X$ such that
    $\extu[f] \in \smooth \otimes \mtidlof{\vphi_L+m_0\psi}$;
  
  \item $\extu[g]$ is any smooth extension of $g$ to a function on
    $X$;
  
  \item $\sigma$ is a non-negative integer and the measure
    $\abs f_\omega^2 \lcV|\sigma,(m_1)|$ is supported (if finite) on a
    (reduced) subvariety $\lc_X^\sigma (S)$ of $S$ (see Definition
    \ref{def:union-of-sigma-lc-centres}).
  \end{itemize}
\end{definition}

Note that, as the given section $f$ takes values in $K_X \otimes L$,
the lc-measure defined above does not depend on $\omega$.

The $\sigma$-lc-measure vanishes when $\sigma$ is large and diverges when
$\sigma$ is small (see Section \ref{sec:compute-lc-measures}), and can be finite
and non-zero only at one particular value of $\sigma$ depending on the
given section $f$.
Here an ad hoc definition of such special value of $\sigma$
is given.
Another definition can be found in Definition \ref{def:mlc-of-f}.
\begin{definition} \label{def:ad-hoc-mlc-of-f}
  Given the setting above, the \emph{codimension of minimal lc centres
    (mlc) of $(X,S)$ (or of $(X,\vphi_L,\psi,m_1)$) with respect to $f$},
  denoted by $\sigma_f = \sigma_{f,\vphi_L+m_1\psi}$, is the smallest
  integer $\sigma$ such that
  \begin{equation*}
    \int_{\lc_X^{\sigma}(S)} \abs f_\omega^2 \lcV|\sigma,(m_1)| <
    \infty \; .
  \end{equation*} 
\end{definition}
From the calculation in Section \ref{sec:compute-lc-measures},
$\sigma_f$ is ranging between $0$ and the codimension of mlc
of $(X,S)$ (when $S$ is an snc divisor).
If $\mtidlof{\vphi_L+m_0\psi} = \holo_X$ and if $\sigma_f \geq 1$,
then $f$ vanishes on all lc centres of $(X,S)$ with codimension
$< \sigma_f$ in $X$ but is non-trivial on at least one lc centre of
codimension $\sigma_f$.
Moreover, from the discussion in Section \ref{sec:blow-up}, if $\pi
\colon \widetilde X \to X$ is a log-resolution of $(X,\vphi_L,\psi)$,
the codimension $\sigma_f$ coincides with $\sigma_{\pi^*f \otimes
  s_E}$ (see Section \ref{sec:blow-up} for the meaning of $s_E$ and
log-resolution of $(X,\vphi_L,\psi)$).

The authors would like to mention that the use of such lc-measure was
inspired by the study of residue currents in \cite{Bjork&Samuelsson},
\cite{Samuelsson_residue} and \cite{ASWY-gKingsFormula}.
In their works, the kind of current-valued function (in $1$-variable
case)
\begin{equation*}
  \fieldR_{> 0} \ni \eps \quad \mapsto \quad \eps\: \frac{\ibar dz \wedge d\conj
    z}{\abs z^{2 (1-\eps)}}
\end{equation*}
is studied.
Such function gives a \emph{holomorphic} family (so $\eps \in
\fieldC$) of currents for $\Re\eps >0$ and can be analytically
continued across $\eps = 0$.
Its value at $\eps =0$ is a residue measure on $\set{z =0}$.
The lc-measure considered in this paper is essentially given by the
value at $\eps =0$ of the current-valued function
\begin{equation*}
  \fieldR_{>0} \ni \eps \quad \mapsto \quad \eps
  \:\frac{\bigwedge_{j=1}^{\sigma} \paren{\ibar dz_j \wedge
    d\conj{z_j}}}{\prod_{j=1}^{\sigma} \abs{z_j}^2 \cdot
    \abs{\sum_{j=1}^{\sigma} \log\abs{z_j}^2}^{\sigma+\eps}}
\end{equation*}
after analytically continued across $\eps =0$.

It happens that the lc-measure above can be fitted into the
Ohsawa--Takegoshi-type $L^2$ extension theorem, at least in the
codimension-$1$ case.
The first main result of this paper can be stated as follows.
\begin{thm}[\thmparen{Theorem \ref{thm:extension-sigma=1}, see also Theorem
  \ref{thm:extension-with-seq-of-potentials} for a more general
  statement}]
  \label{thm:ext-from-lc-with-estimate-codim-1-case}
  Suppose that
  \begin{enumerate}
  \item \label{item:curv-cond-ordinary-intro}
    there exists $\delta > 0$ such that
    \begin{equation*}
      \ibddbar
      \vphi_L + \paren{m_1 +\beta}\ibddbar {\psi} \geq 0
      \quad \text{on $X$ for all }\beta \in [0, \delta]
      \; , \text{ and }
    \end{equation*}

  \item \label{item:normalisation-cond-intro}
    for any given constant $\ell >0$,
    the function $\psi$ is normalised
    (by adding to it a suitable constant) such that
    \begin{equation*}
      \psi < -\frac {e}\ell \quad\text{ and }\quad
      \frac {1}{\abs\psi} + \frac{2}{\abs\psi \log\abs{\frac{\ell\psi}{e}}} 
      \leq \delta \; . 
    \end{equation*} 
  \end{enumerate} 
  Then, for any holomorphic section $f \in \cohgp0[S]{\: K_X\otimes
    L \otimes \frac{\mtidlof{\vphi_L
        +m_0\psi}}{\mtidlof{\vphi_L+m_1\psi}}}$,
  if one has
  \begin{equation*}
    \int_S \abs f_\omega^2 \lcV|1,(m_1)| < \infty 
  \end{equation*}
  (which holds true when either the mlc of $(X,S)$ or the mlc of
  $(X,S)$ with respect to $f$ has codimension $1$, see Definitions
  \ref{def:ad-hoc-mlc-of-f} and \ref{def:mlc-of-f}),
  then there exists a holomorphic section $F \in \cohgp0[X]{K_X\otimes
    L \otimes \mtidlof{\vphi_L +m_0 \psi}}$ such that
  \begin{equation*}
    F \equiv f \mod \mtidlof{\vphi_L+m_1\psi}
  \end{equation*}
  with the estimate
  \begin{equation*}
    \int_X \frac{\abs F^2
      e^{-\vphi_L-m_1\psi}}{\abs\psi \paren{\paren{\log\abs{\ell\psi}}^2 +1}}
    \leq \int_S \abs f_\omega^2 \lcV|1,(m_1)| \; .
  \end{equation*} 
\end{thm}
See Remark \ref{rem:normalisation-control} for the purpose of the
number $\ell$ in the estimate.
Notice that the weight in the estimate of $F$ above is pointwisely
dominating (up to a multiple constant) the weight in the estimate in
\cite{Demailly_extension} (which is in the magnitude of
$\frac{e^{\vphi_L +m_1\psi}}{\abs\psi^2}$).
Therefore, the above estimate includes the estimate in
\cite{Demailly_extension} up to a constant multiple.

For the proof, it is first argued in Section \ref{sec:blow-up} that it
suffices to consider the case where the polar ideal sheaves of
$\vphi_L$ and $\psi$ (see Definition \ref{def:polar-ideal-sheaves})
are the defining ideal sheaves of some snc divisors (and thus $S$ is
an snc divisor in particular).
The proof then goes along the lines of arguments in
\cite{Demailly_extension}.

For the sake of simplicity, the proofs in Section
\ref{sec:extension-from-mlc} are given for the case where $m_0=0$ and
$m_1=1$.
The result for the general $m_0$ and $m_1$ can be obtained by
replacing $\vphi_L$ by $\vphi_L +m_0\psi$ and $\vphi_L +\psi$ by
$\vphi_L +m_1\psi$ in the arguments.

As in the classical cases, the problem is reduced to solve for a
weak solution of a $\dbar$-equation with ``error'' (derived from the smooth
extension of $f$, and depending on the $\eps$ in the definition of the
lc-measure) using the twisted Bochner--Kodaira
inequality \eqref{eq:twisted-BK-ineq} with suitably chosen auxiliary
functions (see Theorems \ref{thm:dbar-eq-with-estimate_sigma=1}), at
least on the complement of the polar sets of $\vphi_L$ and $\psi$.
(When $\sigma_f > 1$ (see Definition \ref{def:ad-hoc-mlc-of-f} or
\ref{def:mlc-of-f}), the curvature term in the twisted
Bochner--Kodaira formula (see Lemma \ref{lem:BK-formula-all-sigma})
has a negative summand in the curvature term, which is the obstacle in
obtaining the result on extending sections from lc centres of higher
codimensions.)

The weak solution with ``error'' can be continued across the polar
sets of $\vphi_L$ and $\psi$ via the $L^2$ Riemann continuation
theorem\footnote{
  This is usually named as the ``Riemann extension theorem''.
  The current naming ``Riemann continuation theorem'' is used just to
  distinguish this theorem from the Ohsawa--Takegoshi-type extension
  theorem which is studied in this paper.
  The use of ``continuation'' is found in Grauert--Remmert's book
  \cite{Grauert&Remmert}*{Section A.3.8} (English translation by
  Huckleberry), but ``extension'' is used in
  \cite{Grauert&Remmert-CAS}*{Section 7.1}, a later publication of the
  same authors.
  \label{fn:Riemann-continuation}
}
\cite{Demailly_complete-Kahler}*{\textfr{Lemme} 6.9}.
The required holomorphic extension of $f$ is then constructed from the
above solution with the ``error'' and letting $\eps \tendsto 0^+$.
The required estimate is also obtained after taking the necessary
limits.
The regularity of the limit is assured by following a similar argument
as in \cite{Demailly_extension}.

The above theorem is applicable when $\vphi_L +m_1\psi$ and $\psi$ has
neat analytic singularities.
For the potential with arbitrary singularities, an approximation of
the potential is needed and is handled in Theorem
\ref{thm:extension-with-seq-of-potentials}.



\subsection{Further questions}
\label{sec:further-Q}

There are several questions that the authors would still like to understand.

\begin{enumerate}
\item As stated in Remark \ref{rem:optimal-constant}, it is not yet
  clear to the authors whether the current result (if allowing $X$ to
  be non-compact) includes the results on
  optimal constant in \cite{Blocki_Suita-conj} and
  \cite{Guan&Zhou_optimal-L2-estimate}.
  Moreover, is it possible to determine the ``optimal constant'' in this
  Ohsawa--Takegoshi-type extension theorem with lc-measure?
  
\item The lc-measure is inspired by the residue
  currents studied in \cite{Bjork&Samuelsson},
  \cite{Samuelsson_residue} and \cite{ASWY-gKingsFormula}, obtained
  by replacing their residue currents (i.e.~limit of $\eps e^{-\vphi_L
    -(1-\eps)m_1\psi} d\vol_{X,\omega}$ in the notation of this paper) by
  a current with ``\textfr{Poincaré}-growth'' singularities
  (i.e.~limit of $\eps
  \frac{e^{-\vphi_L-m_1\psi}}{\abs\psi^{\sigma+\eps}} d\vol_{X,\omega}$).
  Is it possible to replace the lc-measure in the main theorem by some
  measure which is defined by currents which diverge to infinity even
  faster, like the limits of 
  \begin{equation*}
    \eps\frac{e^{-\vphi_L-m_1\psi} d\vol_{X,\omega}}
    {\abs\psi^{\sigma}\paren{\log\abs\psi}^{1+\eps}} \; ,\;
    \eps\frac{e^{-\vphi_L-m_1\psi} d\vol_{X,\omega}}
    {\abs\psi^{\sigma} \log\abs\psi \paren{\log\log\abs\psi}^{1+\eps}}
    \; ,\;
    \eps\frac{e^{-\vphi_L-m_1\psi} d\vol_{X,\omega}}
    {\abs\psi^{\sigma} \log\abs\psi
      \log^{\circ 2}\abs\psi \paren{\log^{\circ 3}\abs\psi}^{1+\eps}} \; ,
    \dots
  \end{equation*}
  and so on (where $\log^{\circ j}$ denotes the composition of $j$
  copies of $\log$ functions)?
  It seems to the authors that this could be related to the question
  stated in Remark \ref{rem:McNeal-Varolin-weights}, which is asking
  for the estimates with some better weights given in
  \cite{McNeal&Varolin_adjunction}.
  
\item The first author started to consider the lc-measure during the study
  of analytic adjoint ideal sheaves with Chen-Yu Chi from National
  Taiwan University.
  The lc-measure on various lc centres can possibly be the means to
  generalise the works of Guenancia (\cite{Guenancia}) and Dano Kim
  (\cite{KimDano-adjIdl}) on this subject.
  Furthermore, these lc-measures provide a way to characterise the lc
  centres which can be defined by the multiplier ideal sheaves of
  quasi-psh functions.
  Considering such linkage, it would be of interest to see more of
  their applications in analytic and algebraic geometry.
\end{enumerate}


\subsection{Improved plt extension of Demailly--Hacon--P\u aun}
\label{sec:intro-plt-extension-result}


Another main result of this paper is the following improvement to the
result of plt extension of Demailly--Hacon--P\u aun in \cite{DHP},
which removes the superfluous assumption on the support of the given
$\fieldQ$-divisor $D$.

\begin{thm}[\thmparen{Theorem \ref{thm:proof-improved-plt}}] \label{thm:improved-plt-intro}
  Let $X$ be a projective manifold and $(X,S+B)$ be a purely
  log-terminal (plt) and log-smooth pair with $B$ being a
  $\fieldQ$-divisor such that $S = \floor{S+B}$.
  Let $\mu \in \Nnum$ be such that $\mu\paren{K_X + S + B}$ is a
  $\Znum$-(Cartier)-divisor.
  Assume that $\mu \geq 2$ and 
  \begin{itemize}
  \item $K_X + S + B$ is pseudo-effective (pseff);

  \item $K_X + S + B \lineq_\fieldQ D$, where $D$
    is an effective $\fieldQ$-divisor with snc support;

  \item $\supp S \subset \supp D$
    (\emph{without} the assumption that $\supp D \subset \supp\paren{S+B}$);

  \item no irreducible components of $S$ lies in the \emph{diminished stable
      base locus} $\sBase[-]{K_X+S+B}$ (see, for example, \cite{DHP}*{\S
      2.1} for the definition).
  \end{itemize}
  Then, every $\orgu \in \cohgp0[S]{\holo_S\paren{\mu\paren{K_X+S+B}}
    \otimes \obstidlSHM}$ extends to a holomorphic section in
  $\cohgp0[X]{\mu\paren{K_X+S+B}}$ (see Section
  \ref{sec:proof-improved-plt} and \cite{DHP} for the definitions of
  the extension obstruction ideal sheaf $\obstidlSHM$ and the
  corresponding extension obstruction divisor $\Xi$).

  In particular, when $K_X+S+B$ is nef, the restriction map
  $\cohgp0[X]{\mu\paren{K_X+S+B}} \to
  \cohgp0[S]{\holo_S\paren{\mu\paren{K_X+S+B}}}$ is surjective.
\end{thm}

The assumption $\supp D \subset \supp\paren{S+B}$ in
\cite{DHP}*{Thm.~1.7} is used there to ``remove'' the logarithmic
singularities in the denominator in the estimate obtained from the
Ohsawa--Takegoshi theorem \cite{DHP}*{Thm.~4.3} so that an estimate in
the unweighted $L^2$ norm, and hence the sup-norm, of the auxiliary
extended holomorphic sections can be obtained.
The superfluous assumption is needed as the logarithmic poles are
estimated in sup-norm.

In Lemma \ref{lem:sup-norm-bdd-from-L2-with-denom}, the logarithmic
poles are estimated in $L^p$ norm via \textde{Hölder}'s inequality,
avoiding the use of the superfluous assumption.
Moreover, the choice of the sequence of auxiliary potentials (which
are denoted as $\vphi_{\tau_m}$ in \cite{DHP}*{\S 5}) is replaced by
the sequence of potentials constructed from Bergman kernels of spaces
of global holomorphic sections (see Sections
\ref{sec:Bergman-potentials} and \ref{sec:construction-vphi_L-psi}),
which is a priori uniformly bounded from above (see
\eqref{eq:wt_properties-bdd-uniformly-in-k}), thus avoiding the
complicated inductive construction of the $\vphi_{\tau_m}$ in
\cite{DHP}*{\S 5} as well as simplifying the proof of uniform
boundedness of such sequence when restricted to the subvariety $S$
(see Theorem \ref{thm:bphi-bdd-below-by-u-on-S}).

Apart from the technicalities, the argument in the proof of the
theorem above essentially follows that in \cite{DHP}*{\S 5}.


\section{The measures on lc centres}
\label{sec:lc-measures}

Let $(X,\omega)$ be a \emph{compact} \textde{Kähler} manifold of
dimension $n$ equipped with a (smooth) \textde{Kähler} metric $\omega$.
All notations follow those given in Section \ref{sec:setup}.

\subsection{Effects of log-resolutions} 
\label{sec:blow-up}


As $\vphi_L$ and $\psi$ are locally differences of quasi-psh functions
with neat analytic singularities, by the result of \cite{Hironaka}
(see also \cite{Kollar_Resoln-of-sing} or \cite{Mustata_resoln-sing}),
there is a \emph{log-resolution $\pi \colon \rs X \to X$ of
  $(X,\vphi_L,\psi)$}, which is the composition of a sequence
of blow-ups at smooth centres such that the inverse image ideal
sheaves $\sheaf P_{\vphi_L} \cdot \holo_{\rs X}$ and $\sheaf P_{\psi}
\cdot \holo_{\rs X}$ of the polar ideal sheaves $\sheaf P_{\vphi_L}$
and $\sheaf P_{\psi}$
of $\vphi_L$ and $\psi$ respectively (see Definition \ref{def:polar-ideal-sheaves})
are principal ideal sheaves given by some divisors, and the sum 
of these divisors together with the exceptional divisors of $\pi$
(i.e.~components of the relative canonical divisor $K_{\widetilde X /
  X} := K_{\widetilde X} / \pi^*K_X$) has only snc.

\emph{Assume that $\vphi_L+m_1 \psi$ is psh.}
Without further assumption, one can decompose $K_{\widetilde X / X}$
into two effective $\Znum$-divisors $E$ and $R$ (with the
corresponding canonical holomorphic sections denoted by $s_E$ and
$s_R$) such that $R$ is the \emph{maximal} divisor satisfying
\begin{equation*}
  \pi^*\vphi_L +m_1\pi^*\psi -\phi_R :=\pi^*\vphi_L +m_1\pi^*\psi
  -\log\abs{s_R}^2 \quad\text{being psh.}
\end{equation*}
Suppose that $\pi$ is restricted to $\pi^{-1}(V)$ where
$(V,z)$ is a coordinate neighbourhood in $X$, and let
$\set{U_\gamma}_\gamma$ be a covering of $\pi^{-1}(V) =
\bigcup_{\gamma} U_\gamma$ by coordinate neighbourhoods $(U_\gamma,
w_\gamma)$ in $\widetilde X$.
Let also $\set{\varrho_\gamma}_\gamma$ be a partition of unity
subordinated to $\set{U_\gamma}_{\gamma}$.
Then, with $s_{E}$, $s_{R}$, $z$ and $w_\gamma$'s suitably chosen, one has,
for any $f \in \cohgp0[V]{K_X\otimes L}$ which is viewed as an 
$(n,0)$-form $f = f_V \:dz^1 \wedge \dots \wedge dz^n$,
\begin{equation*}
  \begin{aligned}
    &\int_{V} \abs f^2 e^{-\vphi_L-m\psi}
    = \int_V \abs{f_V}^2 e^{-\vphi_L -m\psi} \:c_n dz^1 \wedge \dots
    \wedge dz^n \wedge d\conj{z^1} \wedge \dots \wedge d\conj{z^n} \\
    =&\sum_\gamma
    \int_{U_\gamma} \varrho_\gamma \pi^*\!\abs{f_V}^2
    e^{-\pi^*\vphi_L-m\pi^*\psi} \abs{s_{E, \gamma} \otimes s_{R,\gamma}}^2 \:c_n
    dw_\gamma^1 \wedge \dots \wedge dw_\gamma^n \wedge
    d\conj{w_\gamma^1} \wedge \dots \wedge d\conj{w_\gamma^n} \\
    =&\sum_\gamma
    \int_{U_\gamma} \varrho_\gamma \abs{\pi^*\!f_V \otimes s_E}^2
    e^{-\pi^*\vphi_L-m\pi^*\psi+\phi_R} \: c_n
    dw_\gamma^1 \wedge \dots \wedge dw_\gamma^n \wedge
    d\conj{w_\gamma^1} \wedge \dots \wedge d\conj{w_\gamma^n}\; ,
  \end{aligned}
\end{equation*}
where $c_n := \paren{-1}^{\frac{n(n-1)}{2}} \paren{\frac{\cplxi}2}^n$.
This shows that, if $f \in \mtidlof[X]{\vphi_L +(m-\eps) \psi}
\setminus \mtidlof[X]{\vphi_L +m \psi}$ for all sufficiently small
$\eps > 0$, then $\pi^*f \otimes s_E \in
\mtidlof[\rs X]{\pi^*\vphi_L +(m-\eps) \pi^*\psi -\phi_R} \setminus
\mtidlof[\rs X]{\pi^*\vphi_L +m \pi^*\psi -\phi_R}$ also for those
$\eps >0$.
In other words, \emph{if $m_p$ is a jumping number of the family
$\set{\mtidlof[X]{\vphi_L +m \psi}}_{m \in \fieldR_{\geq 0}}$, it is
also a jumping number of the family $\set{\mtidlof[\rs X]{\pi^*\vphi_L
    +m \pi^*\psi -\phi_R}}_{m \in \fieldR_{\geq 0}}$.}\footnote{Thanks
to Chen-Yu Chi for pointing out to the authors that the converse is
not true.
An example is provided by taking $\psi = \log\paren{\abs{x^3}^2
  +\abs{y^2}^2} -1$ and $\vphi_L=0$ on $X = \Delta^2 \subset
\fieldC^2_{(x,y)}$, the unit $2$-disc centred at the origin, and
considering the standard principalisation of the ideal $(x^3,y^2)$
given by blowing-up 3 times.
In this case, $\frac56$ and $\frac76$ are consecutive jumping numbers
of $\set{\mtidlof[X]{m\psi}}_{m\in \fieldR_{\geq 0}}$, but
$\set{\mtidlof[\rs X]{m\pi^*\psi -\phi_R}}_{m\in\fieldR_{\geq 0}}$
also has $1$ as a jumping number.}
In particular, there exists $m_0' \in [m_0, m_1)$ such that the ideal
sheaf $\mtidlof[\rs X]{\pi^*\vphi_L +m \pi^*\psi -\phi_R}$ stays
unchanged for $m \in [m_0', m_1)$ and shrinks when $m = m_1$.
 
Moreover, by the choice of $R$, if $E$ is non-trivial, the weight
$e^{-\pi^*\vphi_L -m_1\pi^*\psi +\phi_R}$ is then integrable around a
general point of $E$.
Since the weight has only snc singularities, it can be easily seen
from Fubini's theorem that, for any $m \in [m_0',m_1]$,
\begin{equation*}
  \pi^*f_V \otimes s_E \in
  \mtidlof[\rs X]{\pi^*\vphi_L +m \pi^*\psi -\phi_R}
  \iff
  \pi^*f_V \in \mtidlof[\rs X]{\pi^*\vphi_L +m \pi^*\psi -\phi_R} \; .
\end{equation*}
Since $\pi$ is bimeromorphic (birational) and thus the image of the
exceptional locus
are of codimension at least $2$ in $X$, it can be
seen from the Riemann continuation theorem\footnote{See
  footnote \ref{fn:Riemann-continuation}.} and the identity theorem
for holomorphic functions that any
holomorphic function $\rs f_{\pi^{-1}(V)}$ on $\pi^{-1}(V)$ for
an open set $V \subset X$ can be expressed as $\pi^* f_{V}$ for some
holomorphic function $f_{V}$ on $V$.
It then follows from the above equality between integrals on $V$ and
$\pi^{-1}(V)$ that, for any $m \in [m_0' ,m_1]$,
\begin{align} 
  \notag
  \mtidlof[X]{\vphi_L+m \psi} \cdot \holo_{\rs X}
  &\subset \mtidlof[\rs X]{\pi^*\vphi_L+m \pi^*\psi -\phi_R}
    \quad\text{ and } \\
  \label{eq:mult-ideal-sheaves-log-resoln}
  \mtidlof[X]{\vphi_L+m \psi} &= \pi_*\mtidlof[\rs X]{\pi^*\vphi_L +m
    \pi^*\psi -\phi_R} \; .
\end{align}
Furthermore, suppose that $\vphi_L +\paren{m_1+\beta}\psi$ is quasi-psh
for all $\beta \in [0,\delta]$ for some $\delta >0$, and if $\rs S$
is the reduced \emph{divisor} defined by 
$\Ann_{\holo_{\rs X}} \paren{\frac{\mtidlof[\rs
    X]{\pi^*\vphi_L -\phi_{R} +m_0'\pi^*\psi}} {\mtidlof[\rs
    X]{\pi^*\vphi_L -\phi_{R} +m_1\pi^*\psi}}}$, it then follows
from 
\eqref{eq:mult-ideal-sheaves-log-resoln} that
\begin{equation*}
  \Ann_{\holo_{X}} \paren{\frac{\mtidlof[X]{\vphi_L+m_0'\psi}}
    {\mtidlof[X]{\vphi_L+m_1\psi}}} 
  =\pi_* \Ann_{\holo_{\rs X}} \paren{\frac{\mtidlof[\rs
      X]{\pi^*\vphi_L -\phi_{R} +m_0'\pi^*\psi}} {\mtidlof[\rs
      X]{\pi^*\vphi_L -\phi_{R} +m_1\pi^*\psi}}} \; .
\end{equation*} 
One therefore has $\pi\paren{\smash[t]{\rs S}} = S$.
  
The above discussion can be concluded as follows.
\begin{SNCassumption} \label{assumption:snc}
  When it helps in the computation, by replacing
  $\pi^*\psi$ by $\psi$, $\pi^*\vphi_L -\phi_R$ by $\vphi_L$, $m_0'$
  by $m_0$, and $\pi^*f_V \otimes s_E$ by $f_V$ (thus replacing
  $\pi^*L \otimes R^{-1}$ by $L$, and therefore $\pi^*K_X \otimes
  \pi^*L \otimes E = K_{\rs X} \otimes \pi^*L \otimes R^{-1}$ by $K_X
  \otimes L$), it can be assumed that 
  \begin{itemize}
  \item $S$ is a \emph{reduced divisor}, and
  \item the polar ideal sheaves $\sheaf P_{\vphi_L}$ and $\sheaf
    P_\psi$ of $\vphi_L$ and $\psi$ respectively are principal
    and the corresponding divisors have only \emph{snc} with each
    other.
  \end{itemize}
  Moreover, the estimates on the holomorphic extension obtained in the
  main theorems in the following sections are valid even before
  blowing up.
\end{SNCassumption}


\subsection{Computation with lc-measures}
\label{sec:compute-lc-measures}


In this section, $\vphi_L$ and $\psi$ are assumed to satisfy the snc
assumption \ref{assumption:snc}.

The well-defined-ness of the measure on lc centres of
$(X,S)$ of codimension $\sigma$ (called the ``lc-measure'' or the
``$\sigma$-lc-measure'' for short) is justified below.
Define $\wphi_L$ by
\begin{equation*}
  \wphi_L + \psi_S := \vphi_L + m_1 \psi  \; ,
\end{equation*}
where $\psi_S := \phi_S - \sm\vphi_S < 0$ (see Notation
\ref{notation:potentials} for the meaning of $\phi_S$ and
$\sm\vphi_S$).


A potential $\vphi$ is said to have \emph{Kawamata log-terminal (klt)}
singularities if $\mtidlof{\vphi} = \holo_X$.

\begin{prop} \label{prop:lc-measure}
  Given the snc assumption \ref{assumption:snc} on $\vphi_L$ and
  $\psi$, suppose further that
  \begin{enumerate}[label=$(\dagger)$]
  \item \label{item:klt-assumption}
    $\wphi_L$ has only klt singularities and
    $\wphi_L^{-1}(-\infty) \cup \wphi_L^{-1}(\infty)$ does not contain
    any lc centres of $(X,S)$.
  \end{enumerate}
  Suppose also that $V$ is an open coordinate neighbourhood on which
  \begin{alignat*}{2}
    \psi|_V &= \sum_{j = 1}^{\sigma_V} \nu_j \log\abs{z_j}^2 \;\;
    &+&\sum_{k=\sigma_V+1}^n c_k \log\abs{z_k}^2 + \alpha
        \quad\text{and} \\
    \res{\wphi_L}_{V} &= &&\sum_{k=\sigma_V+1}^n \ell_k \log\abs{z_k}^2
                            + \beta \; ,
  \end{alignat*}
  where 
  \begin{itemize}
  \item each $z_j$ is a holomorphic coordinate and $(r_j, \theta_j)$
    its corresponding polar coordinates on $V$ for $j=1,\dots, n\;$,
  \item $S\cap V =\set{z_1 \dotsm z_{\sigma_V} =0} \;$,
  \item $\alpha$ and $\beta$ are smooth functions such that
    $\sup_V\alpha < 0$,
  \item $\sup_V \frac{r_j}{2\nu_j} \frac{\diff}{\diff r_j} \alpha > -1$
    (i.e.~$\sup_V r_j$ is sufficiently small) for $j=1,\dots, \sigma_f\;$, 
  \item $\sup_V \log\abs{z_k}^2 < 0$ for $k=\sigma_f+1,\dots, n \;$,
  \item $\nu_j$'s are constants such that $\nu_j > 0$ for $j=1,\dots,
    \sigma_V \;$, and
  \item $c_k$'s are constants such that $c_k \geq 0$ for $k=\sigma_V+1,\dots,n \;$,
  \item $\ell_k$'s are constants such that $\ell_k < 1$ (due to the
    klt assumption, possibly negative) for $k=\sigma_V+1,\dots,n \;$.
  \end{itemize} 
  Then, for any compactly supported smooth function $f$ on $V$ such
  that
  \begin{equation*}
    f = \prod_{k=\sigma_f +1}^{\sigma_V} \abs{z_{k}}^{1+a_{k}} \cdot g
    \quad\;\text{with}\quad \res g_{S^{\sigma_f}} \not\equiv 0 \; ,
  \end{equation*}
  where $S^{\sigma_f} := \set{z_1 =\dotsm =z_{\sigma_f} =0}$,
  $\sigma_f$ is some non-negative integer $\leq \sigma_V$,
  $a_{\sigma_f+1},\dots, a_{\sigma_V}$ are some non-negative integers
  (with the convention $\prod_{k=\sigma_f+1}^{\sigma_V}
  \abs{z_{k}}^{1+a_k} = 1$ when $\sigma_f = \sigma_V$) and $g$ is some
  compactly supported bounded function on $V$ with $\abs g^2$ being
  smooth, one has
  \begin{align*}
    &~\int_{\lc^\sigma_X(S)\cap V} \abs f^2 \lcV|\sigma,(m_1)|
      := \lim_{\eps \tendsto 0^+} \eps \int_V \frac{\abs f^2
      e^{-\vphi_L-m_1\psi} \:d\vol_{X,\omega}}{\abs\psi^{\sigma+\eps}}
    \\
    =& \lim_{\eps \tendsto 0^+} \eps \int_V \frac{
       \prod_{k=\sigma_f +1}^{\sigma_V} \abs{z_{k}}^{2 a_{k}} \cdot \abs g^2
      e^{-\wphi_L+\sm\vphi_S} \:d\vol_{X,\omega}}{\abs{z_1 \dots
       z_{\sigma_f}}^2 \:\abs\psi^{\sigma+\eps}}
    \\
    =&
       \begin{dcases}
         0 & \text{when } \sigma > \sigma_f \text{ or } \sigma_f =0 \; , \\
         \frac{\pi^{\sigma_f}}{\paren{\sigma_f -1}! \prod_{j=1}^{\sigma_f} \nu_j}
         \int_{S^{\sigma_f}} \paren{
           \;\;\prod_{\mathclap{k=\sigma_f+1}}^{\sigma_V}\;
           \abs{z_k}^{2 a_k} \cdot \abs g^2
         }_\omega e^{-\wphi_L+\sm\vphi_S}
         \: d\vol_{S^{\sigma_f}, \omega}
         & \text{when } \sigma = \sigma_f \geq 1 \; , \\
         \infty & \text{otherwise,}
       \end{dcases} 
  \end{align*}
  where 
  $\;\paren\cdot_\omega$ denotes the contraction of a section with the
  metric on $K_{S^{\sigma_f}} \otimes \res{K_X^{-1}}_{S^{\sigma_f}}$
  induced from $\omega$.
\end{prop}

\begin{remark}
  With the snc assumption on $\vphi_L$ and $\psi$,
  it is easy to see that $X$ can be covered by the kind of
  open coordinate neighbourhoods described in the proposition.
\end{remark}

\begin{proof} 
  Writing $\omega$ locally as $\frac{\cplxi}{2} \sum_{1\leq j,k \leq n} h_{j\conj
    k} \:dz_j \wedge d\conj{z_k}$ and choosing the canonical section
  defining $\phi_S$ suitably, it follows that
  \begin{align*}
    \eps \int_V \frac{\abs f^2 e^{-\vphi_L-m_1\psi}
    d\vol_{X,\omega}}{\abs\psi^{\sigma+\eps}}
    &=\eps \int_V
      \frac{
      \overbrace{\abs g^2 e^{-\beta+\sm\vphi_S}
      \det \paren{h_{j\conj k}} }^{F_0\: :=}
      \: \bigwedge_{j=1}^{n} \paren{\frac{\cplxi}{2}\: dz_j \wedge
      d\conj{z_j}}
      }{\abs\psi^{\sigma+\eps} \abs{z_1 \dotsm z_{\sigma_f}}^2 \prod_{k=\sigma_f+1}^n
      \abs{z_k}^{2(\ell_k -a_k)}} 
    \\
    &=\eps \int_V
      \frac{F_0} {\abs\psi^{\sigma+\eps}}
      \bigwedge_{j=1}^{\sigma_f} \paren{\frac{dr_j^2}{2r_j^2} \wedge
      d\theta_j} \wedge \;
      \bigwedge_{\mathclap{k=\sigma_f+1}}^n \; \frac{\frac{\cplxi}{2}\: dz_k \wedge
      d\conj{z_k}}{\abs{z_k}^{2(\ell_k-a_k)}} \; ,
  \end{align*}
  where $\ell_{\sigma_f+1}, \ell_{\sigma_f+2}, \dots
  ,\ell_{\sigma_V}$ and $a_{\sigma_V+1}, a_{\sigma_V+2}, \dots, a_n$
  are all defined to be $0$.
  In view of Fubini's Theorem, integrations with respect to the
  variables $z_{\sigma_f+1}, \dots, z_n$ are done at the last step.
  Since all $\paren{\ell_k-a_k}$'s are $< 1$, the integral with respect to
  all variables is convergent as soon as the integral
  with respect to variables $z_1, \dots, z_{\sigma_f}$ is bounded.
  The differentials corresponding to $z_{\sigma_f+1}, \dots, z_n$ are
  made implicit in what follows.
  Notice that $F_0$ is a smooth function.
  
  Observe that, if $\sigma_f =0$, the integral above is convergent and
  bounded above by $\BigO(\eps)$.
  Therefore, it goes to $0$ when $\eps \tendsto 0^+$.
  
  Assume that $\sigma_f \geq 1$ in what follows.
  Set 
  \begin{equation*} 
    t_j := \nu_j \log r_j^2 = \nu_j \log\abs{z_j}^2 \; . 
  \end{equation*}
  The integrand is integrated with respect to each
  $r_j$ over $[0,1]$ (thus to each $t_j$ over $(-\infty, 0]$) and to
  each $\theta_j$ over $[0,2\pi]$.
  Write also $\diff_{r_j}$ for $\frac{\diff}{\diff r_j}$.
  The integral in question then becomes
  \begin{equation} \tag{$*$} \label{eq:lc-measure-local-form}
    \begin{aligned}
      \frac{\eps}{\underbrace{\textstyle \prod_{j=1}^{\sigma_f} \nu_j}_{=:\: \vect\nu}} \int
      \frac{F_0}{\abs\psi^{\sigma+\eps}} \prod_{j=1}^{\sigma_f} dt_j
      &\cdot \underbrace{\prod_{j=1}^{\sigma_f} \frac{d\theta_j}{2}}_{=:\: \vect{d\theta} } 
      =\frac{\eps}{\vect\nu} \int
      \frac{-F_0}{1+\frac{r_1}{2\nu_1} \diff_{r_1} \alpha} \:
      \frac{d\abs\psi}{\abs\psi^{\sigma+\eps} } \prod_{j=2}^{\sigma_f}
      dt_j \cdot \vect{d\theta} \\
      &=\frac{\eps}{\vect\nu \paren{\sigma-1+\eps}} \int
      \frac{F_0}{1+\frac{r_1}{2\nu_1} \diff_{r_1} \alpha} \:
      d\paren{\frac{1}{\abs\psi^{\sigma-1+\eps}} } \prod_{j=2}^{\sigma_f}
      dt_j \cdot \vect{d\theta}
      \; .
    \end{aligned}
  \end{equation}
  Note that the integral above is treated as an iterated integral
  instead of an integral of differential form, and $1+\frac{r_1}{2\nu_1}
  \diff_{r_1} \alpha >0$ on $\supp F_0$ by assumption.
  
  If $\sigma \geq \sigma_f$, one can apply integration by parts and
  induction to yield
  \begin{align*}
    \eqref{eq:lc-measure-local-form}
    =&~\frac{-\eps}{\vect\nu \paren{\sigma-1+\eps}} \int
       \underbrace{\diff_{r_1} \paren{\frac{F_0}{1+\frac{r_1}{2\nu_1}
       \diff_{r_1}\alpha }}}_{=:\: F_1} dr_1
       \frac{dt_2}{\abs\psi^{\sigma-1+\eps}} \prod_{j=3}^{\sigma_f}
       dt_j \cdot \vect{d\theta} \\
    =&~ \frac{-\eps}{\vect\nu \paren{\sigma-1+\eps} \paren{\sigma -2+\eps}} \int
       \frac{F_1}{1+\frac{r_2}{2\nu_2} \diff_{r_2}\alpha } 
       d\paren{\frac{1}{\abs\psi^{\sigma-2+\eps}}} dr_1 \prod_{j=3}^{\sigma_f}
       dt_j \cdot \vect{d\theta} \\
    \tag{$**$} \label{eq:lc-measure-int-by-parts}
    =& \dots = \frac{(-1)^{\sigma_f} \:\eps}{\vect\nu
       \prod_{j=1}^{\sigma_f}\paren{\sigma-j+\eps}} \int
       \frac{F_{\sigma_f}}{\abs\psi^{\sigma -\sigma_f+\eps}}
       \prod_{j=1}^{\sigma_f}
       dr_j \cdot \vect{d\theta}  \; .
  \end{align*}
  Note that the $F_j$'s are defined inductively by
  \begin{equation*}
    F_j :=
    \diff_{r_j} \paren{\frac{F_{j-1}}{1+\frac{r_j}{2\nu_j}\diff_{r_j}\alpha
      }} \; ,
  \end{equation*}
  and all of them are smooth functions.
  When $\sigma > \sigma_f$, the last expression in
  \eqref{eq:lc-measure-int-by-parts} is bounded above by
  $\BigO\paren\eps$.
  Therefore, the integral tends to $0$ again as $\eps \tendsto 0^+$.
  When $\sigma = \sigma_f$, the last expression in
  \eqref{eq:lc-measure-int-by-parts} is bounded above, but the
  multiple constant in front of the
  integral
  does not converge to $0$ as $\eps \tendsto 0^+$.
  After letting $\eps \tendsto 0^+$, the dominated convergence theorem
  and the fundamental theorem of calculus gives
  \begin{equation*}
    \begin{aligned}
      \eqref{eq:lc-measure-int-by-parts}
      &=\frac{(-1)^{\sigma_f}}{\paren{\sigma_f -1}!\: \vect\nu} \int
      F_{\sigma_f} \prod_{j=1}^{\sigma_f} dr_j \cdot \vect{d\theta} 
      =\frac{(-1)^{\sigma_f-1}}{\paren{\sigma_f -1}!\: \vect\nu} \int
      \res{F_{\sigma_f-1}}_{r_{\sigma_f} = 0} \prod_{j=1}^{\sigma_f-1}
      dr_j \cdot \vect{d\theta} \\
      &= \dots = \frac{1}{\paren{\sigma_f -1}!\: \vect\nu} \int
      \res{F_{0}}_{S^{\sigma_f}} \vect{d\theta}
      =\frac{\pi^{\sigma_f}}{\paren{\sigma_f -1}!\: \vect\nu}
      \int_{S^{\sigma_f}} \paren{\;\; \prod_{\mathclap{k=\sigma_f+1}}^{\sigma_V}
        \abs{z_k}^{2 a_k} \cdot \abs g^2}_{\mathclap{\omega}} e^{-\wphi_L+\sm\vphi_S}
      \:d\vol_{S^{\sigma_f}, \omega} \; ,
    \end{aligned}
  \end{equation*}
  which is the desired result.

  It remains to check for the case $\sigma < \sigma_f$ but $\sigma_f
  \geq 1$.
  By assumption, $g$ is not identically $0$ on $S^{\sigma_f}$.
  As the integrand in question is non-negative, by shrinking $V$ (in
  such a way that $V \cap S^{\sigma_f} \neq \emptyset$ after
  shrinking) if necessary, one can assume without loss of generality
  that $\abs g^2 > 0$ on $V$, and thus $F_0 > 0$ on $V$.
  Consider a further change of variables
  \begin{equation*}
    \abs\psi = \abs\psi \; ,\quad q_j := \frac{t_j}{\psi}
    =\frac{\abs{t_j}}{\abs\psi} \quad \text{for } 
    j=2,\dots,\sigma_f \; ,
  \end{equation*}
  where each $q_j$ varies within $[0,1]$ on $V$.
  The expression in \eqref{eq:lc-measure-local-form} then becomes
  \begin{equation*}
    \begin{aligned}
      \eqref{eq:lc-measure-local-form}
      &= \frac{(-1)^{\sigma_f} ~\eps}{\vect\nu} \int
      \frac{F_0}{1+\frac{r_1}{2\nu_1} \diff_{r_1} \alpha} \:
      \frac{\abs\psi^{\sigma_f-1} d\abs\psi}{\abs\psi^{\sigma+\eps} }
      \prod_{j=2}^{\sigma_f} dq_j \cdot \vect{d\theta} \\ 
      &= \frac{(-1)^{\sigma_f} ~\eps}{\vect\nu\: \paren{\sigma_f
          -\sigma-\eps}} \int
      \frac{F_0}{1+\frac{r_1}{2\nu_1} \diff_{r_1} \alpha} \:
      d\paren{\abs\psi^{\sigma_f -\sigma -\eps}}
      \prod_{j=2}^{\sigma_f} dq_j \cdot \vect{d\theta} \; .
    \end{aligned}
  \end{equation*}
  Notice that the factor $(-1)^{\sigma_f}$ is there only to account
  for the difference in orientation between the coordinate systems
  $(r_1,\dots,r_{\sigma_f})$ and $(\abs\psi, q_2, \dots,
  q_{\sigma_f})$.
  The whole expression is itself non-negative.
  As $\frac{F_0}{1+\frac{r_1}{2\nu_1} \diff_{r_1} \alpha} > 0$ on $V$
  and $d\paren{\abs\psi^{\sigma_f -\sigma -\eps}}$ is non-integrable
  on $V$ when $\eps$ is sufficiently small, the expression above tends
  to $\infty$ as $\eps \tendsto 0^+$.
\end{proof}

\begin{remark} \label{rem:lc-measure_gen-f}
  Having the Taylor expansion in mind, for a general compactly supported
  smooth function $f$ on $X$, on every local coordinate neighbourhood
  $V$ where $\res f_{S\cap V} \not\equiv 0$ and $S \cap V =
  \set{z_1\dots z_{\sigma_V} =0}$, there is an
  integer $\sigma_f$ (dependent on $V$) such that
  \begin{equation*}
    \res f_V = \quad \sum_{\mathclap{\substack{p \in \symmgp_{\sigma_V}
        / \paren{\symmgp_{r} \times \symmgp_{\sigma_f}} \\ r\:
        :=\:\sigma_V-\sigma_f}}} \qquad\;\;\;
    \prod_{k=\sigma_f+1}^{\sigma_V} \abs{z_{p(k)}}^{1+a_{p(k)}} \cdot
    g_p
    \quad\;\text{ with}\quad
    \res{g_{p'}}_{{}^{p'}\!S^{\sigma_f}} \not\equiv 0 \;\;
    \text{ for some }p' \; ,
  \end{equation*}
  where every $p$ is a choice of $\sigma_V -\sigma_f$ elements from
  the set $\set{1,2,\dots, \sigma_V}$ (which is abused to mean a
  corresponding permutation of $\set{1,2,\dots, \sigma_V}$),
  each $g_{p}$ is a bounded function on $V$ with $\abs{g_p}^2$ being
  smooth, and ${}^{p'}\!S^{\sigma_f} := \set{z_{p'(1)} =\dotsm
    =z_{p'(\sigma_f)} =0}$.
  The same kind of calculation in the proof of the proposition shows
  that summands of the sum over $p \in \symmgp_{\sigma_V}
  / \paren{\symmgp_{\sigma_V-\sigma_f} \times 
    \symmgp_{\sigma_f}}$ are mutually orthogonal (by considering only
  the monomials in $\abs{z_j}$'s with constant coefficients) with
  respect to the inner product induced from $\lcV|\sigma,(m_1)|$ when
  $\sigma > \sigma_f -2$.
  Therefore, using a partition of unity, the results in the proposition still
  hold for $f|_V$, except that the integral in the case $\sigma =
  \sigma_f$ is now the sum of integrals over all lc centres in $V$ of
  codimension $\sigma_f$, i.e.~
  \begin{equation*}
    \int_{\lc^\sigma_X(S)\cap V} \abs f^2 \lcV|\sigma,(m_1)|
    =
    \begin{dcases}
      0 & \text{when } \sigma > \sigma_f \text{ or } \sigma_f =0 \; , \\
      \infty & \text{when } \sigma < \sigma_f \text{ and }
      \sigma_f \geq 1 \; ,
    \end{dcases} 
  \end{equation*}
  (that the integral diverges when $\sigma \leq \sigma_f -2$ follows
  from the inequality $\frac 1{\abs\psi^{\sigma_f-1+\eps}} \leq \frac
  1{\abs\psi^{\sigma+\eps}}$ on a neighbourhood of any
  ${}^p\!S^{\sigma_f}$ for $\sigma \leq \sigma_f -2$ and the fact that
  the integral diverges when $\sigma=\sigma_f -1$) 
  and
  \begin{align*}
    &~\int_{\lc^\sigma_X(S)\cap V} \abs f^2 \lcV|\sigma,(m_1)| \\
    =&\sum_p \frac{\pi^{\sigma_f}}{\paren{\sigma_f -1}! \prod_{j=1}^{\sigma_f} \nu_{p(j)}}
    \int_{{}^p\!S^{\sigma_f}} \paren{
      \;\;\prod_{\mathclap{k=\sigma_f+1}}^{\sigma_V}\;
      \abs{z_{p(k)}}^{2 a_{p(k)}} \cdot \abs{g_p}^2
    }_\omega e^{-\wphi_L+\sm\vphi_S}
    \: d\vol_{{}^p\!S^{\sigma_f}, \omega}
  \end{align*}
  when $\sigma = \sigma_f$.
  Note that the \emph{largest} $\sigma_f$ among all different local
  neighbourhoods $V$ covering $X$ is the codimension of mlc of $(X,S)$
  with respect to $f$ (see Definition \ref{def:ad-hoc-mlc-of-f} or
  \ref{def:mlc-of-f}).
  Considering all such $V$'s, the proposition also holds true for $f$ with
  $\sigma_f$ being the codimension of mlc of $(X,S)$ with respect to
  $f$ in all cases of $\sigma$ (after the modification for the
  case $\sigma=\sigma_f$).
\end{remark}

\begin{remark} \label{rem:gen-lc-measure}
  If $f$ vanishes to suitable orders along the polar subspaces of
  $\wphi_L$ and $\psi$, the assumption \ref{item:klt-assumption}
  is not necessary in the proposition.
  Suppose that, on an open set $V \subset X$, $\res\psi_V$ is given as
  in the proposition while
  \begin{equation*}
    \res{\wphi_L}_V := \sum_{j = 1}^n \ell_j \log\abs{z_j}^2 +\beta \; ,
  \end{equation*}
  where $\beta$ is a smooth function on $V$ and $\ell_j \geq 0$ are
  arbitrary constants for $j =1,\dots, n$ such that the $m_1$ in
  $\wphi_L+\psi_S =\vphi_L +m_1 \psi$ is a jumping number and
  $\Ann_{\holo_{X}} \paren{\frac{\mtidlof{\vphi_L +m_0\psi}}
    {\mtidlof{\vphi_L +m_1\psi}}}$ defines $S$ as in Section
  \ref{sec:setup} with $S\cap V = \set{z_1 \dotsm z_{\sigma_V} = 0}$
  (i.e.~no upper bound on $\ell_j$'s and $\ell_j >0$ for any
  $j=1,\dots, \sigma_V$ is allowed).
  If
  $f \in \smooth \otimes \mtidlof{\vphi_L +m_0\psi}$,
  it implies that $f \in \smooth \otimes \mtidlof{\wphi_L}$ as
  \begin{equation*}
    G :=\abs f^2 e^{-\wphi_L}
    =\abs f^2 e^{-\vphi_L -m_1\psi -\sm\vphi_S +\phi_S}
    =\abs{f\otimes s_S}^2 e^{-\vphi_L -m_1\psi -\sm\vphi_S} \; .
  \end{equation*}
  Note that $G^{-1}(\infty)$ does not contain any lc centres of
  $(X,S)$
  (if $G^{-1}(\infty)$ contains an lc centre of $(X,S)$, then it can
  be seen that, as $G e^{-\psi_S} \:d\vol_{X,\omega} = \abs f^2
  e^{-\vphi_L -m_1 \psi} \:d\vol_{X,\omega}$ is not integrable,
  there exists some $m_0 < m' < m_1$ such that $\abs f^2 e^{-\vphi_L
    -m'\psi}\:d\vol_{X,\omega}$ is also not integrable, contradicting
  the fact that $m_1$ is the only jumping number in the interval
  $(m_0,m_1]$).
  That means that $\res G_V$ \emph{cannot} be decomposed locally into a
  quotient $\frac gh$ of continuous functions where $\res g_{\set{z_j
      =0}} \not\equiv 0$ while $\abs{z_j}^{2 a_j}$ divides $h$ (in the
  ring of continuous functions) for some positive number $a_j >0$
  for any $j = 1, \dots, \sigma_V$.
  Let $\sigma_f$ be the maximal codimension in $X$ of all lc centres of
  $(X,S)$ \emph{not} contained in $G^{-1}(0)$.
  Then, by expanding $f$ locally on any suitable open set $V$ into a
  sum as in Remark \ref{rem:lc-measure_gen-f} and cancelling the
  factors of $\abs{z_j}$'s suitably, one obtains
  \begin{equation*}
    \res G_V =\res{\abs f^2 e^{-\vphi_L-m_1\psi +\psi_S}}_V
    =\abs{
      \sum_{\substack{p \in \symmgp_{\sigma_V}
            / \paren{\symmgp_{r} \times \symmgp_{\sigma_f}} \\ r\:
            :=\:\sigma_V-\sigma_f}}
      \prod_{k=\sigma_f+1}^{\sigma_V} \abs{z_{p(k)}}^{1+a_{p(k)}}
      \cdot g_p
    }^2 \cdot \frac 1{\prod_{k=\sigma_V+1}^n \abs{z_k}^{2\gamma_k}} \; ,
  \end{equation*}
  where $\gamma_k$'s are constants such that $\gamma_k < 1$ for
  $k=\sigma_V +1, \dots, n$ (as $G$ is integrable on $X$), and each
  $g_p$ is a bounded function on $V$ with $\abs{g_p}^2$ being smooth.
  Moreover, there exist an open set $V$ and some $p$ such that the
  corresponding $g_p$ on $V$ satisfies
  $\res{g_{p}}_{{}^{p}\!S^{\sigma_f}} \not\equiv 0$.
  Following the discussion in Remark \ref{rem:lc-measure_gen-f}, 
  the trichotomy in the conclusion of the proposition ($\sigma >
  \sigma_f$, $\sigma =\sigma_f$ and $\sigma < \sigma_f$) still holds
  even without the assumption \ref{item:klt-assumption}.
\end{remark}

The following definition of $\sigma_f$ is given after the discussion
in Remark \ref{rem:gen-lc-measure}.

\begin{definition} \label{def:mlc-of-f}
  Given any function or vector-bundle-valued section $f$ on $S$ such
  that $f \in \smooth_X \otimes \frac{\mtidlof{\vphi_L
      +m_0\psi}}{\mtidlof{\vphi_L +m_1\psi}}$ with $\extu[f] \in
  \smooth_X \otimes\mtidlof{\vphi_L +m_0\psi}$ denoting any local
  lifting of $f$, define the \emph{codimension of mlc of
    $(X,\vphi_L,\psi,m_1)$ with respect to $f$}, denoted by
  $\sigma_{f,\vphi_L+m_1\psi}$, to be the maximal codimension of all
  the lc centres of $(X,\vphi_L,\psi,m_1)$ which are not contained in
  the zero locus of $G :=\abs{\extu[f]}^2 e^{-\wphi_L}$.
  When it is understood that $S$ is defined from the data $\vphi_L$,
  $\psi$ and $m_1$, the quantity is also called the \emph{codimension
    of mlc of $(X,S)$ with respect to $f$} and denoted by $\sigma_f$.
  An lc centre of $(X,S)$ with such codimension which does not lie in
  $G^{-1}(0)$ is called an \emph{mlc of $(X,S)$ with respect to $f$}.
\end{definition}

It can be seen from Proposition \ref{prop:lc-measure} and the
subsequent remarks that this definition coincides with Definition
\ref{def:ad-hoc-mlc-of-f} when the snc assumption on $\vphi_L$ and
$\psi$ holds.


\subsection{Illustration}
\label{sec:illustration-1}

The following examples are to show that the lc-measures are potentially
good replacement of the Ohsawa measure.


\begin{example} \label{eg:Ohsawa-example}
  In \cite{Ohsawa-example}*{after Prop.~5.4}, Ohsawa provides the
  following example (with a slight modification, which the authors
  owe Bo Berndtsson for).
  On the unit bi-disc $\Delta \subset \fieldC^2$ centred at the origin
  with coordinates $(z,w)$, let $\dep\vphi := \log\paren{\abs{z-w}^2
    +\frac 1k}$ (or any decreasing sequence which converges to
  $\vphi := \log\abs{z-w}^2$) and $\omega$ be the Euclidean metric.
  Then, there is \emph{no} universal constant $C > 0$ such that, for any
  holomorphic function $f$ on $S := \setd{(z,w) \in \Delta}{zw=0 }$
  with $\int_S \abs f^2 e^{-\vphi}\:d\vol_S < \infty$, there exist
  holomorphic functions $\dep F$ on $\Delta$ such that $\res{\dep F}_S = f$ and
  \begin{equation} \label{eq:example-estimate} \tag{$*$}
    \int_{\Delta} \abs{\dep F}^2 e^{-\dep\vphi} \:d\vol_\Delta
    \leq C \int_{S} \abs f^2 e^{-\dep\vphi} \:d\vol_S
    \leq C \int_{S} \abs f^2 e^{-\vphi} \:d\vol_S \; .
  \end{equation}
  Indeed, if
  \begin{equation*}
    f =
    \begin{cases}
      z &\text{on }\set{w=0} \\
      0 &\text{on }\set{z=0}
    \end{cases} \; ,
  \end{equation*}
  then, the existence of $\dep F$ and the estimate
  \eqref{eq:example-estimate} imply the existence of 
  a holomorphic extension $F$ of $f$ such that the estimate
  \eqref{eq:example-estimate} holds with $F$ replacing $\dep F$.
  This in turn implies that $F = zG$ for some holomorphic function $G$
  with $G|_{z=w} = 0$, which is impossible since it means that $F$
  vanishes to order $2$ but $f$ vanishes only up to order $1$ on
  $\set{w=0}$.

  Set $\psi := \phi_S = \log\abs{zw}^2$.
  The lc-measures can filter out Ohsawa's example (so does the Ohsawa
  measure $\int_S \abs f^2 \:d\vol_{S,\vphi}[\psi] := \lim_{t\to
    -\infty} \int_{t < \psi < t+1} \abs{\extu[f]}^2 e^{-\vphi -\psi}
  \:d\vol_\Delta$; see \cite{Ohsawa-V} or \cite{Demailly_extension}
  for the precise definition) such that the estimates with universal
  constant could still be possible.
  Note that $S$ is defined by the annihilator
  $\Ann_{\holo_{\Delta}}\paren{\frac{\mtidlof{\dep\vphi
        +\frac 12 \psi}} {\mtidlof{\dep\vphi
        +\psi}}}$ (the coefficient $\frac 12$ of $\psi$ is chosen such
  that the annihilator still defines $S$ scheme theoretically when
  $\dep\vphi$'s are replaced by $\vphi$), and the mlc of $(X,S)$ is of
  codimension $2$.
  Taking $f$ as above and letting $\extu[f] = z$ (noting that
  $\mtidlof{\dep\vphi +\frac 12\psi} = \holo_\Delta$), computation in
  Proposition \ref{prop:lc-measure} shows that
  \begin{equation*}
    \int_{\set{z=0=w}} \abs f^2 \lcV|2|<\dep\vphi> = 0
  \end{equation*}
  and
  \begin{equation*}
    \int_{\mathrlap{\set{zw=0} \cap \Delta}} \qquad
    \abs f^2 \lcV|1|<\dep\vphi> 
    =\pi \int_{\mathrlap{\set{w=0 \:, \:\abs z < 1}}} \qquad
    e^{-\log\paren{\abs z^2 +\frac 1k}} \:\pi\ibar dz 
    \wedge d\conj z
    =\pi^2 \log\paren{1+k} 
  \end{equation*}
  (so $\sigma_{f, \dep\vphi+\psi} = 1$), where the latter integral
  diverges to $\infty$ as $k \tendsto \infty$.
  This shows that the function $f$ in Ohsawa's example is ruled out
  by the $L^2$ extension theorem in the first place, provided
  that lc-measures are used.

  Indeed, this can also be seen without taking any approximating
  sequence of $\vphi = \log\abs{z-w}^2$.
  Note that $f$ has no local lifting to $\smooth \otimes
  \mtidlof{\vphi +\frac 12 \psi}$ on $\Delta$ as shown by the similar
  argument using vanishing order above.
  This already shows that the lc-measures $\abs f^2 \lcV<\vphi>$ are
  not well-defined for $\sigma =1$ and $2$.
  As a comparison with the computation above,
  let $\pi\colon \rs \Delta \to \Delta$ be the blow-up of $\Delta$ at
  the origin with exceptional divisor $E$, which is a log-resolution
  of $\paren{\Delta, \vphi, \psi}$.
  Consider a neighbourhood $\rs U \subset \rs\Delta$ of
  $\pi^{-1}\set{w=0}$ with coordinates $(s_E, w_1)$ such that $\pi^*w
  = s_E w_1$, $\pi^*z = s_E$ and $\abs{w_1} < \frac 12$, where $E \cap
  \rs U = \set{s_E =0}$ and $\rs S \cap \rs U =\set{s_E w = 0}$ (where
  $\rs S$ is defined as in Section \ref{sec:blow-up}).
  Take a smooth extension $\extu[f]^\pi$ of $\pi^*f$ on $\rs\Delta$
  in $\mtidlof[\rs\Delta]{\pi^*\vphi-\phi_E +\frac 12 \pi^*\psi}$ such
  that $\res{\extu[f]^\pi}_{\rs U} = s_E$ and vanishes outside of a larger
  neighbourhood $\rs U_1 := \set{\abs{w_1} < 1}$ of $\rs U$.
  It then follows that
  \begin{align*}
    \int_{\lcc<\rs\Delta>(\rs S)} \abs{\pi^*f}^2
    \lcV<\pi^*\vphi-\phi_E><\rs\omega>[\pi^*\psi]
    &:=\lim_{\eps \tendsto 0^+} \eps \int_{\rs U_1}
    \frac{\abs{\extu[f]^\pi}^2}{\abs{\pi^*\psi}^{\sigma+\eps}} 
    e^{-\pi^*\vphi+\phi_E-\pi^*\psi} \:d\vol_{\rs\Delta} \\
    &=\lim_{\eps \tendsto 0^+} \eps \int_{\rs U_1}
    \frac{d\vol_{\rs\Delta}}{\abs{s_E}^2 \abs{w_1}^2 \abs{1-w_1}^2
      \abs{\pi^*\psi}^{\sigma+\eps}} \\
    &\overset{\mathclap{\text{Prop.~\ref{prop:lc-measure}}}}= \qquad
      \begin{dcases}
        \frac{\pi^2}{2} & \text{when }\sigma =2 \; ,\\
        \infty & \text{when }\sigma = 1 \; .
      \end{dcases} 
  \end{align*}
  This again shows that Ohsawa's example will be excluded for the
  consideration of $L^2$ extension if the $L^2$ extension theorem with
  respect to the lc-measures is proved.
  This also provides an example that $\abs{\pi^*f}^2
  \lcV|2|<\pi^*\vphi-\phi_E><\rs\omega>[\pi^*\psi] \neq \lim_{{\color{RoyalBlue}k}\tendsto
    \infty} \abs{\pi^*f}^2
  \lcV|2|<\pi^*\dep[{\color{RoyalBlue}k}]\vphi-\phi_E><\rs\omega>[\pi^*\psi]$, even though
  both sides are finite.
  %
\end{example}



\begin{example}[from a private note by Bo Berndtsson; see also
  \cite{Cao&Paun_OT-ext}*{Appendix A}] \label{eg:Berndtsson-example}
  Berndtsson computes a concrete estimate for
  holomorphic functions on the unit bi-disc $\Delta \subset \fieldC^2$
  extended from holomorphic functions on $S :=\setd{(z,w) \in
    \Delta}{zw=0}$ with minimal $L^2$ norm with singular weight
  $e^{-\vphi}$ on $\Delta$.
  It turns out that the estimate can be expressed in terms of
  lc-measures, thus giving a hint on how the estimate looks like in
  general.

  Assume that $\vphi^{-1}\paren{-\infty}$ does not contain any lc centre of
  $(\Delta, S)$ and set $\psi := \phi_S = \log\abs{zw}^2$.
  Assume also that $\vphi$ is psh and has only neat analytic
  singularities for simplicity.
  Let $H := \cohgp0[\Delta]{\mtidlof{\vphi}} \cap L^2\paren{\Delta;
    e^{-\vphi}}$ be the space of $L^2$ holomorphic functions $F$ on
  $\Delta$ with respect to the norm-square $\norm F_\vphi^2 :=
  \int_\Delta \abs F^2 e^{-\vphi} \:d\vol_\Delta$ and consider the filtration
  \begin{equation*}
    H =: H_3 \supset H_2 \supset H_1 \supset H_0 = \set 0 \; ,
  \end{equation*}
  where $H_\sigma$ is the closed subspace of functions which vanish on
  $\lcc<\Delta>$ for $\sigma=1,2$.
  Note that
  \begin{equation*}
    \lcc[1]<\Delta> = S \quad \text{ and } \quad
    \lcc[2]<\Delta> = \set{z=w=0} = \set{(0,0)} \; .
  \end{equation*}
  Let $A_\sigma$ be the orthogonal complement such that $H_{\sigma+1}
  = H_\sigma \oplus A_\sigma$ for $\sigma = 0,1,2$, and thus $H = A_2
  \oplus A_1 \oplus A_0$.
  
  Suppose $F \in H$ is the \emph{minimal} holomorphic extension with
  respect to the norm $\norm\cdot_\vphi$ of some $L^2$ holomorphic
  function $f$ on $S$ (with respect to the potential
  $\res{\vphi}_S$).
  Then $F$ is orthogonal to $H_1 = A_0$.
  Write $F = F_2 + F_1$ such that $F_\sigma \in A_\sigma$ for $\sigma
  =1, 2$, so $F_1$ vanishes on $\lcc[2]<\Delta>$ but is non-trivial on
  $\lcc[1]<\Delta>$ (if $F_1 \not\equiv 0$), which implies
  $\sigma_{F_1,\vphi+\psi} = 1$, while $F_2$ is non-trivial
  on $\lcc[2]<\Delta>$ (if $F_2 \not\equiv 0$), which implies
  $\sigma_{F_2, \vphi+\psi} =2$.
  Therefore, $f = \parres{F_2 +F_1}_S$ and $f_0 := f(0,0)
  =F_2(0,0)$.

  To compute $\norm{F_2}_\vphi^2$,
  let $\Berg(\cdot,\cdot)$ be the Bergman kernel of $H$ with respect to
  the norm $\norm\cdot_\vphi$ and write $\vect 0 := (0,0)$ and $\vect
  z := (z,w)$ when necessary.
  By considering an orthonormal basis $\set{e_0, e_1 , \dots}$ of $H$ such
  that $\Berg\paren{\vect 0 , \vect 0} = \abs{e_0\paren{\vect 0}}^2$,
  one sees that $F_2\paren{\vect z} = c \Berg\paren{\vect z , \vect
    0}$ for some constant $c$ and thus
  \begin{equation*}
    F_2\paren{\vect z} = \frac{f_0 \Berg\paren{\vect z , \vect
        0}}{\Berg\paren{\vect 0,\vect 0}}
    \quad\imply\quad
    \norm{F_2}_\vphi^2 = \frac{\abs{f_0}^2}{\Berg\paren{\vect 0 ,\vect
        0}} \; .
  \end{equation*}
  Note that $\vphi\paren{\vect 0}$ is finite by assumption.
  By getting the estimate of a holomorphic function in $H$ with a
  prescribed value at $\vect 0$ via the Ohsawa--Takegoshi extension
  theorem (see the argument in the proof of
  \cite{Demailly_regularization}*{Prop.~3.1} or Example
  \ref{eg:extension-from-a-pt}), it follows that $e^{\vphi}\paren{\vect 0}
  \lesssim \pi^2 \Berg\paren{\vect 0 ,\vect 0}$, where the constant involved
  in $\lesssim$ is independent of $\vphi$ and $\Berg$, and therefore
  \begin{equation*}
    \norm{F_2}_\vphi^2 \lesssim \pi^2 \abs{f_0}^2
    e^{-\vphi}\paren{\vect 0}
    \overset{\text{by Prop.~\ref{prop:lc-measure}}}=
    \int_{\lcc[2]<\Delta>} \abs{f}^2 \:\lcV|2|<\vphi> \; .
  \end{equation*} 
  
  Next is to compute $\norm{F_1}_\vphi^2$.
  Since $F_1\paren{\vect 0} =0$, there exist holomorphic functions
  $h_1$ and $h_2$ such that $F_1 = z h_1 +w h_2$.
  Notice that $F_1$ is the \emph{minimal} holomorphic extension of $f
  -\res{F_2}_S$ with respect to the norm $\norm\cdot_\vphi$.
  If $h_1$ (resp.~$h_2$) is replaced by the minimal extension
  $\widetilde h_1$ (resp.~$\widetilde h_2$) of $\res{h_1}_{w=0}$
  (resp.~$\res{h_2}_{z=0}$) with respect to $\norm\cdot_\vphi$, the
  sum $z \widetilde h_1 + w \widetilde h_2$ is still an extension of
  $f -\res{F_2}_S$.
  The classical Ohsawa--Takegoshi extension theorem provides estimates
  for minimal holomorphic extensions on Stein manifold which are
  extended from smooth hypersurfaces (note that $\res{\vphi}_S$ is
  well-defined on each irreducible component of $S$ by assumption).
  Therefore, one has
  \begin{align*}
    \norm{F_1}_\vphi^2
    &\leq \norm{z \widetilde h_1 + w \widetilde
      h_2}_\vphi^2
    \lesssim \pi^2 \int_{\mathrlap{\set{\abs z < 1 \: ,\: w=0}}} \qquad
    \abs{h_1}^2 e^{-\vphi} \:\ibar dz\wedge d\conj z
    +\pi^2 \int_{\mathrlap{\set{\abs w <1 \:, \: z=0}}} \qquad
      \abs{h_2}^2 e^{-\vphi} \:\ibar dw\wedge d\conj w \\
    &=\pi^2 \int_{\set{\abs z < 1 \: ,\: w=0}} 
    \frac{\abs{F_1}^2}{\abs z^2} e^{-\vphi} \:\ibar dz\wedge d\conj z
    +\pi^2 \int_{\set{\abs w <1 \:, \: z=0}}
      \frac{\abs{F_1}^2}{\abs w^2} e^{-\vphi} \:\ibar dw\wedge d\conj
      w \\
    &\overset{\mathclap{\text{by Prop.~\ref{prop:lc-measure}}}}= \qquad
      \int_S \abs{F_1}^2 \:\lcV|1|<\vphi>
      =\int_{\lcc[1]<\Delta>} \abs{f -F_2}^2 \:\lcV|1|<\vphi> \; ,
  \end{align*}
  where the constant involved in $\lesssim$ is ``universal'', i.e.~it
  does not depend on $\vphi$ or any functions appearing in the
  integrands of the integrals on either side of the inequality.

  As a result, one has the estimate
  \begin{equation*}
    \norm{F}_\vphi^2 =\norm{F_2}_\vphi^2 +\norm{F_1}_\vphi^2
    \lesssim \int_{\lcc[2]<\Delta>} \abs f^2 \:\lcV|2|<\vphi> 
    +\int_{\lcc[1]<\Delta>} \abs{f-F_2}^2 \:\lcV|1|<\vphi> \; ,
  \end{equation*}
  where the constant involved in $\lesssim$ is universal.
\end{example}

\begin{remark}
  The estimate, though essentially the best one could expect in
  general, may look unsatisfactory in the sense that one seems to have
  lost control of the estimate due to the integral of $f -
  \res{F_2}_{\lcc[1]<\Delta>}$ on the right-hand-side.
  In practice, one may need to manipulate the estimate on
  $\norm{F_2}_\vphi^2$ in order to obtain some control of
  $\int_{\lcc[1]<\Delta>} \abs{f-F_2}^2 \:\lcV|1|<\vphi>$. 
\end{remark}


\section{Extension with estimates with respect to lc-measures on
  codimension-$1$ lc centres}
\label{sec:extension-from-mlc}

For simplicity, suppose that $m_0 = 0$ and $m_1 = 1$.
The arguments remain the same for the case of general jumping
numbers.

As discussed in Section \ref{sec:blow-up}, one can assume that $S$
is a \emph{reduced divisor} in $X$ and that
$(X,S)$ is a \emph{log-smooth} and \emph{log-canonical (lc)} pair.

\subsection{Setup for the extension theorem}
\label{sec:setup-extension-thm}

The goal of the following is to replace the generalised Ohsawa measure
in the Ohsawa--Takegoshi $L^2$ extension theorem
by the lc-measure given by
\begin{equation} \label{eq:lc-measure}
  \abs f_\omega^2 \lcV
  := \lim_{\eps \tendsto 0^+} \eps \abs{\extu[f]}_\omega^2 
  \frac{e^{-\vphi_L -\psi}}{\abs{\psi}^{{\sigma} + \eps}}
  d\vol_{X,\omega} \; ,
\end{equation}
where $\extu[f]$ is any smooth extension of $f$ on
$X$ such that $\abs{\extu[f]}^2 e^{-\vphi_L}$ is locally integrable,
\emph{for the case $\sigma =1$}.
The behaviour of such measure is discussed in Section \ref{sec:lc-measures}.


Set
\begin{itemize}
\item $P_{\vphi_L} := \vphi_L^{-1}(-\infty)$ and $P_\psi :=
  \psi^{-1}(-\infty)$ (only the \emph{negative} poles), which are
  closed analytic subsets of $X$ such that $P_{\vphi_L} \cup
  P_{\psi}$ has only snc by the assumptions on $\vphi_L$ and
  $\psi$ (Sections \ref{sec:setup} and \ref{sec:blow-up}); 

\item 
  $X^\circ := X \backslash \paren{P_{\vphi_L} \cup P_\psi} \;$,
  which has the structure of a complete \textde{Kähler} manifold;

\item $\vphi := \vphi_L + \psi + \nu$, which is a potential (of the
  curvature of a hermitian metric) on $L$, where $\nu$ is a
  real-valued smooth function on $X^\circ$;

\end{itemize}

\subsection{Bochner--Kodaira formula}
\label{sec:BK-formula}


\newcommand{\termRcolor}{NavyBlue}
\newcommand{\termQcolor}{ForestGreen}
\newcommand{\termNcolor}{VioletRed}
\newcommand{\termNpcolor}{violet}

\newenvironment{termR}{\color{\termRcolor}}{}
\newenvironment{termQ}{\color{\termQcolor}}{}
\newenvironment{termN}{\color{\termNcolor}}{}
\newenvironment{termNp}{\color{\termNpcolor}}{}

\NewDocumentCommand{\colouredInner}{
  m        
  O{\eps}  
  m        
  m        
  O{black} 
}{{\color{#5}#1_{#2}\!\ptinner{#3}{#4}}}

\newcommand{\innerR}[3][\eps]{\colouredInner{R}[#1]{#2}{#3}[\termRcolor]}
\newcommand{\innerQ}[3][\eps]{\colouredInner{Q}[#1]{#2}{#3}[\termQcolor]}
\newcommand{\innerN}[3][\eps]{\colouredInner{N}[#1]{#2}{#3}[\termNcolor]}
\newcommand{\innerNp}[3][\eps]{\colouredInner{N^+}[#1]{#2}{#3}[\termNpcolor]}
\newcommand{\innerD}[3][\eps]{\colouredInner{D}[#1]{#2}{#3} }

The key tool for proving this version of extension theorem is still
the twisted Bochner--Kodaira formula (see
\cite{McNeal&Varolin_adjunction}*{Eq.~(8)}, also
\cite{Demailly}*{Ch.~VIII} or \cite{Demailly_extension}*{\S 3.C}).
The following notations are used:
\begin{itemize}
\item $\Theta^\omega\paren{\zeta,\zeta}_\vphi$ denotes, for any real
  $(1,1)$-form $\Theta$ (usually in the form $\ibddbar
  \widetilde\vphi$) and any $K_X\otimes L$-valued $(0,q)$-form
  $\zeta$, the trace of the contraction between $\ibar^{-1} \Theta$
  and $e^{-\vphi} \zeta \wedge \conj\zeta$ with respect to the
  hermitian metric on $X$ given by $\omega$ (in the convention such
  that $\Theta^\omega \paren{\zeta,\zeta}_\vphi \geq 0$ whenever
  $\Theta \geq 0$);
  
\item $\dfadj$ denotes the formal adjoint of $\dbar$ with respect to
  the inner product $\inner\cdot\cdot_{X^\circ,\omega,\vphi}$
  corresponding to the global $L^2$-norm
  $\norm\cdot_{X^\circ,\omega,\vphi}$ on $X^\circ$;

\item $\idxup{\diff\psi} \ctrt \cdot $ denotes the adjoint of
    $\dbar \psi \wedge \cdot $ with respect to
    $\inner\cdot\cdot_{X^\circ,\omega,\vphi}$ on $X^\circ$.

\end{itemize}

\begin{lemma} \label{lem:BK-formula-all-sigma}
  Let $\sigma \geq 1$ be a positive integer.
  With the auxiliary functions defined for every $\eps \in \fieldR$
  as 
  \begin{equation*}
    \begin{aligned}
      \nu &:=-\log\log\abs{\frac{\ell\psi}{e}} \; ,
      & \eta_\eps &:= \abs\psi^{\sigma\paren{1-\eps}} e^{-\nu} =
      \abs\psi^{\sigma\paren{1-\eps}} \log\abs{\frac{\ell\psi}{e}} 
      \quad\text{ and } \\
      & &\lambda_\eps &:=\eta_\eps \paren{\sigma\paren{1-\eps}
        \log\abs{\frac{\ell\psi}{e}} +1}^2 \; ,
    \end{aligned}
  \end{equation*}
  and letting $L$ be endowed with the metric with potential $\vphi :=
  \vphi_L +\psi +\nu$, the Bochner--Kodaira formula becomes
  \begin{align*}
    &~
      \begin{termQ}
        \int_{X^\circ} \abs{\dbar\zeta}_{\omega,\vphi}^2 \eta_\eps
        +\int_{X^\circ} \abs{\dfadj\zeta}_\vphi^2
        \paren{\eta_\eps +\lambda_\eps} 
      \end{termQ}
      -\eps
      \begin{termN}
        \int_{X^\circ} \paren{\frac{\sigma\paren{1-\eps}}{\abs\psi^2}
          +\frac{2}{\abs\psi^2 \log\abs{\frac{\ell\psi}{e}}} }
        \abs{\idxup{\diff\psi} \ctrt \zeta}_\vphi^2 \eta_\eps
      \end{termN}
    \\
    =&
       \begin{aligned}[t]
         &~
         \begin{termR}
           \int_{X^\circ} \abs{\nabla^{(0,1)}
             \zeta}_{\omega,\vphi}^2 \eta_\eps +\int_{X^\circ}
           \idxup{\ibddbar\paren{\vphi_L +\psi}
             +\paren{\frac{\sigma \paren{1-\eps}}{\abs\psi}
               +\frac{2}{\abs\psi
                 \log\abs{\frac{\ell\psi}{e}}}} \ibddbar\psi
           } \ptinner\zeta\zeta_\vphi \eta_\eps
         \end{termR}
         \\
         &
         \begin{termR}
           +\int_{X^\circ} \abs{\dfadj\zeta
             +\frac{\eta_\eps}{\lambda_\eps}
             \idxup{\diff\log\eta_\eps} \ctrt \zeta}_\vphi^2
           \lambda_\eps
         \end{termR} 
         \\
         &-\paren{\sigma-1}\paren{1-\eps}
         \begin{termN}
           \int_{X^\circ}
           \paren{\frac{\sigma \paren{1-\eps}}{\abs\psi^2}
             +\frac{2}{\abs\psi^2
               \log\abs{\frac{\ell\psi}{e}}} }
           \abs{\idxup{\diff\psi} \ctrt \zeta}_\vphi^2 \eta_\eps
         \end{termN} 
       \end{aligned} 
  \end{align*}
  for any compactly supported $K_X\otimes L$-valued smooth $(0,1)$-forms
  $\zeta \in \smform*/0,1/\paren{X^\circ; K_X\otimes L} $ on
  $X^\circ$.
\end{lemma}

\begin{proof}

  From \cite{Siu}*{\S 1.3} or
  \cite{McNeal&Varolin_adjunction}*{Eq.~(8)}, it follows that
  \begin{align*}
    &~
      \begin{termQ}
        \int_{X^\circ} \abs{\dbar\zeta}_{\omega,\vphi}^2 \eta_\eps
        +\int_{X^\circ} \abs{\dfadj\zeta}_\vphi^2 \eta_\eps
      \end{termQ}
      -
      \begin{termR}
        2\Re \int_{X^\circ}
        \ptinner{\dfadj\zeta\:}{\:\idxup{\diff\log\eta_\eps} \ctrt
          \zeta}_\vphi \eta_\eps
      \end{termR}
      +
      \begin{termN}
        \int_{X^\circ} \abs{\idxup{\diff\log\eta_\eps} \ctrt
          \zeta}_\vphi^2 \eta_\eps
      \end{termN}
    \\
    =&~
       \begin{termR}
         \int_{X^\circ} \abs{\nabla^{(0,1)} \zeta}_{\omega,\vphi}^2
         \eta_\eps
       \end{termR} 
       +\int_{X^\circ} \idxup{\ibddbar\paren{\vphi_L +\psi +\nu}
       -\ibddbar\log\eta_\eps } \ptinner\zeta\zeta_\vphi \eta_\eps
       \; .
  \end{align*}
  A direct computation with the choices of $\nu$ and $\eta_\eps$
  yields 
  \begin{align*}
    &~
      \begin{termQ}
        \int_{X^\circ} \abs{\dbar\zeta}_{\omega,\vphi}^2 \eta_\eps
        +\int_{X^\circ} \abs{\dfadj\zeta}_\vphi^2
        \eta_\eps 
      \end{termQ} 
      \begin{aligned}[t]
        &+
        \begin{termN}
          \int_{X^\circ} \paren{\frac{\sigma^2\paren{1-\eps}^2}{\abs\psi^2}
            +\frac{2\sigma\paren{1-\eps}}{\abs\psi^2
              \log\abs{\frac{\ell\psi}{e}} }}
          \abs{\idxup{\diff\psi} \ctrt \zeta}_\vphi^2 \eta_\eps
        \end{termN}
        \\
        &+\int_{X^\circ} \frac{1}{\abs\psi^2
          \paren{\log\abs{\frac{\ell\psi}{e}}}^2 }
        \abs{\idxup{\diff\psi} \ctrt \zeta}_\vphi^2 \eta_\eps
      \end{aligned}
    \\
    =&
       \begin{aligned}[t]
         &~
         \begin{termR}
           \int_{X^\circ} \abs{\nabla^{(0,1)} \zeta}_{\omega,\vphi}^2
           \eta_\eps
         \end{termR} 
         \begin{aligned}[t]
           &
           \begin{termR}
             +\int_{X^\circ} \idxup{\ibddbar\paren{\vphi_L +\psi}
               +\paren{\frac{\sigma\paren{1-\eps}}{\abs\psi}
                 +\frac{2}{\abs\psi
                   \log\abs{\frac{\ell\psi}{e}}} }
               \ibddbar\psi } \ptinner\zeta\zeta_\vphi \eta_\eps
           \end{termR}
           \\
           &+\int_{X^\circ} \paren{
             \begin{termN}
               \frac{\sigma\paren{1-\eps}}{\abs\psi^2}
               +\frac{2}{\abs\psi^2
                 \log\abs{\frac{\ell\psi}{e}}}
             \end{termN} 
             +\frac{2}{\abs\psi^2 \paren{\log\abs{\frac{\ell\psi}{e}}}^2}
           } 
           \abs{\idxup{\diff\psi} \ctrt \zeta}_\vphi^2 \eta_\eps
         \end{aligned}
         \\
         &
         \begin{termR}
           +2\Re \int_{X^\circ}
           \ptinner{\dfadj\zeta\:}{\:\idxup{\diff\log\eta_\eps} \ctrt
             \zeta}_\vphi \eta_\eps
         \end{termR}
         \; .
       \end{aligned} 
  \end{align*}
  It follows from the choices of $\lambda_\eps$ that
  \begin{align*}
    \int_{X^\circ} \abs{\idxup{\diff\log\eta_\eps} \ctrt
    \zeta }_\vphi^2 \frac{\eta_\eps^2}{\lambda_\eps}
    &=\int_{X^\circ}
      \frac{1}{\abs\psi^2 \paren{\log\abs{\frac{\ell\psi}{e}}}^2
      } \abs{\idxup{\diff\psi} \ctrt \zeta}_\vphi^2 \eta_\eps \; .
  \end{align*}
  As a result, the acclaimed formula is obtained after completing the
  square for the inner-product terms by adding $\int_{X_\circ}
  \abs{\dfadj\zeta}_\vphi^2 \lambda_\eps$ to both sides, and
  collecting terms of $
  \begin{termN}
    \abs{\idxup{\diff\psi} \ctrt \zeta}_\vphi^2
  \end{termN}
  $ (in
  \begin{termN}
    \termNcolor
  \end{termN})
  suitably.
\end{proof}

It follows from Lemma \ref{lem:BK-formula-all-sigma} that, \emph{when
  $\sigma = 1$ and the remaining terms on the right-hand-side (in
  \begin{termR}
    \termRcolor
  \end{termR})
  are semi-positive}, one has
\begin{equation} \label{eq:twisted-BK-ineq}
  \int_{X^\circ} \abs{\dbar\zeta}_{\omega,\vphi}^2 \eta_\eps
  +\int_{X^\circ} \abs{\dfadj\zeta}_\vphi^2 \paren{\eta_\eps +\lambda_\eps}
  \geq
  \eps \int_{X^\circ} \frac{1-\eps}{\abs\psi^2} 
  \abs{\idxup{\diff\psi} \ctrt \zeta}_\vphi^2 \eta_\eps  
\end{equation}
for all compactly supported $\zeta$.
Positivity of the terms in
\begin{termR}
  \termRcolor
\end{termR}
is provided by suitable curvature assumption.

The completeness of $X^\circ$ guarantees that $\omega$ can be modified
to a complete metric, and, in that case, the inequality
\eqref{eq:twisted-BK-ineq} holds true also for all (weighted) $L^2$
$(0,1)$-forms $\zeta$ in both of the domains of $\dbar$ and its
Hilbert space adjoint $\dbar^*$ (see, for example,
\cite{Demailly}*{Ch.~VIII, \S 3}), and thus the Riesz Representation
Theorem can be invoked.


\subsection{Proof of the extension theorem with $1$-lc-measure}
\label{sec:extension-from-lc-sigma=1}

Let $\theta \colon [0,\infty) \to [0,1]$ be a smooth non-increasing
function such that $\theta \equiv 1$ on $[0,\frac{1}{A}]$ and $\equiv
0$ on $[\frac{1}{B}, \infty)$, where $1 < B < A$, and $\abs{\theta'}
\leq \frac{AB}{A-B} + \eps_0$ on $[0,\infty)$ for some positive
constant $\eps_0$.
Define also that $\theta_\eps := \theta \circ \abs{\psi}^{-\eps}$
and $\theta_\eps' := \theta' \circ \abs{\psi}^{-\eps}$
for convenience.

It is shown below (Theorem \ref{thm:extension-sigma=1}) that the Ohsawa 
measure in the Ohsawa--Takegoshi extension theorem can be replaced by
the lc-measure \eqref{eq:lc-measure} in the classical case,
i.e.~when \emph{mlc of $(X,S)$ are of codimension $1$}
(and $S$ is smooth as $(X,S)$ is log-smooth), or when the holomorphic
section $f$ to be extended vanishes on the singular locus of $S$,
or more precisely, when the \emph{mlc of $(X,S)$ with respect to $f$}
(see Definition \ref{def:ad-hoc-mlc-of-f} or \ref{def:mlc-of-f})
\emph{is of codimension $1$}.

\begin{thm} \label{thm:dbar-eq-with-estimate_sigma=1}
  Suppose that
  \begin{enumerate}[series=ext-hypo-12]
  \item \label{item:curv-cond-ordinary}
    there exists $\delta > 0$ such that
    \begin{equation*}
      \ibddbar
      \paren{\vphi_L +\psi} + \beta\ibddbar {\psi} \geq 0
      \quad \text{on $X$ for all }\beta \in [0, \delta]
      \; , \text{ and }
    \end{equation*}
    
  \item \label{item:normalisation-cond}
    for any given constant $\ell >0$, the function $\psi$ is
    normalised (by adding to it a suitable constant) such that
    \begin{equation*}
      \psi < -\frac e \ell \quad\text{ and }\quad
      \frac 1{\abs\psi} + \frac{2}{\abs\psi \log\abs{\frac{\ell\psi} e}}
      \leq \delta \; . 
    \end{equation*}
    (See Remark \ref{rem:normalisation-control} for the use of the
    constant $\ell$.)

  \end{enumerate}
  Let $\extu[f]$ be an $K_X\otimes L$-valued smooth section on $X$
  such that 
  $\abs{\dbar\extu[f]}_\omega^2 e^{-\vphi_L-\psi} \log\abs{\frac{\ell \psi}
    e}$ is integrable over $X$.
  Then, 
  for 
  any numbers $\eps, \eps' >0$,
  the $\dbar$-equation
  \begin{equation*} 
    \dbar u_\eps = v_\eps :=
    \dbar\paren{\theta\paren{\frac{1}{\abs{\psi}^\eps}} \extu[f]}
    = \underbrace{\frac{\eps~ \theta_\eps'~ \dbar\psi \wedge \extu[f]}
      {\abs{\psi}^{1+\eps}}}_{=:\: v_\eps^{(1)}}
    +\underbrace{\theta_\eps \dbar\extu[f] 
      \vphantom{\frac{\eps~ \theta_\eps'~ \dbar\psi \wedge
          \extu[f]}{\abs{\psi}^{1+\eps}}} }_{=:\: v_\eps^{(2)}} 
  \end{equation*}
  can be solved with an $\eps'$-error,
  in the sense that there are a smooth $K_X\otimes L$-valued
  $(0,1)$-form $w_{\eps',\eps}$
  and a smooth section $u_{\eps',\eps}$ on $X^\circ$ such that
  \begin{equation} \label{eq:dbar-with-error-on-X}
    \dbar u_{\eps',\eps}
    + w_{\eps',\eps}
    = v_\eps \quad\text{on } X^\circ \; ,
  \end{equation}
  with the estimates 
  \begin{equation*}
    \begin{aligned} 
      \int_{X^\circ} \frac{\abs{u_{\eps',\eps}}^2 e^{-\vphi_L -\psi}} 
      {\abs{\psi}^{1 -\eps} \paren{\paren{\log\abs{\ell\psi}}^2 +1}}
      &+\frac 1{\eps'} \int_{X^\circ}
      \abs{w_{\eps',\eps}}_\omega^2 e^{-\vphi_L-\psi}
      \log\abs{\frac{\ell\psi} e} \\ 
      &\leq \frac 1{\eps'} \int_{X}  \abs{\theta_\eps \dbar \extu[f]}_\omega^2
        e^{-\vphi_L -\psi} \log\abs{\frac{\ell\psi} e}
        +\frac{\eps}{1-\eps} \int_X \frac{ \abs{\theta_\eps'}^2
          \abs{\extu[f]}^2 e^{-\vphi_L -\psi} }
        {\abs{\psi}^{1+\eps}} \; .
    \end{aligned}
  \end{equation*}
\end{thm}

\begin{remark} \label{rem:vphi_L+psi=q-psh}
It is well known that a locally $L^1$ function $f$, which satisfies
\begin{equation*}
\ibddbar f \geq 0 \quad\text{as a current,}
\end{equation*}
coincides with a uniquely determined psh function almost everywhere
(see, for example, \cite{Hormander}*{Thm.~1.6.10, Thm.~1.6.11}).
Since $\vphi_L$ and $\psi$ are locally differences of
quasi-psh functions, a simple argument shows that
$\paren{\varphi_L+\psi}+\beta\psi$ is a psh potential on $X$ for every $\beta \in
[0,\delta]$ by the assumption (\ref{item:curv-cond-ordinary}).
\end{remark}

\begin{proof} 
  \newcommand{\arbN}[2][]{\mathcal{N}_{#1}\paren{#2}}
  
  Let $L$ be endowed with a metric with potential
  $\vphi:= \vphi_L+\psi+\nu$ and choose the auxiliary functions $\nu$,
  $\eta_\eps$ and $\lambda_\eps$ as in Lemma
  \ref{lem:BK-formula-all-sigma} with $\sigma = 1$.
  The curvature assumption \eqref{item:curv-cond-ordinary} and the
  normalisation assumption \eqref{item:normalisation-cond} assure that
  the terms on the right-hand-side (in
  \begin{termR}
    \termRcolor
  \end{termR})
  in Lemma \ref{lem:BK-formula-all-sigma} is semi-positive, and thus
  the twisted Bochner--Kodaira inequality \eqref{eq:twisted-BK-ineq}
  holds true.
  %
  Write $\inner\cdot\cdot := \inner\cdot\cdot_{X^\circ,\omega,\vphi}$ as
  the global inner product on $X^\circ$ induced by the potential
  $\vphi := \vphi_L +\psi +\nu$ and $\norm\cdot :=
  \norm\cdot_{X^\circ,\omega,\vphi}$ the corresponding norm.\footnote{Note
    that $\vphi_L+\psi$ is, being psh by Remark
    \ref{rem:vphi_L+psi=q-psh}, locally bounded from above,
    so the weight in the norm $\norm\cdot_{X^\circ,\omega,\vphi}$ is
    everywhere positive on $X^\circ$ even though $\vphi_L$ itself
    may go to $+\infty$.}
  Although $\omega$ is not assumed to be complete in the statement,
  the standard argument (see, for example, \cite{Demailly}*{Ch.~VIII,
    \S 6}) reduces the problem to the case where $\omega$ is complete
  on $X^\circ$, which is assumed to be the case in what follows.

  \emph{Assuming that $v_\eps^{(2)} = 0$ on $X$}, the usual
  argument with the Cauchy--Schwarz inequality and the twisted
  Bochner--Kodaira inequality \eqref{eq:twisted-BK-ineq} yields, for
  any compactly supported smooth $K_X\otimes L$-valued $(0,1)$-form
  $\zeta$ on $X^\circ$, that
  \begin{align*}
    \abs{\inner\zeta{v_\eps}}
    &=\abs{\inner{\paren{\zeta}_{\ker\dbar} \:}{~v_\eps^{(1)} }}
      =\abs{\inner{\paren{\diff\psi}^\omega \ctrt \paren{\zeta}_{\ker\dbar} \:}
      {~\frac{\eps \theta_\eps' \extu[f]}{\abs\psi^{1+\eps}} }} \\
    &\leq
      \paren{\eps \int_{X^\circ} \frac{\abs{\:\paren{\diff\psi}^\omega
      \ctrt \paren{\zeta}_{\ker\dbar}}^2_{\vphi}}{\abs\psi^2} \:
      \eta_\eps }^{\frac 12}
      \paren{\int_{\mathrlap{\supp \theta_\eps'}} \quad \frac{\eps \abs{\theta_\eps'}^2
      \abs{\extu[f]}^2 e^{-\vphi_L -\psi -\nu}} 
      {\abs\psi^{2\eps} \eta_\eps}}^{\frac 12} \\
    &\overset{\mathclap{\text{by \eqref{eq:twisted-BK-ineq}}}}\leq \quad\;\;
      \underbrace{\paren{\int_{X^\circ} \abs{\dfadj\zeta}^2_{\vphi} \paren{\eta_\eps +\lambda_\eps}
      }^{\frac 12}}_{=: ~\arbN[1]{\dfadj\zeta}}
      \underbrace{\paren{\frac{\eps}{1-\eps} \int_{\mathrlap{\supp \theta_\eps'}}
      \quad \frac{\abs{\theta_\eps'}^2 \abs{\extu[f]}^2 e^{-\vphi_L
      -\psi}} {\abs{\psi}^{1+\eps}} }^{\frac 12}}_{=: ~\arbN[2]{\extu[f]}} \; ,
  \end{align*}
  where $\paren\cdot_{\ker\dbar}$ denotes the orthogonal projection to
  the closed subspace $\ker\dbar$ with respect to $\inner\cdot\cdot$.
  The completeness of $X^\circ$ and the Riesz representation theorem then assure the
  existence of the solution $u_{\eps}$ to the equation $\dbar
  u_{\eps} = v_\eps$ with the estimate
  \begin{equation*}
    \int_{X^\circ} \frac{\abs{u_{\eps}}^2 e^{-\vphi_L -\psi
        -\nu}}{\eta_\eps +\lambda_\eps}
    \leq
    \frac{\eps}{1-\eps} \int_{\mathrlap{\supp \theta_\eps'}}
    \quad \frac{\abs{\theta_\eps'}^2 \abs{\extu[f]}^2 e^{-\vphi_L
        -\psi}} {\abs{\psi}^{1+\eps}}  \; .
  \end{equation*}
  One then obtains the required estimate by noticing that
  $\paren{\eta_\eps + \lambda_\eps} e^{\nu} \leq
  \abs\psi^{1-\eps} \paren{\paren{\log\abs{\ell\psi}}^2 +1}$.

  \emph{When $v_\eps^{(2)} \neq 0$,} one can handle the situation using the
  argument as in \cite{Demailly_extension}*{after (5.20)} or the
  following slight variation of that.
  For any compactly supported smooth $K_X \otimes L$-valued
  $(0,1)$-form $\zeta$ on $X^\circ$, applying the Cauchy--Schwarz
  inequality directly yields
  \begin{align*}
    \abs{\inner\zeta{v_\eps}}
    &\leq \abs{\inner{\paren{\zeta}_{\ker\dbar}
      \:}{~v_\eps^{(1)} }}
      +\abs{\inner{\paren{\zeta}_{\ker\dbar}\:}{~v_\eps^{(2)} }} \\
    &\leq \arbN[1]{\dfadj\zeta} \arbN[2]{\extu[f]}
      +\norm{\zeta} \norm{\theta_\eps \dbar\extu[f] } \\ 
    &\leq \paren{\paren{\arbN[1]{\dfadj\zeta}}^2 +
      \eps'\norm{\zeta}^2}^{\frac 12}
      \paren{\paren{\arbN[2]{\extu[f]}}^2 + \frac{1}{\eps'}
      \norm{\theta_\eps \dbar\extu[f]}^2}^{\frac 12}
  \end{align*}
  for any $\eps' > 0$.
  Note that the norm-square $\norm{\theta_\eps \dbar\extu[f]}^2
  = \int_{X} \abs{\theta_\eps \dbar\extu[f]}_\omega^2
  e^{-\vphi_L-\psi-\nu}$ converges on $X$ by assumption
  (given the choice of $\nu$ in Lemma
  \ref{lem:BK-formula-all-sigma}).
  The Riesz representation theorem then assures the acclaimed
  existence of solution $(u_{\eps',\eps}, w_{\eps',\eps})$ and
  estimate, with the fact that
  $\paren{\eta_\eps + \lambda_\eps} e^{\nu} \leq
  \abs\psi^{1-\eps} \paren{\paren{\log\abs{\ell\psi}}^2 +1}$.
  
  Note also that the smoothness of $(u_{\eps',\eps}, w_{\eps',\eps})$
  follows from the smoothness of $v_\eps$ and the regularity of the
  $\dbar$ operator.
  This completes the proof.
\end{proof}

Theorem \ref{thm:dbar-eq-with-estimate_sigma=1} holds true
irrespective of the codimension of mlc of $(X,S)$.
The required extension of $f$ with estimate given in terms of the
measure in \eqref{eq:lc-measure} can be obtained by
letting $\eps \tendsto 0^+$ (after estimating $\abs{\theta'_\eps}^2$
by a constant and followed by $\eps' \tendsto 0^+$), provided
that
the right-hand-side of the estimate converges.
However, before starting the limit process, the solutions of the
$\dbar$-equation \eqref{eq:dbar-with-error-on-X} should be continued to the
whole of $X$.



\begin{prop} \label{prop:dbar-soln-continuated-to-X}
  Under the assumptions (\ref{item:curv-cond-ordinary}) and
  (\ref{item:normalisation-cond}) in Theorem
  \ref{thm:dbar-eq-with-estimate_sigma=1}, there exists solution
  $(u_{\eps',\eps}, w_{\eps',\eps})$ to the $\dbar$-equation
  \eqref{eq:dbar-with-error-on-X}, namely $\dbar u_{\eps',\eps}
  +w_{\eps',\eps} =v_\eps$, with the estimate given in Theorem
  \ref{thm:dbar-eq-with-estimate_sigma=1}, which holds true on the
  whole of $X$ (not only on $X^\circ$).
\end{prop}

\begin{proof}
  First, for fixed $\eps$ and $\eps'$, apply Theorem
  \ref{thm:dbar-eq-with-estimate_sigma=1} with 
  $\vphi =\vphi_L +\psi +\nu$ replaced by $\vphi_L +\paren{1+r}\psi
  +\nu$, where $0 <r \ll 1$, and obtain $u_{\eps',\eps,r}$ and
  $w_{\eps',\eps,r}$ satisfying the $\dbar$-equation
  \eqref{eq:dbar-with-error-on-X} with the estimate in the Theorem.
  The number $r$ is chosen sufficiently small (which depends on
  $\eps$) such that the assumptions \eqref{item:curv-cond-ordinary} and
  \eqref{item:normalisation-cond} in Theorem
  \ref{thm:dbar-eq-with-estimate_sigma=1} imply that, with $\sigma
  =1$, the terms on the right-hand-side (in
  \begin{termR}
    \termRcolor
  \end{termR})
  in Lemma \ref{lem:BK-formula-all-sigma} (after $r$ is inserted) are
  semi-positive, so that Theorem
  \ref{thm:dbar-eq-with-estimate_sigma=1} can be invoked.
  
  Notice that $v_\eps$ is smooth on $X$.
  In view of \cite{Demailly_complete-Kahler}*{\textfr{Lemme} 6.9},
  it suffices to show that both $u_{\eps',\eps,r}$ and $w_{\eps',\eps,r}$
  are in $\Lloc[2](X)$ in order to show that the $\dbar$-equation
  \eqref{eq:dbar-with-error-on-X} with solution
  $(u_{\eps',\eps,r},w_{\eps',\eps,r})$ holds true on the
  whole of $X$.
  The claim is then proved after letting $r \to 0^+$.
  
  The curvature assumption (\ref{item:curv-cond-ordinary}) in Theorem
  \ref{thm:dbar-eq-with-estimate_sigma=1} infers that $\vphi_L+(1+r)\psi$
  is psh on $X$ (see Remark \ref{rem:vphi_L+psi=q-psh}), thus locally
  bounded above by some constant.
  Since $\frac{\ell\psi}{e}$ is also bounded above from $0$ by
  assumption \eqref{item:normalisation-cond} in Theorem
  \ref{thm:dbar-eq-with-estimate_sigma=1}, it follows that
  $e^{-\vphi_L-(1+r)\psi} \log\abs{\frac{\ell \psi} e}$ is bounded from below by
  some \emph{positive} constant.
  From the estimate provided by Theorem
  \ref{thm:dbar-eq-with-estimate_sigma=1}, $w_{\eps',\eps,r}$ is
  in $\Lloc[2](X)$.


  From the fact that
  \begin{equation} \label{eq:xlogx-estimate}
    x^{\eps} \abs{\log x}^s \leq \frac{s^s}{e^s \eps^s}
  \end{equation}
  for all $x \in [0,1)$, $\eps > 0$ and $s \geq 0$ (if $0^0$ is
  treated as $1$), it can be seen
  easily that
  \begin{align*}
    \frac{e^{-r\psi}}{\abs\psi^{1-\eps} \paren{\paren{\log\abs{\ell\psi}}^2 +1}}
    &=\paren{e^{-r\abs\psi} \abs\psi^{2-\eps}
      \frac{\ell}{\abs{\ell\psi}} \paren{\paren{\log\abs{\ell\psi}}^2
        +1}}^{-1} \\
    &\geq \paren{\paren{\frac{2-\eps}{e r}}^{2-\eps}
      \ell \paren{\frac{2}{e}}^2 +\paren{\frac{1-\eps}{e r}}^{1-\eps}}^{-1} \; . 
  \end{align*} 
  Together with the fact that $\vphi_L+\psi$ being locally bounded
  from above, it yields $u_{\eps',\eps,r} \in \Lloc[2]\paren{X}$.

  It follows from \cite{Demailly_complete-Kahler}*{\textfr{Lemme} 6.9}
  that $(u_{\eps',\eps,r}, w_{\eps',\eps,r})$ satisfies the
  $\dbar$-equation \eqref{eq:dbar-with-error-on-X} on the whole of
  $X$.
  It follows from the estimate in Theorem
  \ref{thm:dbar-eq-with-estimate_sigma=1} (with $-\vphi_L -\psi$
  replaced by $-\vphi_L-\paren{1+r}\psi$) that one can let $r \to 0^+$
  and obtain weak limits $u_{\eps',\eps,r} \wktendsto u_{\eps',\eps}$
  and $w_{\eps',\eps,r} \wktendsto w_{\eps',\eps}$ in their respective
  weighted $L^2$ spaces (after possibly passing to convergent
  subsequences).
  The estimate in Theorem \ref{thm:dbar-eq-with-estimate_sigma=1}
  still holds true for the limits.
  Since the $\dbar$-equation \eqref{eq:dbar-with-error-on-X} holds
  true for $(u_{\eps',\eps,r}, w_{\eps',\eps,r})$ on $X$ in the sense
  of currents, it also holds true for $(u_{\eps',\eps},
  w_{\eps',\eps})$ on $X$ in the same sense.
  This completes the proof.
\end{proof}


The theorem of holomorphic extension from the codimension-$1$ lc
centres of $(X,S)$ is summarised in the following theorem.
\begin{thm}[\thmparen{Theorem
  \ref{thm:ext-from-lc-with-estimate-codim-1-case}}] \label{thm:extension-sigma=1} 
  Assume the assumptions (\ref{item:curv-cond-ordinary}) and
  (\ref{item:normalisation-cond}) in Theorem
  \ref{thm:dbar-eq-with-estimate_sigma=1}.
  Let $f$ be any holomorphic section in $\cohgp0[S]{\: K_X\otimes
      L \otimes \frac{\mtidlof{\vphi_L}}{\mtidlof{\vphi_L+\psi}}}$.
  If one has
  \begin{equation*}
    \int_S \abs f_\omega^2 \lcV|1| < \infty 
  \end{equation*}
  (which holds true when either the mlc of $(X,S)$ or the mlc of
  $(X,S)$ with respect to $f$ has codimension $1$, see Definitions
  \ref{def:ad-hoc-mlc-of-f} and \ref{def:mlc-of-f}),
  then there exists a holomorphic section $F \in \cohgp0[X]{K_X\otimes
    L \otimes \mtidlof{\vphi_L}}$ such that
  \begin{equation*}
    F \equiv f \mod \mtidlof{\vphi_L+\psi}
  \end{equation*}
  with the estimate
  \begin{equation*}
    \int_X \frac{\abs F^2
      e^{-\vphi_L-\psi}}{\abs\psi \paren{\paren{\log\abs{\ell\psi}}^2 +1}}
    \leq \int_S \abs f_\omega^2 \lcV|1| \; .
  \end{equation*} 
\end{thm}

\begin{proof}
  Given any local holomorphic liftings $\set{\extu[f]_\gamma}_\gamma$ of $f$
  (i.e.~$\extu[f]_\gamma \in \mtidlof{\vphi_L}$ on some open set
  $V_\gamma$ in $X$ and $\extu[f]_\gamma \equiv f \mod \mtidlof{\vphi_L+\psi}$ on
  $V_\gamma$ for each $\gamma$) and a partition of unity
  $\set{\chi_\gamma}_\gamma$ subordinated to an open cover
  $\set{V_\gamma}_\gamma$ of $X$, the \emph{smooth} section $\extu[f] :=
  \sum_\gamma \chi_\gamma \extu[f]_\gamma$ of the coherent sheaf $K_X\otimes
  L \otimes \mtidlof{\vphi_L}$ satisfies the properties
  \begin{equation*}
    f \equiv \extu[f] \mod \smooth_X \otimes \mtidlof{\vphi_L+\psi}
    \qquad\text{ and }\qquad
    \dbar \extu[f] \equiv 0 \mod \smooth_X \otimes \mtidlof{\vphi_L +\psi}
  \end{equation*}
  as shown in \cite{Demailly_extension}*{Proof of Thm.~2.8}.
  Notice that one has the inequality $\log\abs{\frac{\ell\psi} e} \leq
  \frac {\ell}{e^3 \delta'} e^{-\delta' \psi}$ using
  \eqref{eq:xlogx-estimate} for any $\delta' > 0$, and the assumption
  (\ref{item:curv-cond-ordinary}) in Theorem
  \ref{thm:dbar-eq-with-estimate_sigma=1} infers that
  $\vphi_L+\paren{1+\delta'}\psi$ is psh for all $\delta' \in
  [0,\delta]$.
  Upper-boundedness of $\psi$ also implies that $\vphi_L
  +\paren{1+\delta'}\psi \leq \vphi_L +\psi$.
  Therefore, by the strong effective openness property of multiplier ideal
  sheaves of psh functions (see \cite{Hiep_openness}*{Main Thm.}, also
  \cite{Guan&Zhou_effective_openness}), it follows that
  \begin{equation*}
    \dbar \extu[f] \in \smooth_X \otimes \mtidlof{\vphi_L +(1+\delta')\psi}
    \quad\text{for } 0 < \delta' \ll 1 \; ,
  \end{equation*}
  which in turn implies that
  \begin{equation*}
    \abs{\dbar\extu[f]}_\omega^2 e^{-\vphi_L-\psi} \log\abs{\frac{\ell\psi} e}
  \end{equation*}
  is integrable over $X$.

  Theorem \ref{thm:dbar-eq-with-estimate_sigma=1} and Proposition
  \ref{prop:dbar-soln-continuated-to-X} can then be invoked to provide
  the sections $u_{\eps',\eps}$ and $w_{\eps',\eps}$ with the
  estimate as stated in the Theorem such that 
  %
  they satisfy the $\dbar$-equation
  \eqref{eq:dbar-with-error-on-X}, namely $\dbar u_{\eps',\eps}
  +w_{\eps',\eps} =v_\eps$, on the whole of $X$. 
  Both $u_{\eps',\eps}$ and $w_{\eps',\eps}$ are smooth on $X$
  by the regularity of the $\dbar$ operator and the smoothness of
  $v_\eps$.

  Notice that
  $\frac{e^{-\vphi_L-\psi}}{\abs\psi^{1-\eps} \paren{\paren{\log\abs{\ell\psi}}^2 
      +1}}$ is \emph{not} integrable at any point of $S$ for any
  $\eps >0$,
  the finiteness of the integral of $u_{\eps',\eps}$ implies that,
  around every point in $X$,
  there exists a \emph{local function} $g \in 
  \mtidlof{\vphi_L+\psi}$ (a monomial in local coordinates under the
  snc assumption on $\vphi_L$ and $\psi$) such that
  $\abs{u_{\eps',\eps}} \leq C \abs g$ for some constant $C >0$, which
  in turn implies that $u_{\eps',\eps} \in \smooth_X \otimes
  \mtidlof{\vphi_L+\psi}$.

  Recall that $\abs{\theta'_\eps} \leq \frac{A-B}{AB} +\eps_0$ on $X$
  by the choice of $\theta_\eps$.
  Setting $F_{\eps',\eps} := \theta_\eps \extu[f] - u_{\eps',\eps}$
  (which is an extension of $f$) and using the inequality
  \begin{equation*}
    \abs{F_{\eps',\eps}}^2 \leq \paren{1 +\alpha^{-1}}
    \abs{\theta_\eps \extu[f]}^2 +\paren{1+\alpha} \abs{u_{\eps',\eps}}^2
  \end{equation*}
  for any positive real number $\alpha$,
  one obtains the estimate
  \begin{align*}
      &\qquad \int_{X} \frac{\abs{F_{\eps',\eps}}^2
        e^{-\vphi_L-\psi}}{\abs\psi \paren{\paren{\log\abs{\ell\psi}}^2
          +1}}
      +\frac 1{\eps'} \int_{X}
      \abs{w_{\eps',\eps}}_\omega^2 e^{-\vphi_L-\psi}
      \log\abs{\frac{\ell\psi} e} \\
      &\leq
      \begin{aligned}[t]
        \paren{1+\alpha^{-1}} \int_{X} \frac{\abs{\theta_\eps
            \extu[f]}^2 e^{-\vphi_L-\psi}}
        {\abs\psi \paren{\paren{\log\abs{\ell\psi}}^2 +1}}
        +\paren{1+\alpha} &\int_{X}
        \frac{\abs{u_{\eps',\eps}}^2 e^{-\vphi_L -\psi}}
        {\abs{\psi}^{1 -\eps} \paren{\paren{\log\abs{\ell\psi}}^2 +1}} \\
        &+\frac 1{\eps'} \int_{X}
        \abs{w_{\eps',\eps}}_\omega^2 e^{-\vphi_L-\psi}
        \log\abs{\frac{\ell\psi} e}
      \end{aligned}
      \\ 
      &\leq
      \begin{aligned}[t]
        \paren{1+\alpha^{-1}} \int_{X} \frac{\abs{\theta_\eps
            \extu[f]}^2 e^{-\vphi_L-\psi}}
        {\abs\psi \paren{\paren{\log\abs{\ell\psi}}^2 +1}}
        &+\frac{1+\alpha}{\eps'} \int_{X} \abs{\theta_\eps \dbar
          \extu[f]}_\omega^2
        e^{-\vphi_L -\psi} \log\abs{\frac{\ell\psi} e} \\
        &+\paren{1+\alpha} \paren{\frac{A-B}{AB} +\eps_0}^2
        \frac{\eps}{1-\eps} \int_X \frac{\abs{\extu[f]}^2 e^{-\vphi_L
            -\psi} } {\abs{\psi}^{1+\eps}} \; .
      \end{aligned}
  \end{align*}
  The assumption that $\int_S \abs f_\omega^2 \lcV|1|$ being
  well-defined and finite infers that the integral $\int_X
  \frac{\abs{\extu[f]}^2 e^{-\vphi_L-\psi}} {\abs\psi^{1+\eps}}$
  converges for all $\eps >0$, and thus so is $\int_{X}
  \frac{\abs{\extu[f]}^2 e^{-\vphi_L-\psi}}
  {\abs\psi \paren{\paren{\log\abs{\ell\psi}}^2 +1}}$.
  As a result, the first two terms on the right-hand-side both
  converge to $0$ as $\eps \tendsto 0^+$ by the dominated convergence
  theorem, and the last term converges to $\text{const.} \times \int_S
  \abs f_\omega^2 \lcV|1|$, which is finite by assumption.

  Set $\eps' := \paren{\int_{X} \abs{\theta_\eps \dbar
      \extu[f]}_\omega^2 e^{-\vphi_L -\psi} \log\abs{\frac{\ell\psi}
      e}}^{\frac 12}$, which converges to $0$ as $\eps \tendsto 0^+$.
  All the subscripts ``$\eps'\:$'' are omitted in what follows.
  Then, it follows from the above estimate that
  $w_{\eps} \tendsto 0$ in $L^2(X; e^{-\vphi_L-\psi})$ as
  $\eps \tendsto 0^+$.
  One can also extract a weakly convergent subsequence from
  $\seq{F_{\eps}}_\eps$ such that $F := \lim_{\eps_\mu \tendsto
    0^+} F_{\eps_\mu}$ exists as a weak limit in $L^2\paren{X;
    \frac{e^{-\vphi_L -\psi}}
    {\abs\psi \paren{\paren{\log\abs{\ell\psi}}^2 +1}}}$, which
  turns out to be the desired holomorphic extension of $f$, as is
  justified below.
  
  That $F$ is truly a holomorphic extension of $f$ can be seen using
  the argument similar to that in \cite{Demailly_extension}*{(5.24)}.
  On any open set $V$ (which can be assumed to be a polydisc on which
  $L$ is trivialised without loss of generality) in the given open cover
  $\set{V_\gamma}_\gamma$ of $X$, one can solve $\dbar s_{\eps}
  = w_{\eps}$ for $s_{\eps}$ with the $L^2$
  \textde{Hörmander} estimate $\norm{s_{\eps}}_{V,\vphi_L+\psi}^2 \leq C
  \norm{w_{\eps}}_{X,\vphi_L+\psi}^2$ (which implies
  $s_{\eps} \in \smooth_X \otimes \mtidlof{\vphi_L+\psi}$ on
  $V$, where $\norm\cdot_{V,\vphi_L+\psi}$,
  resp.~$\norm\cdot_{X,\vphi_L+\psi}$, denotes the $L^2$ norm on 
  $V$, resp.~on $X$, with the weight $e^{-\vphi_L-\psi}$).
  Therefore, $s_\eps \tendsto 0$ in $L^2(V; e^{-\vphi_L-\psi})$ as
  $\eps \tendsto 0^+$,
  and, passing to suitable subsequences of $\seq{F_{\eps_\mu}}_\mu$ and
  $\seq{s_\eps}_\eps$, one has $s_{\eps_{\mu_k}} \tendsto 0$
  pointwisely almost everywhere (a.e.) on $V$ while $F_{\eps_{\mu_k}}
  \wktendsto F$ weakly in the weighted $L^2$ space on $X$ as
  $\eps_{\mu_k} \tendsto 0^+$.
  Moreover, $F_{\eps_{\mu_k}} - s_{\eps_{\mu_k}}$ is a holomorphic
  extension of $f$ on $V$ with both norm-squares
  \begin{equation*}
    \int_V \frac{\abs{F_{\eps_{\mu_k}} - s_{\eps_{\mu_k}}}^2}
    {\abs\psi \paren{\paren{\log\abs{\ell\psi}}^2 +1}}
    \leq     \int_V \frac{\abs{F_{\eps_{\mu_k}} - s_{\eps_{\mu_k}}}^2
      e^{-\vphi_L-\psi}} 
    {\abs\psi \paren{\paren{\log\abs{\ell\psi}}^2 +1}}
  \end{equation*}
  being bounded above uniformly in $\eps_{\mu_k}$.
  As $\abs\psi \paren{\paren{\log\abs{\ell\psi}}^2 +1}$ belongs
  to $L^1(V)$ (or $L^1(X)$), the Cauchy--Schwarz
  inequality applied to the norm-square on the left-hand-side above
  assures that $F_{\eps_{\mu_k}} -s_{\eps_{\mu_k}}$ is also
  bounded above in $L^{1}(V)$ uniformly in $\eps_{\mu_k}$.
  Being holomorphic, Cauchy's estimate and the above boundedness
  guarantee that the sequence $\seq{F_{\eps_{\mu_k}}
    -s_{\eps_{\mu_k}}}_k$ is locally bounded above in $V$.
  Montel's theorem then assures that there is a subsequence which
  converges locally uniformly in $V$ to a holomorphic function $F_V$
  on $V$.
  Notice that, if $V\cap S \neq \emptyset$, then $F_V \equiv f \mod
  \mtidlof{\vphi_L+\psi}$ on $V$, as can be seen, under the snc
  assumption \ref{assumption:snc}, from the facts that
  $F_{\eps_{\mu_k}} -s_{\eps_{\mu_k}} \equiv f \mod
  \mtidlof{\vphi_L+\psi}$ for all $\eps_{\mu_k}$ and that all Taylor
  coefficients of $F_{\eps_{\mu_k}} -s_{\eps_{\mu_k}}$ around any
  point have to converge to the corresponding Taylor coefficients of
  $F_V$.
  As a result, there is a subsequence of $\seq{F_{\eps_\mu}}_\mu$
  which converges pointwisely a.e.~on $V$ to the holomorphic extension
  $F_V$ of $f$.
  It turns out that $F = F_V$ a.e.~on $V$.
  By considering all open sets $V$ in a cover of $X$, it follows that
  $F$ is indeed a holomorphic extension of $f$ on $X$, after possibly
  altering its values on a measure $0$ set.

  Finally, to obtain the acclaimed estimate for $F$, noting that $F$
  comes with the estimate
  \begin{equation*}
    \int_X \frac{\abs F^2 e^{-\vphi_L-\psi}}
    {\abs\psi \paren{\paren{\log\abs{\ell\psi}}^2 +1}}
    \leq \paren{1+\alpha} \paren{\frac{A-B}{AB} +\eps_0}^2 \int_S \abs
    f_\omega^2 \lcV|1| 
  \end{equation*}
  and letting $\alpha \tendsto 0^+$, $A \tendsto +\infty$, $B \tendsto
  1^+$ and $\eps_0 \tendsto 0^+$ (and choosing the limit of $F$
  suitably such that it converges locally uniformly) yield the desired result.
\end{proof}

\begin{remark} \label{rem:normalisation-control}
  In some applications, it is necessary to control how fast the
  estimate grows when the constant $\delta$ in the normalisation of $\psi$
  shrinks.
  The constant $\ell$ in the estimate is there to give a more precise
  control.
  Choose $\ell := \delta$ and write
  \begin{equation*}
    \psi =: \psi_0 - \frac a\delta \; ,
  \end{equation*}
  where $a > 0$ is a constant and $\sup_X\psi_0 = 0$.
  Then $a$ can be chosen independent of $\delta$ such that the assumption
  (\ref{item:normalisation-cond}) in Theorem
  \ref{thm:dbar-eq-with-estimate_sigma=1} is satisfied.
  Indeed, choosing $a$ such that
  \begin{equation*}
    a > e \qquad\text{and}\qquad \frac 1a + \frac 2{a \log \frac ae} = 1
  \end{equation*}
  suffices (thus $a \approx 4.6805$).
  In this case, the estimate obtained is
  \begin{equation*}
    \int_X \frac{\abs F^2 e^{-\vphi_L-\psi_0}}
    {\abs{\delta\psi_0-a} \paren{\paren{\log\abs{\delta\psi_0 -a}}^2
        +1}}
    \leq \frac 1\delta \int_S \abs f_\omega^2 \lcV|1|[\psi_0] \; .
  \end{equation*}
  Note that $\frac{e^{-\psi_0}}
  {\abs{\delta\psi_0-a} \paren{\paren{\log\abs{\delta\psi_0 -a}}^2
      +1}}$ is bounded below by a positive constant independent of
  $\delta$ when $\delta < a$ (which can be seen easily by
  applying \eqref{eq:xlogx-estimate} suitably).
\end{remark}

\begin{remark} \label{rem:McNeal-Varolin-weights}
  Concerning the weight in the norm of the extension $F$, McNeal and
  Varolin prove in \cite{McNeal&Varolin_adjunction} some estimates with 
  better weights.
  More precisely, for the case $\psi := \psi_S = \phi_S -\sm\vphi_S$
  (which is suitably normalised for each of the weights below),
  they obtain holomorphic extension with an estimate in the norm with
  any of the following weights:
  \begin{equation*}
    \frac{\delta' e^{-\psi_S}}{\abs{\psi_S}^{1+\delta'}} \; , \;
    \frac{\delta' e^{-\psi_S}}{\abs{\psi_S}
      \paren{\log\abs{\psi_S}}^{1+\delta'}} \; , \;
    \dots \; , \;
    \frac{\delta' e^{-\psi_S}}{\abs{\psi_S} \cdot \log\abs{\psi_S} \dotsm
      \log^{\circ (N-1)}\abs{\psi_S} \cdot \paren{\log^{\circ
          N}\abs{\psi_S}}^{1+\delta'}} \; ,
  \end{equation*}
  where $\delta' \in (0,1]$ is a fixed number in each case, and
  $\log^{\circ j}$ denotes the composition of $j$ copies of
  $\log$ functions here.
  It would be interesting to see if it is possible to obtain these
  weights in the setting of this paper. 
\end{remark}

\begin{remark} \label{rem:optimal-constant}
  It is not clear to the authors whether Theorem
  \ref{thm:extension-sigma=1}, if allowing $X$ to be non-compact, does
  include the results in \cite{Blocki_Suita-conj} and
  \cite{Guan&Zhou_optimal-L2-estimate} on the optimal constant for the
  estimate, although the constant in the current estimate looks
  ``optimal''.
\end{remark}

\subsection{Extension theorem with a sequence of potentials}
\label{sec:ext-thm-with-seq-potentials}


In applications it is often necessary to deal with a sequence of
potentials $\seq{\dep\vphi_L +m_1\dep\psi}_{k\in\Nnum}$ rather than
just a single one.
Following the idea of J.-P.~Demailly, it is advantageous to allow the
curvatures of such sequence possessing slight 
negativity which diminishes in the limit.
It is the purpose of this section to handle such cases.

Assume that
\begin{enumerate}
\item  there are sequences $\seq{\dep\vphi_L}_{k\in\Nnum}$
  and $\seq{\dep\psi}_{k\in\Nnum}$ which satisfy all the assumptions
  in Section \ref{sec:setup} in place of $\vphi_L$ and $\psi$
  respectively,

\item \label{item:potential-seq-limit}
  both $\dep\vphi_L +m_1\dep\psi$ and $\dep\psi$ converge in $L^1$ to
  $\dep[\infty]\vphi_L +m_1\dep[\infty]\psi$ and $\dep[\infty]\psi$
  respectively, with the property that
  \begin{equation*}
    \dep[\infty]\vphi_L +m_1\dep[\infty]\psi
    \lesssim_\tlog
    \dep\vphi_L +m_1\dep\psi
    \quad\text{ and }\quad
    \abs{\dep[\infty]\psi} \lesssim_\tlog \abs{\dep\psi}
  \end{equation*}
  for every $k \in\Nnum$ (where the constants involved in
  $\lesssim_\tlog$'s may depend on $k$), and 
  
\item \label{item:potential-seq-decreasing-ideals}
  the multiplier ideal sheaf of $\dep\vphi_L+m_1\dep\psi$
  decreases as $k$ increases, i.e.
  \begin{equation*}
    \mtidlof{\dep[k+1]\vphi_L+m_1\dep[k+1]\psi}
    \subset \mtidlof{\dep\vphi_L+m_1\dep\psi}
    \quad\text{for all } k \in\Nnum \; .
  \end{equation*}
\end{enumerate}
In particular, all $\dep\vphi_L$'s and $\dep\psi$'s are assumed to
have only neat analytic singularities.
All families $\set{\mtidlof{\dep\vphi_L+m \dep\psi}}_{m
\in \fieldR_{\geq 0}}$ have the same jumping numbers $m_0$ and $m_1$
and the annihilator $\Ann_{\holo_X} \paren{
  \frac{\multidl\paren{\dep\vphi_L +  m_0 \dep\psi}} {\multidl\paren{\dep\vphi_L
      +  m_1 \dep\psi}} }$ defines the same reduced subvariety $S$ for all
$k$ by assumption.
However, the snc assumption \ref{assumption:snc} is \emph{not}
assumed unless explicitly mentioned, as there may not be simultaneous
resolution for all the potentials in general.

\begin{thm}\label{thm:extension-with-seq-of-potentials}
  Suppose that
  \begin{enumerate}[series=ext-hypo-12, label=(\theenumi)${}_k$, ref=(\theenumi)${}_k$]
  \item 
    \label{item:curv-cond-seq}
    there exists $\delta > 0$ (independent of $k$) such that, for any
    $k \in\Nnum$,
    \begin{equation*}
      \ibddbar
      \paren{\dep\vphi_L +m_1\dep\psi} + \beta\ibddbar {\dep\psi} \geq
      -\frac{1}{k} \omega
      \quad \text{on $X$ for all }\beta \in [0, \delta]
      \; , \text{ and }
    \end{equation*}
    
  \item 
    \label{item:normalisation-seq}
    for any given constant $\ell >0$ and for each $k\in\Nnum$, the
    function $\dep\psi$ is normalised (by adding a suitable
    constant for each $k$ without affecting convergence) such that
    \begin{equation*}
      \dep\psi < -\frac e \ell \quad\text{ and }\quad
      \frac 1{\abs{\dep\psi}} + \frac{2}{\abs{\dep\psi}
        \log\abs{\frac{\ell\dep\psi} e}}
      \leq \delta \; . 
    \end{equation*}

  \end{enumerate}
  Let $f$ be any holomorphic section in $\cohgp0[S]{K_X\otimes
      L \otimes \frac{\bigcap_k \mtidlof{\dep\vphi_L +m_0 \dep\psi}}
      {\bigcap_k \mtidlof{\dep\vphi_L +m_1 \dep\psi} } }$ such that
  \begin{equation*}
    \lim_{k \tendsto \infty} \int_S \abs f_\omega^2
    \lcV|1,(m_1)|<\dep\vphi_L>[\dep\psi] < \infty \; .
  \end{equation*}
  Then, there exists a holomorphic section $F \in \cohgp0[X]{K_X\otimes
    L \otimes \bigcap_k \mtidlof{\dep\vphi_L +m_0\dep\psi}}$ such that
  \begin{equation*}
    F \equiv f \mod \bigcap_k \mtidlof{\dep\vphi_L+m_1\dep\psi}
  \end{equation*}
  with the estimate
  \begin{equation*}
    \int_X \frac{\abs F^2
      e^{-\dep[\infty]\vphi_L-m_1\dep[\infty]\psi}}{\abs{\dep[\infty]\psi}
      \paren{\paren{\log\abs{\ell\dep[\infty]\psi}}^2 +1}} 
    \leq
    \lim_{k \tendsto \infty} \int_S \abs f_\omega^2
    \lcV|1,(m_1)|<\dep\vphi_L>[\dep\psi] \; .
  \end{equation*} 
\end{thm}

\begin{remark}
  In general, $\bigcap_k \mtidlof{\dep\vphi_L +m_1 \dep\psi} \neq
  \mtidlof{\dep[\infty]\vphi_L +m_1 \dep[\infty]\psi}$, as the example
  in \cite{DPS94}*{Example 1.7} shows (see also
  \cite{Koike15_min-sing-metric_nef}*{Example 3.5}, or Example
  \ref{eg:Ohsawa-example}).
\end{remark}

\begin{proof}
  \newcommand{\arbN}[2][]{\mathcal{N}_{#1}\paren{#2}}

  For simplicity, assume that $m_0 =0$ and $m_1=1$ as before.
  The proof goes with the standard technique applied in, for example,
  \cite{Demailly_extension} (which applies
  \cite{Demailly_extension}*{Prop.~3.12} to handle the diminishing
  negative curvature).
  
  For each $k \in\Nnum$, applying the curvature assumption
  \ref{item:curv-cond-seq} and the normalisation assumption
  \ref{item:normalisation-seq} to the curvature term (in
  \begin{termR}
    \termRcolor
  \end{termR})
  of the twisted Bochner--Kodaira formula in Lemma
  \ref{lem:BK-formula-all-sigma} with $\sigma =1$ yields the
  inequality
  \begin{equation*}
    \int_{X^\circ} \abs{\dbar\zeta}_{\omega,\vphi}^2 \dep\eta_\eps
    +\int_{X^\circ} \abs{\dfadj\zeta}_\vphi^2 \paren{\dep\eta_\eps
      +\dep\lambda_\eps}
    +\frac 1k \int_{X^\circ} \abs{\zeta}_{\omega,\vphi}^2 \dep\eta_\eps
    \geq
    \eps \int_{X^\circ} \frac{1-\eps}{\abs{\dep\psi}^2} 
    \abs{\:\idxup{\diff\dep\psi} \ctrt \zeta}_\vphi^2 \dep\eta_\eps
  \end{equation*}
  for all compactly supported $K_X \otimes L$-valued $(0,1)$-forms
  $\zeta$ (in which $\vphi =\dep\vphi_L+\dep\psi+\dep\nu$ and the
  formal adjoint $\dfadj$ both depend on $k$).
  Using the notation in the proof of Theorem
  \ref{thm:dbar-eq-with-estimate_sigma=1} and with the same argument
  there, one obtains
  \begin{equation*}
    \begin{aligned}
      \abs{\inner\zeta{v_\eps}}
      &\leq \abs{\inner{\paren{\zeta}_{\ker\dbar}
          \:}{~v_\eps^{(1)} }}
      +\abs{\inner{\paren{\zeta}_{\ker\dbar}\:}{~v_\eps^{(2)} }} \\
      &\leq \paren{\paren{\arbN[1]{\dfadj\zeta}}^2 +
        \paren{\frac 1k +\eps'}\norm{\zeta}^2}^{\frac 12}
      \paren{\paren{\arbN[2]{\extu[f]}}^2 + \frac{1}{\eps'}
        \norm{\theta_\eps \dbar\extu[f]}^2}^{\frac 12}
    \end{aligned}
  \end{equation*}
  for any $\eps' > 0$ and $k \in \Nnum$ (here, $v_\eps$ also depends
  on $k$, as $\theta_\eps = \theta\circ \abs{\dep\psi}^\eps$ does).
  The Riesz representation theorem, together with the argument in
  Proposition \ref{prop:dbar-soln-continuated-to-X}, provides a
  solution $\paren{\dep u_{\eps',\eps}, \dep w_{\eps',\eps}}$ to the
  $\dbar$-equation $\dbar \dep u_{\eps',\eps} +\dep w_{\eps',\eps}
  =v_\eps$ on the whole of $X$.
  Setting $\dep F_{\eps',\eps} := \theta_\eps \extu[f] -\dep
  u_{\eps',\eps}$, the argument in the proof of Theorem
  \ref{thm:extension-sigma=1} then yields
  \begin{align*}
    &\qquad \int_{X} \frac{\abs{\dep F_{\eps',\eps}}^2
      e^{-\dep\vphi_L-\dep\psi}}{\abs{\dep\psi} \paren{\paren{\log\abs{\ell\dep\psi}}^2
      +1}}
      +\frac 1{\frac 1k +\eps'} \int_{X}
      \abs{\dep w_{\eps',\eps}}_\omega^2 e^{-\dep\vphi_L-\dep\psi}
      \log\abs{\frac{\ell\dep\psi} e} \\
    &\leq
      \begin{aligned}[t]
        \paren{1+\alpha^{-1}} \int_{X} \frac{\abs{\theta_\eps
            \extu[f]}^2 e^{-\dep\vphi_L-\dep\psi}}
        {\abs{\dep\psi} \paren{\paren{\log\abs{\ell\dep\psi}}^2 +1}}
        +&\frac{1+\alpha}{\eps'} \int_{X} \abs{\theta_\eps \dbar
          \extu[f]}_\omega^2
        e^{-\dep\vphi_L -\dep\psi} \log\abs{\frac{\ell\dep\psi} e} \\
        +\paren{1+\alpha} &\paren{\frac{A-B}{AB} +\eps_0}^2
        \frac{\eps}{1-\eps} \int_X \frac{\abs{\extu[f]}^2 e^{-\dep\vphi_L
            -\dep\psi} } {\abs{\dep\psi}^{1+\eps}}
      \end{aligned}
  \end{align*}
  for some $\alpha >0$, where
  \begin{equation*}
    \dep F_{\eps',\eps} \equiv f \mod \smooth_X \otimes
    \mtidlof{\dep\vphi_L +\dep\psi} \; .\footnotemark
  \end{equation*}
  \footnotetext{Here $f$ is abused to mean its image under the map
  $\frac{\bigcap_{k'} \mtidlof{\dep[k']\vphi_L +m_0 \dep[k']\psi}}
  {\bigcap_{k'} \mtidlof{\dep[k']\vphi_L +m_1 \dep[k']\psi}} \to
  \frac{\mtidlof{\dep\vphi_L +m_0 \dep\psi}} {\mtidlof{\dep\vphi_L
      +m_1 \dep\psi} }$.}

  Notice that the assumption of $f$ being $L^2$ with respect to the
  limit of lc-measures
  implies that $f$ is $L^2$ with respect to
  $\lcV|1|<\dep\vphi_L>[\dep\psi]$ for every $k\gg 0$, which in turn
  implies that $\int_{X} \frac{\abs{\extu[f]}^2
    e^{-\dep\vphi_L-\dep\psi}}
  {\abs{\dep\psi} \paren{\paren{\log\abs{\ell\dep\psi}}^2 +1}}$
  is finite for each $k \gg 0$ (see the proof of Theorem
  \ref{thm:extension-sigma=1}).

  Choose $\eps' := \paren{\int_{X} \abs{\theta_\eps \dbar
      \extu[f]}_\omega^2 e^{-\dep\vphi_L -\dep\psi}
    \log\abs{\frac{\ell\dep\psi} e}}^{\frac 12}$ and omit all
  subscripts ``$\eps'\:$'' as before.
  Notice that the right-hand-side of the above estimate is bounded above
  uniformly, thanks to the assumption that $f$ being $L^2$ with respect to the
  lc-measure, when the limits are taken in the order $\eps \tendsto 0^+$
  followed by $k \tendsto \infty$. 
  The required section $F$ is then obtained after first taking the limit
  $\eps \tendsto 0^+$ (obtaining the weak limits $\dep F$ of $\dep F_\eps$
  and $\dep w$ of $\dep w_\eps$ in their respective $L^2$ spaces), then
  $k \tendsto \infty$ (obtaining the weak limit $F$ of $\dep F$ while
  $\dep w \tendsto 0$ strongly), and followed by $\alpha \tendsto
  0^+$, $A \tendsto +\infty$, $B \tendsto 1^+$ and $\eps_0 \tendsto
  0^+$.
  The acclaimed estimate also follows.
  
  To justify that $F$ is the required holomorphic extension of $f$,
  consider any polydisc $V$ in the given open cover of $X$ and solve
  $\dbar \dep s_\eps = \dep w_\eps$ for $\dep s_\eps$ on $V$ with the
  $L^2$ \textde{Hörmander} estimate
  \begin{equation*}
    \norm{\dep
      s_{\eps}}_{V,\dep\vphi_L+\dep\psi}^2 \leq C \norm{\dep
      w_{\eps}}_{X,\dep\vphi_L+\dep\psi}^2 \; ,
  \end{equation*}
  where the constant $C$ is independent of $k$ and $\eps$.
  This assures that one can extract weak limit $\dep s$ of $\dep
  s_\eps$ as $\eps \to 0^+$.
  As $\dep F_\eps - \dep s_\eps$ is holomorphic on $V$ with the unweighted
  $L^1$ norm bounded from above uniformly in $\eps$ (followed from the same
  argument as in Theorem \ref{thm:extension-sigma=1}), it converges
  \emph{locally uniformly} on $V$ to the holomorphic section $\dep F -
  \dep s$ after passing to a subsequence.
  This also implies that
  \begin{equation} \label{eq:pf-extension-at-k} \tag{$*$}
    \dep F - \dep s \equiv f \mod \mtidlof{\dep\vphi_L+\dep\psi}
    \quad\text{ on }V
  \end{equation}
  (which can be seen by temporary taking a log-resolution of
  $(X,\dep\vphi_L+\dep\psi)$ and arguing as in Theorem
  \ref{thm:extension-sigma=1}).

  Notice that the unweighted $L^1$ norm of $\dep F -\dep s$ may not be
  bounded above uniformly in $k$ since $\dep\psi$ depends on $k$.
  To get around that, notice that $\frac 1{\abs{\dep\psi}^2} \lesssim
  \frac 1{\abs{\dep\psi} \paren{\paren{\log\abs{\ell\dep\psi}}^2 +1}}$
  via inequality \eqref{eq:xlogx-estimate} with the constant in
  $\lesssim$ independent of $k$.
  \textde{Hölder's} inequality infers that
  \begin{equation*}
    \int_V \abs{\dep F -\dep s}^{\frac 23}
    \leq \paren{\int_V \frac{\abs{\dep F -\dep
          s}^{2}}{\abs{\dep\psi}^2}}^{\frac 13}
    \paren{\int_V \abs{\dep\psi}}^{\frac 23} \; .
  \end{equation*}
  As $\dep\psi \tendsto \dep[\infty]\psi$ in the $L^1$ norm, this
  assures that $\dep F -\dep s$ is bounded above uniformly in $k$ in
  the $L^{\frac 23}$ norm.
  This retains the local uniform boundedness of $\dep F -\dep s$ in
  the sup-norm via the use of the Harnack inequality for
  psh functions (see, for example, \cite{Demailly}*{Ch.~I,
    Prop.~4.22(b)}).
  The rest is then the same as the treatment in the proof of Theorem
  \ref{thm:extension-sigma=1}.
  This shows that $F$ is holomorphic.

  Since \eqref{eq:pf-extension-at-k} also implies that
  \begin{equation*} 
    \dep F - \dep s \equiv f \mod \mtidlof{\dep[k']\vphi_L+\dep[k']\psi}
    \quad\text{ on }V
  \end{equation*}
  for all $k' \leq k$, as followed from the assumption
  \eqref{item:potential-seq-decreasing-ideals} stated at the beginning of
  Section \ref{sec:ext-thm-with-seq-potentials}.
  One then sees that
  \begin{equation*}
    F \equiv f \mod \bigcap_k \mtidlof{\dep\vphi_L +\dep\psi} \; .
  \end{equation*}

  To see that $F$ is in $\bigcap_k \mtidlof{\dep\vphi_L}$, notice that
  $\abs{\dep[\infty]\psi} \leq \abs{\dep\psi} + C$ for some $C >0$,
  and therefore
  \begin{equation*}
    \frac{e^{-\dep\vphi_L -\dep\psi}}
    {\paren{\abs{\dep\psi}
        +C} \paren{\paren{\log\paren{\ell\abs{\dep\psi} +\ell C}}^2 +1}}
    \lesssim 
    \frac{e^{-\dep[\infty]\vphi_L -\dep[\infty]\psi}}
    {\abs{\dep[\infty]\psi} \paren{\paren{\log\abs{\ell\dep[\infty]\psi}}^2 +1}}
  \end{equation*}
  by the assumption \eqref{item:potential-seq-limit} stated at the
  beginning of Section \ref{sec:ext-thm-with-seq-potentials}.
  Then $F$ belonging to $\bigcap_k \mtidlof{\dep\vphi_L}$ can be seen
  from the estimate.
  This completes the proof.
\end{proof}


\subsection{Illustration}
\label{sec:illustrations}


The following example illustrates how Theorem
\ref{thm:extension-with-seq-of-potentials} can be applied to obtain
the classical result on prescribing value at a point to a holomorphic
section with estimate.

\begin{example}[Extension from a point] \label{eg:extension-from-a-pt}
  Let $X$ be a projective $n$-fold and $A$ an ample line bundle on $X$
  endowed with a potential $\sm\vphi_A$.
  Set $\omega := \ibddbar\sm\vphi_A$.
  Suppose that $L$ is a pseudo-effective line bundle equipped with a
  psh potential $\vphi_L$ (with arbitrary singularities).
  The goal is to obtain a global section $F$ of the line bundle $K_X
  \otimes L^{\otimes \mu} \otimes A^{\otimes \mu}$ for some
  sufficiently large $\mu \in \Nnum$ with the prescribed value $a$ at a
  point $p \in X \setminus \paren{\vphi_{L}}^{-1}\paren{-\infty}$ with
  estimate.

  Let $\seq{\dep\vphi_L}_{k\in\Nnum}$ be a sequence of quasi-psh
  potentials with neat analytic singularities which approximates
  $\vphi_L$ and satisfies the properties
  \begin{equation*}
    \vphi_L \leq \dep[k+1]\vphi_L \lesssim_\tlog \dep\vphi_L
    \qquad\text{and}\qquad
    \ibddbar\dep\vphi_L \geq -\frac 1k \ibddbar \sm\vphi_A
    = -\frac 1k \omega
    \qquad\text{on }X
  \end{equation*}
  for all $k \in \Nnum$ (for example, $\seq{\dep\vphi_L}_{k\in\Nnum}$
  can be the approximation of $\vphi_L$ constructed in
  \cite{Demailly_regularization}).

  Let $\theta \colon [0,1] \to [0,1]$ be a smooth cut-off function
  such that $\theta \equiv 1$ on $[0,\frac 12]$ and is
  compactly supported on $[0,1)$.
  For any $p \in X \setminus \paren{\vphi_{L}}^{-1}\paren{-\infty}$
  which lies in a coordinate chart $\paren{V,\vect z}$ with
  coordinates $\vect z = (z_1,\dots, z_n)$ such that $\vect z(p)
  =\vect 0$ and $\abs{\vect z}^2 =\sum_{j=1}^n \abs{z_j}^2 < 1$ on
  $V$ and such that both $L$ and $A$ are trivialised, define
  \begin{equation*}
    \psi := \theta\paren{\abs{\vect z}^2} \log\abs{\vect z}^{2n} -e \; ,
  \end{equation*}
  which is a global function on $X$.
  It can be seen that $\psi \leq -e$ on $X$, and $m =1$ is the first jumping
  number of the family $\set{\mtidlof{\mu\dep\vphi_{L+A} +m\psi}}_{m
    \in \fieldR_{\geq 0}}$ for any $k \in\Nnum$ and 
  $\mu\in\Nnum$, where $\dep\vphi_{L+A} :=\dep\vphi_L +\sm\vphi_A$ is
  set for convenience.
  Indeed, the annihilator
  $\Ann_{\holo_X} \paren{\frac{\mtidlof{\mu\dep \vphi_{L+A}}}
    {\mtidlof{\mu\dep\vphi_{L+A} +\psi}}}$ defines the set $\set p$
  with reduced structure.

  It can be seen that, for any $k \geq 2$ and a sufficiently large
  integer $\mu \in \Nnum$, one has
  \begin{equation*}
    \ibddbar \paren{\mu \dep\vphi_{L+A} +\psi} +\beta \ibddbar\psi
    \geq \mu \paren{1-\frac 1k} \omega
    +\paren{1+\beta} \ibddbar \paren{\theta\paren{\abs{\vect z}^2} \log\abs{\vect
      z}^{2n}}
    \geq 0
  \end{equation*}
  for $\beta \in [0,1]$.
  The curvature assumption \ref{item:curv-cond-seq} and the
  normalisation assumption \ref{item:normalisation-seq} of Theorem
  \ref{thm:extension-with-seq-of-potentials} are satisfied.

  Set $\extu[f] :=\theta\paren{\abs{\vect z}^2} \:dz_1 \wedge \dots
  \wedge dz_n$ and $f := \extu[f](p)$.
  Given any constant $a$, in order to obtain a holomorphic extension
  of $af$ with estimate, it remains to check that the limit of
  $1$-lc-measures $
  \lim_{k\tendsto \infty} \abs{af}_\omega^2 \lcV|1|<\mu\dep\vphi_{L+A}>$
  is finite at $p$.
  
  Let $\pi \colon \rs X \to X$ be the blow-up of $X$ at $p$
  with exceptional divisor $E$.
  Then, $\pi^*\psi = n\phi_E  -\sm\vphi_{nE}$ for some smooth potential
  $\sm\vphi_{nE}$ (on $E^{\otimes n}$) and $K_{\rs X /X} = E^{\otimes
    (n-1)}$.
  Let $\rs U$ be a neighbourhood in $\rs X$ covering a dense subset of
  $E$ with coordinates $(\vect w, s_E) = (w_1 , w_2 , \dots ,
  w_{n-1}, s_E)$ given by $\pi^*z_j = s_E w_j$ for $j= 1,\dots, n-1$
  and $\pi^*z_n = s_E$ such that $E \cap \rs U = \set{s_E = 0}$.
  It follows from Proposition \ref{prop:lc-measure} that
  \begin{align*}
    \int_{\set p} \abs{af}_\omega^2 \lcV|1|<\mu\dep\vphi_{L+A}>
    &=\lim_{\eps \tendsto 0^+} \eps \int_{\rs X}
    \frac{\abs{a \pi^*\theta}^2}{\abs{\mu\pi^*\psi}^{1+\eps}}
      e^{-\mu \pi^*\dep\vphi_{L+A} -\phi_E +\sm\vphi_{nE}} 
      \:d\vol_E
      \:\wedge \pi\ibar ds_E \wedge d\conj{s_E}
    \\
    &=\frac{\pi}{n}\int_E \abs{a}^2
      e^{-\mu \pi^*\dep\vphi_{L+A} +\sm\vphi_{nE}} \:d\vol_{E} \\
    &=\abs{a}^2 e^{-\mu\dep\vphi_{L+A}}(p) \;\frac{\pi}{n}
      \int_{\rs U \cap E} \frac{e^e}{\paren{1+\abs{\vect w}^2}^n}
      \:\bigwedge_{j=1}^{n-1}\paren{\pi \ibar dw_j \wedge d\conj{w_j}}
    \\
    &=\abs{a}^2 e^{-\mu\dep\vphi_{L+A}}(p) \;\frac{\pi}{n}
      \frac{\pi^{n-1}}{(n-1)!} e^e \;\;
      \tendsto \; \abs{a}^2 e^{-\mu\vphi_L-\mu\sm\vphi_A}(p)
      \;\frac{\pi^n}{n!} e^e
  \end{align*}
  as $k \tendsto \infty$, which is definitely finite.
  Theorem \ref{thm:extension-with-seq-of-potentials} can now be
  invoked to obtain the required $F$ with estimate.
\end{example}


\section{Improvement to the result of Demailly--Hacon--P\u aun on plt extension}
\label{sec:improvement-DHP}

Divisors are treated as line bundles without further mention in
this section.

\subsection{Setup for the plt extension}
\label{sec:basic-setup-plt}


Let $X$ be projective.
Consider a pair $(X, S+B)$ which is plt and log-smooth with $B$ being
a $\fieldQ$-divisor and $S = \floor{S+B} = \sum_{j\in I_S} S_j$ where
each $S_j$ is reduced and irreducible (thus $S$ and $B$ have no common
irreducible component and have only snc, and \emph{the irreducible
components $S_j$ of $S$ are mutually disjoint}).
Let $\mu \in \Nnum$ be such that $\mu\paren{K_X + S + B}$ is a
$\Znum$-divisor and write
\begin{equation*}
  K_X + L := K_X + S + F := \mu \paren{K_X + S + B} \quad\paren{ \text{i.e.~$F
      := \paren{\mu - 1} \paren{K_X + S + B} + B$} \,} \; .
\end{equation*}
($F$ is defined for the convenience of readers when referred to \cite{DHP}.)
Assume that $\mu \geq 2$ and 
\begin{itemize}
\item $K_X + S + B$ is pseudo-effective (pseff);

\item $K_X + S + B \lineq_\fieldQ D$, where
  \begin{equation*}
    D
    = \sum_{j \in I_S} \nu_j S_j + D_2
    =: \nu_S \cdot S + D_2
  \end{equation*}
  is an effective $\fieldQ$-divisor with snc support, and $S$ and
  $D_2$ have no common components;

\item $\supp S \subset \supp D$ (i.e.~$\nu_j \neq 0$ for all $j\in I_S$).

\item no irreducible components of $S$ lies in the \emph{diminished stable
  base locus} $\sBase[-]{K_X+S+B}$ (see, for example, \cite{DHP}*{\S
    2.1} for the definition).
\end{itemize}

Let $\rho := K_X +S +B -D$ be the $\fieldQ$-line bundle in
$\Pic^0(X) \otimes \fieldQ$ which admits a
smooth \emph{pluriharmonic} potential $\sm\vphi_\rho$, i.e.~it
curvature form is $\ibddbar \sm\vphi_\rho = 0$.
The potentials $\phi_S$, $\phi_{\nu_S \cdot S}$, $\phi_B$ and
$\phi_{D_2}$, which are defined from canonical sections of their
respective $\fieldQ$-line bundles as shown in their subscripts, are
fixed such that they are \emph{negative} under the given choice of
trivialisations.

Moreover, choose a sufficiently ample divisor $A$ on $X$ such that it
is globally generated.
Let $\set{s_{A,i}}_{i \in I_A}$ be a basis of $\cohgp 0[X]{A}$ and endowed
$A$ with a smooth psh potential $\sm\vphi_A =\log\paren{\sum_{i\in
    I_A} \abs{s_{A,i}}^2}$, which in turn provides a \textde{Kähler}
form $\omega := \ibddbar \sm\vphi_A$ on $X$, and induces a smooth
potential $\sm\vphi_{K_X}$ on $K_X$.
Fix also smooth potentials $\sm\bphi := \sm\bphi_{K_X+S+B}$ on
$K_X+S+B$ and $\sm\vphi_B$ on $B$.
All the smooth potentials are chosen such that they are
negative under the given choice of trivialisations for convenience.


\subsection{Bergman kernel potentials}
\label{sec:Bergman-potentials}


Let $\bphi_\tmin \leq \sm\bphi_{K_X+S+B}$ be a psh potential with
minimal singularities on the pseff $\fieldQ$-line bundle $K_X+S+B$
(ref.~\cite{DPS01}*{Thm.~1.5}).
Since $\phi_{\nu_S \cdot S} +\phi_{D_2} +\sm\vphi_\rho$ is also a psh
potential of $K_X+S+B$, after adding suitable constants to the
potentials, one can assume that
\begin{equation} \label{eq:min-potential-less-sing-than-D}
  \phi_{\nu_S \cdot S} +\phi_{D_2} + \sm\vphi_\rho \leq \bphi_\tmin
  \leq \sm\bphi_{K_X+S+B} \leq 0 \; . 
\end{equation}

The following construction of an approximation of $\bphi_\tmin$ is
almost a paraphrase of the discussion on the ``algebraic version of the
super-canonical metric'' in \cite{Demailly_Analytic-methods}*{\S 20.6}
with the generalisation in \cite{Demailly_Analytic-methods}*{\S 20.13}
taken into account.

Let $\Berg_{\ell, k} := \Berg_{\ell, k, A}$ be the \emph{Bergman kernel} of
$\cohgp 0 [X]{\ell k\mu\paren{K_X + S + B} + \ell A}$ with respect to the
potential $\ell k\mu \sm\bphi_{K_X+S+B} + \ell \sm\vphi_A$.
Note that, for all $j \in I_S$, $S_j \not\subset \sBase[-]{K_X+S+B}$
implies that
\begin{equation} \label{eq:S_j-not-in-Berg-kernel}
  S_j \not\subset \paren{\Berg_{\ell, k}}^{-1}(0) \quad \text{when }
  k\in\Nnum \text{ and } \ell \gg 0 \; .
\end{equation}

Now, for every $k \in \Nnum$, fix an $\ell := \dep\ell \gg 0$ and define
the \emph{Bergman kernel potentials} $\dep\bphi$ by
\begin{equation} \label{eq:definition-bphi_Berg}
  \dep\bphi := \dep\bphi_\Berg := \dep[\ell, k]\bphi_\Berg :=
  \dep[\ell, k]\bphi_{\Berg,A} := \frac{1}{\ell k\mu}
  \log\Berg_{\ell,k} - \frac{1}{k\mu} \sm\vphi_A \; .
\end{equation}
The integer $\ell$ is chosen such that the polar set of
$\dep[\ell,k]\bphi_{\Berg,A}$ is precisely the stable base locus of
the linear system of $k\mu\paren{K_X+S+B}+A$ (see Lemma
\ref{lem:bphi-indep-on-ell} for the existence of such $\ell$).
These $\dep\bphi$'s have only neat analytic singularities.

Choose $\ell \gg 0$ such that the Ohsawa--Takegoshi
extension theorem with respect to the potential $\ell k\mu \bphi_\tmin +
\ell \sm\vphi_A$ can be applied to obtain global sections of $\ell k
\mu \paren{K_X+S+B} + \ell A$ with prescribed value at any point on the
projective manifold $X$ outside of $\paren{\bphi_\tmin}^{-1}(-\infty)$
for every $k \geq 1$ (for example, one can use the version of the
extension theorem in \cite{Demailly_Analytic-methods}*{Thm.~13.6}, or
Example \ref{eg:extension-from-a-pt}, together
with the analytically singular approximation of psh functions
in \cite{Demailly_regularization}*{Prop.3.7}).
Following the arguments in \cite{Demailly_regularization}*{Prop.~3.1},
since $\bphi_\tmin \leq \sm\bphi_{K_X+S+B} =: \sm\bphi$, after adding
a suitable constant to $\bphi_\tmin$ if necessary
(where the constant is independent of $k$), one
obtains
\begin{equation} \label{eq:Berg_approx-ineq}
  \bphi_\tmin(z) \leq \dep[\ell,k]\bphi_\Berg (z)
  \leq \sup_{\zeta \in
    \Delta(z;r)} \paren{\sm\bphi(\zeta) + \frac{1}{k\mu}
    \sm\vphi_A(\zeta)} - \frac{1}{k \mu} \sm\vphi_A(z) +
  \frac{C - 2n\log r}{\ell k\mu} 
\end{equation}
for all $z \in X$, where $\Delta(z; r)$ is the polydisc of radius $r$
centred at $z$ in some coordinate chart, and $C$ is some constant
independent of $k$ and $r$.
It is emphasised here that the inequalities are valid under the fixed
trivialisations of $K_X+S+B$ and $A$ on each open subset in a fixed
cover of $X$.

The properties of $\dep\bphi = \dep[\ell,k]\bphi_\Berg$ necessary for
the present purpose are collected as follows:
\begin{subequations} \label{eq:potential-bphi}
  \begin{align}
    \label{eq:wt_properties-sing} 
    &~\dep{\bphi} \geq \phi_{\nu_S \cdot S} +\phi_{D_2} +\sm\vphi_\rho
    & \mathllap{\text{(from \eqref{eq:min-potential-less-sing-than-D}
      and \eqref{eq:Berg_approx-ineq})}} & ,
    \\
    \label{eq:wt_properties-bdd-uniformly-in-k}
    &~\dep\bphi \text{ is locally bounded above uniformly in $k$}
    &\text{(from \eqref{eq:Berg_approx-ineq})} & , \\
    \label{eq:wt_properties-curv}
    &~\ibddbar \dep{\bphi} \geq -\frac 1{k\mu} \ibddbar \sm\vphi_A
      = -\frac{1}{k\mu} \omega
    & \text{(from \eqref{eq:definition-bphi_Berg})} & ,
    \\
    &~\dep\bphi \text{ has only neat analytic singularities }
    & \text{(from \eqref{eq:definition-bphi_Berg})} & ,
                                                      \text{ and}
    \\
    \label{eq:S-not-in-polar_set}
    &~S_j \not\subset \paren{\dep\bphi}^{-1}(-\infty) \quad \forall~j
      \in I_S \text{ and } \forall~k\in\Nnum
    & \text{(from \eqref{eq:S_j-not-in-Berg-kernel})}& .
  \end{align}
\end{subequations}

The following lemma justifies the definition, in particular, the
existence of a suitable $\ell$.
\begin{lemma} \label{lem:bphi-indep-on-ell}
  There exists an $\ell_0 \in \Nnum$ such that, for every $k \in
  \Nnum$ and for all $\ell' \geq \ell \geq \ell_0$,
  $\dep[\ell,k]\bphi_\Berg \sim_\tlog \dep[\ell',k]\bphi_\Berg$, and the
  constants involved in $\sim_\tlog$ are independent of $k$, $\ell$
  and $\ell'$.
\end{lemma}

\begin{proof}
  The integer $\ell_0 >0$ is chosen
  sufficiently large such that, for any integers $\ell' > \ell \geq
  \ell_0$ and for any $p \in X
  \setminus \paren{\dep[\ell',k]\bphi_\Berg}^{-1}\paren{-\infty}$ being the
  centre of a polydisc $(\Delta,\vect z)$ with coordinates $\vect z$
  in some coordinate chart, one has
  \begin{align*}
    &~\ibddbar \paren{\ell k \mu \dep[\ell',k]\bphi_\Berg
      +\ell\sm\vphi_A -\sm\vphi_{K_X} +\paren{1+\beta}
      \theta\paren{\abs{\vect z}^2} \log\abs{\vect z}^{2n}} \\
    \overset{\text{by \eqref{eq:wt_properties-curv}}}\geq
    &~\paren{\ell -\frac{\ell k\mu}{\ell' k\mu}} \ibddbar\sm\vphi_A
      +\ibddbar\lparpht{-\sm\vphi_{K_X}
      +\paren{1+\beta}\theta\paren{\abs{\vect z}^2} \log\abs{\vect
      z}^{2n}}%
      -\sm\vphi_{K_X} +\paren{1+\beta}
      \underbrace{\theta\paren{\abs{\vect z}^2} \log\abs{\vect
      z}^{2n}}_{=:~\psi}
      \rparpht{-\sm\vphi_{K_X}
      +\paren{1+\beta}\theta\paren{\abs{\vect z}^2} \log\abs{\vect
      z}^{2n}}%
      \; \geq 0
  \end{align*}
  for every $\beta \in [0,\delta]$ for some constant $\delta >0$.
  Here $\theta \colon [0,1] \to [0,1]$ is a smooth cut-off function
  which is identically equal to $1$ on a neighbourhood of $0$ and
  vanishes outside of a larger neighbourhood.
  It can be seen that $\ell_0$ can be chosen independent of $k$, $\ell$ and
  $\ell'$, even the point $p$ (as $X$ is compact).
  
  With almost the same proof as in the proof of the first inequality
  in \eqref{eq:Berg_approx-ineq}, namely, applying the
  Ohsawa--Takegoshi extension theorem (or Example
  \ref{eg:extension-from-a-pt}) with respect to the potential
  $\ell k\mu \dep[\ell',k]\bphi_\Berg +\ell\sm\vphi_A$ to obtain a global
  section $f$ of $\ell k\mu\paren{K_X+S+B} +\ell A$ with prescribed value
  at any point $p \in
  X\setminus \paren{\dep[\ell',k]\bphi_\Berg}^{-1}\paren{-\infty}$ such that
  \begin{align*}
    \norm f_{\ell, k}^2
    &:= \int_X \abs f^2 e^{-\ell k\mu \sm\bphi
      -\ell\sm\vphi_A +\sm\vphi_{K_X}} \\
    &\lesssim \int_X
      \frac{\abs f^2}{\abs{\psi} \paren{\paren{\log\abs{\delta\psi}}^2 +1}}
      e^{-\ell k\mu \dep[\ell',k]\bphi_\Berg
      -\ell\sm\vphi_A +\sm\vphi_{K_X} -\psi} \\
    &\lesssim \paren{\abs f^2 e^{-\ell k\mu \dep[\ell',k]\bphi_\Berg
      -\ell\sm\vphi_A}}(p) \\
    \imply \quad
    e^{\ell k\mu \dep[\ell',k]\bphi_\Berg}(p)
    &\lesssim \paren{\frac{\abs f^2}{\norm f_{\ell,k}^2}
      e^{-\ell\sm\vphi_A}}(p)
      \leq \paren{\Berg_{\ell,k} e^{-\ell\sm\vphi_A}}(p) \; ,
  \end{align*}
  where constant in the first $\lesssim$ is independent of $k$,
  $\ell$ and $\ell'$ thanks to the fact that $\dep[\ell',k]\bphi_\Berg$ is
  bounded above uniformly in $k$ and $\ell'$ (see
  \eqref{eq:Berg_approx-ineq}) and the use of inequality
  \eqref{eq:xlogx-estimate} to estimate the terms with $\psi$, while
  the constant in the second $\lesssim$ is independent of $k$,
  $\ell$ and $\ell'$ thanks to the universality of the constant in the
  Ohsawa--Takegoshi extension.
  As a result, one sees that
  \begin{equation} \label{eq:pf-Berg-with-OT} \tag{$*$}
    \dep[\ell',k]\bphi_\Berg \lesssim_\tlog \dep[\ell,k]\bphi_\Berg \; ,
  \end{equation}
  where the constant involved in $\lesssim_\tlog$ is independent of
  $k$, $\ell$ and $\ell'$.

  For the reverse inequality, it follows easily by means of
  mean-value-inequality that, for any $m \in \Nnum$, 
  \begin{equation} \label{eq:pf-Berg-potential} \tag{$**$}
    \dep[\ell, k]\bphi_\Berg \lesssim_\tlog \dep[m\ell, k]\bphi_\Berg
  \end{equation}
  with the constant in $\lesssim_\tlog$ being independent of $k$,
  $\ell$ and $m$.
  Indeed, 
  for any fixed $x \in X$, take $h \in \cohgp 0[X]{\ell
    k\mu\paren{K_X+S+B}+\ell A}$ with $\norm h_{\ell,k}^2 := \norm
  h_{\ell k\mu \sm\bphi + \ell \sm\vphi_A}^2 = 1$ and $\abs{h(x)}^2 =
  \Berg_{\ell,k}(x)$.
  Then, one has
  \begin{equation*} 
    \begin{aligned}
      \abs{h^m(x)}^2
      &\leq \Berg_{m\ell, k}(x) \norm{h^m}_{m\ell, k}^2
      \leq \Berg_{m\ell, k}(x) \norm{h}_{\ell, k}^2 \paren{\sup_X
        \abs{h}^2 e^{-\ell k\mu \sm\bphi - \ell \sm\vphi_A}}^{m-1}
      \\
      &\overset{\mathclap{\text{mean-value-ineq.}}}{\leq} \qquad\;
      \Berg_{m\ell,k}(x) \paren{\frac C{\paren{\pi r^2}^n} 
        e^{\sup_{X} \paren{\ell k\mu\sm\bphi + \ell\sm\vphi_A} -
          \inf_X \paren{\ell k\mu\sm\bphi + \ell\sm\vphi_A}} }^{m-1}
    \end{aligned}
  \end{equation*}
  for some constant $C > 0$ and small $r > 0$ which are independent of
  $k$, $\ell$ and $m$.
  Note that $\sup_{X} \paren{\ell k\mu\sm\bphi + \ell\sm\vphi_A}$
  means the maximum of $\sup_{V_\gamma} \paren{\ell k\mu\sm\bphi +
    \ell\sm\vphi_A}$ among all $V_\gamma$ in a finite cover
  $\set{V_\gamma}_\gamma$ (and the same interpretation applies to the
  term with $\inf_X$).
  The claim for $\lesssim_\tlog$ follows after applying $\frac{1}{m\ell
    k\mu} \log$ on and subtracting $\frac{1}{k\mu} \sm\vphi_A$ from
  both sides.

  When $m \ell > \ell'$, one can apply \eqref{eq:pf-Berg-with-OT} to
  obtain $\dep[\ell,k]\bphi_\Berg \lesssim_\tlog
  \dep[m\ell,k]\bphi_\Berg \lesssim_\tlog \dep[\ell',k]\bphi_\Berg$.
  This completes the proof.
\end{proof}

The choice of the ample divisor $A$ does not play a role as soon as
only the asymptotic behaviour is concerned.
\begin{lemma} \label{lem:bphi-indep-of-A-and-lim-k-exists}
  Let $A$ and $A'$ be two arbitrary ample divisors.
  It follows that
  \begin{equation*}
    \dep[k'']\bphi_{\Berg,A'} \lesssim_\tlog \dep[k']\bphi_{\Berg,A}
    \lesssim_\tlog \dep\bphi_{\Berg,A'}
  \end{equation*}
  for any $k,k',k'' \in \Nnum$ whenever $k'A' - kA$ and $k''A - k'A'$
  are ample, where the constants involved in $\lesssim_\tlog$ are
  independent of $k$, $k'$ and $k''$.

  In particular, when $A =A'$, it follows that $\dep\bphi
  =\dep\bphi_{\Berg,A}$ is getting more singular as $k$ increases, i.e.
  \begin{equation*}
    \dep[k'']\bphi \lesssim_\tlog \dep[k']\bphi
    \lesssim_\tlog \dep\bphi
  \end{equation*}
  for any $k \leq k' \leq k''$.
\end{lemma} 

\begin{proof}  
  Take $\ell \gg 0$ such that $\ell k'A' - \ell kA$ is globally
  generated, and choose $\sm\vphi_{A}$ and $\sm\vphi_{A'}$ such that
  $k'\sm\vphi_{A'} = k\sm\vphi_A + \sm\vphi_{k'A' - k A}$, where
  $\sm\vphi_{k'A' - k A}$ is constructed from a basis of $\cohgp
  0[X]{\ell k'A' - \ell kA}$.
  It follows that $\Berg_{\ell, k'k, kA} \abs{s_{\ell k'A' - \ell kA}}^2 \leq
    \Berg_{\ell, k'k, k'A'}$
  for all $s_{\ell k'A' - \ell kA}$ in the basis of $\cohgp 0[X]{\ell
    k'A' - \ell kA}$.
  One gets, after summing up the inequalities for the whole basis of
  $\cohgp 0[X]{\ell k'A' - \ell kA}$ (with dimension being bounded
  above by $\BigO\paren{\ell^n\paren{k+k'}^n}$),
  \begin{equation*}
    \dep[\ell k, k']\bphi_{\Berg, A} = \dep[\ell, k'k]\bphi_{\Berg, kA}
    \lesssim_\tlog \dep[\ell, k'k]\bphi_{\Berg, k'A'} = \dep[\ell k',
    k]\bphi_{\Berg,A'} \; , 
  \end{equation*}
  hence the inequality on the right-hand-side in the claim after
  taking Lemma \ref{lem:bphi-indep-on-ell} into account, with the
  constant involved in $\lesssim_\tlog$ being independent of $\ell$,
  $k$ and $k'$.
  The other inequality follows by interchanging the role of $A$ and
  $A'$.
\end{proof}

By passing to a subsequence when necessary, one can assume that
$\seq{\dep\bphi}_{k\in \Nnum}$ converges in $\Lloc$ (thanks to
\eqref{eq:wt_properties-bdd-uniformly-in-k} and the fact that
$\dep\bphi \geq \bphi_\tmin \not\equiv -\infty$) to a psh (thanks to
\eqref{eq:wt_properties-curv}) potential $\dep[\infty]\bphi$,
which is given pointwisely by the upper regularised limit
\begin{equation*}
  \dep[\infty]\bphi(z) 
  := \reglim_{k\tendsto + \infty} \dep\bphi(z) 
  :=  \varlimsup_{\substack{\zeta \tendsto z \\  \text{($\zeta = z$ allowed)}}} \lim_{k
    \tendsto + \infty} \dep\bphi(\zeta)  \quad\text{for all } z \in X \; .
\end{equation*}
By the minimality of $\bphi_\tmin$, it follows that $\dep[\infty]\bphi \sim_\tlog
\bphi_\tmin$.
Indeed, it follows from \eqref{eq:Berg_approx-ineq} that, by letting
$r \to 0^+$ (after $k \to \infty$), one has $\bphi_\tmin \leq
\dep[\infty]\bphi \leq \sm\bphi$, and thus $\dep[\infty]\bphi =\bphi_\tmin$
by the minimality of $\bphi_\tmin$.


\subsection{The choice of $\dep\vphi_L$ and $\dep\psi$}
\label{sec:construction-vphi_L-psi}


Set $\numax := \max_{j \in I_S} \nu_j$.
For each $k \in \Nnum$, define the \emph{global function} $\dep{\psi}$
on $X$ and the potential $\dep\vphi_L$ on $L$ such that
\begin{align}
  \label{eq:psi}
  \dep{\psi}
  &:= \frac 1\numax \paren{\phi_{\nu_S \cdot S} +\phi_{D_2}
    +\sm\vphi_\rho -\dep{\bphi}} \quad\text{and} \\
  \label{eq:wt_on_L}
  \dep\vphi_L +\dep\psi
  &:= \paren{\mu - 1} \dep{\bphi} +\phi_S +\phi_B \; .
\end{align}
For the convenience of readers who would like to compare the current
choices with those in \cite{DHP}, define also $\vphi_{\tau_k}$,
$\vphi_{F, k}$ and
$\psi_{\nu_S \cdot S, \: k}$ (which may not follow the convention in
Notation \ref{notation:potentials}) by
\begin{equation*}
  \vphi_{\tau_k} := \dep\bphi \; , \quad
  \vphi_{F, k} := \paren{\mu - 1} \dep{\bphi} + \phi_B
                \quad\text{and}\quad 
  \psi_{\nu_S \cdot S,\: k} := \dep{\bphi} -\phi_{D_2} -\sm\vphi_\rho \; .
\end{equation*}

It follows from \eqref{eq:wt_properties-sing}
that $\dep\psi \leq 0$ on $X$, so, for any $m \leq m'$, it follows
that
\begin{equation*}
  \mtidlof{\dep\vphi_L + m' \dep\psi} \subset \mtidlof{\dep\vphi_L +
    m \dep\psi} \; ,
\end{equation*}
i.e.~the family $\set{\mtidlof{\dep\vphi_L + m \dep\psi}}_{m \in
  \fieldR}$ is decreasing.
As
\begin{equation} \label{eq:vphi_L+m-psi}
  \begin{aligned}
    \dep\vphi_L + m \dep\psi = \paren{\mu - 1 + \frac{1-m}{\numax}}
    \dep\bphi - \frac{1-m}{\numax} &\paren{\phi_{D_2} + \sm\vphi_\rho} \\
    &+ \phi_S - \frac{1-m}{\numax} \phi_{\nu_S\cdot S} + \phi_B \;
    ,
  \end{aligned} 
\end{equation}
it can be seen that $m = 1$ is a jumping number of the family by
considering the coefficients of $\phi_{S_j}$'s, after
taking \eqref{eq:S-not-in-polar_set} into account.
Since the coefficient of $\dep\bphi$ is decreasing as $m$ grows and
that of $\phi_{D_2}$ is negative as $m$ varies within $[0,1)$,
the decreasing family has to remain unchanged for $m \in [0,1)$.
In the context of Theorem
\ref{thm:ext-from-lc-with-estimate-codim-1-case}, one has $m_0 = 0$
and $m_1 = 1$ for all $k \in \Nnum$.
The ideal $\Ann_{\holo_X} \paren{ \frac{\mtidlof{\dep\vphi_L}}
  {\mtidlof{\dep\vphi_L + \dep\psi}} }$ obviously defines the reduced
subvariety $S$, which is already an snc divisor.

Each $\dep\vphi_L + \dep\psi$ is locally bounded from above uniformly
in $k$ since so is $\dep\bphi$ (see \eqref{eq:wt_on_L}).

Applying $\ibddbar$ to \eqref{eq:vphi_L+m-psi} and putting $m = 1 +\beta
:= 1 + \numax \lambda$, where $\lambda \in [0, \mu - 1]$, it follows from the
\textfr{Poincaré--Lelong} formula $\currInt E = \ibddbar \phi_E$ that
\begin{equation} \label{eq:curvature-for-final-step}
  \begin{aligned}
    \ibddbar \dep\vphi_L + \paren{1+\beta} \ibddbar \dep{\psi}
    &=~\paren{\mu - 1 - \lambda} \ibddbar \dep{\bphi} + \lambda
    \currInt{D_2 +\nu_S \cdot S}
    + \currInt{S +B} \\
    &\overset{\mathclap{\text{by
          \eqref{eq:wt_properties-curv}}}}{\geq} \quad\;\;
    -\frac{1}{k} \omega \quad \text{on }~X \; .
  \end{aligned}
\end{equation}
Setting $\delta_0 := (\mu - 1)\numax$ such that the above inequality
holds true when $\beta$ varies within $[0,\delta_0]$, this gives the
curvature assumption \ref{item:curv-cond-seq} in Theorem
\ref{thm:extension-with-seq-of-potentials}.
As $\delta_0$ is independent of $k$ and $\dep\psi$'s are bounded above
uniformly in $k$, the normalisation assumption
\ref{item:normalisation-seq} in Theorem
\ref{thm:extension-with-seq-of-potentials} can be made satisfied by
adding a suitable constant (independent of $k$) to each $\dep\psi$.

It remains to verify the $L^2$-ness of the given section to be
extended with respect to the $1$-lc-measure under the above choice of
metrics in order to invoke Theorem
\ref{thm:extension-with-seq-of-potentials}.


\subsection{The main technical lemma}
\label{sec:sup-norm-estimate-from-L2-norm-with-denom}


\begin{lemma} \label{lem:sup-norm-bdd-from-L2-with-denom}

  Suppose that $\delta > 0$ is a constant independent of $k$ and $U$
  is a section in $\cohgp 0[X]{k\mu \paren{K_X +S +B} +A}$ such that 
  \begin{equation*}
    \norm \holU_{\setE,s}^{\frac{2}{k\mu}\paren{1+\delta}}
    :=\int_X \frac{\abs U^{\frac
        2{k\mu} \paren{1+\delta}}}{\abs{\dep\psi}^{s}} e^{-\delta \dep\bphi
      -\phi_S-\phi_B -\frac 1{k\mu}\paren{1+\delta} \sm\vphi_A} \leq M
  \end{equation*}
  for some numbers $s > 0$ and $M > 0$.
  Then, one has
  \begin{equation*}
    \int_X \abs U^{\frac 1{k\mu}\paren{1+\delta}}
    e^{-\frac 12\paren{1+\delta} \paren{\sm\bphi +\frac 1{k\mu}
        \sm\vphi_A}} \:d\vol_{X,\omega}
    \lesssim
    \paren{M \int_X \abs{\dep\psi}^s \: d\vol_{X,\omega}}^{\frac 12}
    \; ,
  \end{equation*}
  where the constant involved in $\lesssim$ is independent of $k$.
  This in turn implies that
  \begin{equation*}
    \abs U^{\frac 2{k\mu}\paren{1+\delta}}
    e^{-\paren{1+\delta} \paren{\sm\bphi +\frac 1{k\mu}
        \sm\vphi_A}}
    \lesssim M \qquad\text{on }X \; ,
  \end{equation*}
  where the constant involved in $\lesssim$ is independent of $k$.
\end{lemma}

\begin{proof}
  Notice that
  \begin{equation*}
    \delta \dep\bphi +\phi_S+\phi_B+\frac 1{k\mu} (1+\delta) \sm\vphi_A
  \lesssim_\tlog \paren{1+\delta} \paren{\sm\bphi +\frac
    1{k\mu}\sm\vphi_A} -\sm\vphi_{K_X}
  \end{equation*}
  on $X$, with the constant involved in
  $\lesssim_\tlog$ being independent of $k$ since $\dep\bphi$ is
  locally bounded above uniformly in $k$ (see
  \eqref{eq:wt_properties-bdd-uniformly-in-k}) and $\delta$ is
  independent of $k$.
  The first claim then follows immediately from this inequality and
  the Cauchy--Schwarz inequality.
  
  An argument with the Harnack inequality for plurisubharmonic
  functions (see, for example, \cite{Demailly}*{Ch.~I, Prop.~4.22(b)})
  then yields
  \begin{align*}
    \abs U^{\frac 1{k\mu}\paren{1+\delta}}
    e^{-\frac 12\paren{1+\delta} \paren{\sm\bphi +\frac 1{k\mu}
    \sm\vphi_A}}
    &\lesssim
    \int_X \abs U^{\frac 1{k\mu}\paren{1+\delta}}
    e^{-\frac 12\paren{1+\delta} \paren{\sm\bphi +\frac 1{k\mu}
        \sm\vphi_A}} \:d\vol_{X,\omega} \\
    &\lesssim
    \paren{M \int_X \abs{\dep\psi}^s \: d\vol_{X,\omega}}^{\frac 12}
  \end{align*}
  on $X$, where the constants in both $\lesssim$'s are
  independent of $k$.
  
  It remains to show that $\int_X \abs{\dep\psi}^s \:d\vol_{X,\omega}$
  is bounded above uniformly in $k$.
  Since $\dep\psi \leq 0$ and $s > 0$, it follows that
  \begin{align*}
    \abs{\dep\psi}^s
    = \paren{-\dep\psi}^s \quad
    &\overset{\mathclap{\text{\eqref{eq:psi}}}}= \;\;\quad
      \frac 1{\numax^s} \paren{\dep\bphi -\phi_{\nu_S\cdot S}
      -\phi_{D_2} -\sm\vphi_\rho}^s \\
    &\leq \frac 1{\numax^s} \paren{C +\sm\bphi -\phi_{\nu_S\cdot S}
      -\phi_{D_2} -\sm\vphi_\rho}^s 
  \end{align*}
  where the constant $C$ is independent of $k$ by
  \eqref{eq:wt_properties-bdd-uniformly-in-k}.
  The far right-hand-side is independent of $k$ and is $L^1$ since
  $\phi_{\nu_S \cdot S} + \phi_{D_2}$ has only logarithmic poles and
  $\nu_S \cdot S + D_2$ has only snc.
  This completes the proof.
\end{proof}


\subsection{A lower bound for $\res{\dep\bphi}_S$}
\label{sec:lower-bound-for-bphi-on-S}


\newcommand{\bphiU}[1][\abs{\Umin}]{\bphi_{#1}}

For any $v_{k,A} \in \cohgp 0[S]{\holo_{S}\paren{k\mu \paren{K_X+S+B} +
    A}}$, consider the set 
\begin{equation} \label{eq:def_setE}
  \setE := \setE\paren{v_{k,A}}
  := \setd{\holU \in \cohgp 0[X]{k\mu\paren{K_X+S+B} + A}} {
    \holU|_{S} = v_{k,A}} \; .
\end{equation}
(The section $v_{k,A}$ will in fact be $\res{\orgu^k \otimes
  s_{A,i}}_{S}$ in application.)
For a given number $\delta >0$, define an $L^{\frac
  2{k\mu}\paren{1+\delta}}$-norm $\norm\cdot_\setE$ on $\setE$ such that
\begin{equation} \label{eq:norm-U_setE}
  \norm \holU_{\setE}^{\frac{2}{k\mu}\paren{1+\delta}}
  :=\int_X \frac{\abs\holU^{\frac 2{k\mu}\paren{1+\delta}}} {\abs{\dep\psi}^2}
  e^{-\delta \dep\bphi -\phi_S-\phi_B-\frac 1{k\mu}\paren{1+\delta} \sm\vphi_A} \; .
\end{equation} 
Note that indeed $\norm \holU_{\setE}^{\frac{2}{k\mu}\paren{1+\delta}} < \infty$ for all
$\holU \in \cohgp 0[X]{k\mu\paren{K_X+S+B} + A}$ since
\begin{itemize}
\item $\abs\holU^{\frac 2{k\mu} \delta} e^{-\delta \dep\bphi} \leq
  \paren{\norm{\holU^\ell}_{\ell k\mu \sm\bphi_{K_X+S+B}+\ell
      \sm\vphi_A}^2}^{\frac{\delta}{\ell k\mu}} <\infty$ as
  $\dep\bphi$ is a Bergman kernel potential (although the bound may
  depend on $k$),
\item $e^{-\phi_B}$ is integrable on the compact $X$,
\item the exponent on $\abs{\dep\psi}$ is $>1$ while
  $S \subset \paren{\dep\psi}^{-1}(-\infty)$ (see \eqref{eq:psi} and
  \eqref{eq:S-not-in-polar_set}), and
\item $S$ and $B$ having only snc.
\end{itemize} 
For convenience, define also 
\begin{equation} \label{eq:norm-v_setE_S}
  \norm{v_{k,A}}_{\setE,S}^{\frac 2{k\mu}\paren{1+\delta}} :=
  \int_{S} \abs{v_{k,A}}_\omega^{\frac 2{k\mu}\paren{1+\delta}}
  e^{-\delta\dep\bphi -\phi_B -\frac 1{k\mu}\paren{1+\delta}
    \sm\vphi_A} \:d\vol_{S, \omega} \; ,
\end{equation}
where $d\vol_{S,\omega} = \sum_{j\in I_S} d\vol_{S_j,\omega}$.
This norm of $v_{k,A}$ is also finite as $\res{e^{-\phi_B}}_S$ is
integrable on $S$, and 
$\parres{\abs{v_{k,A}}^{\frac 2{k\mu} \delta} e^{-\delta \dep\bphi}}_S
=\parres{\abs\holU^{\frac 2{k\mu} \delta} e^{-\delta \dep\bphi}}_S$
for any $\holU \in \setE$ and is thus bounded from above.
In order to have better control on the dependence of its upper bound on
$k$, \cite{DHP}*{Lemma 5.5} (which is due to \textde{Hörmander}
(\cite{Hormander}*{Thm.~4.4.5}) in its original form; the
form in \cite{DHP}*{Lemma 5.5} is due to Tian
(\cite{Tian_KE-on-Fano}*{Prop.~2.1})) is invoked to obtain the
following lemma.

\begin{lemma} \label{lem:uniform-delta}
  There exist constants $\delta > 0$ and $C > 0$ which are independent
  of $k$, $v_{k,A}$ and $\dep\bphi$ such that
  \begin{equation*}
    \norm{v_{k,A}}_{\setE,S}^{\frac 2{k\mu}\paren{1+\delta}}
    \leq C \sup_{S} \paren{\abs{v_{k,A}}_\omega^{\frac 2{k\mu}\paren{1+\delta}}
    e^{-\delta\sm\bphi -\sm\vphi_B -\frac 1{k\mu}\paren{1+\delta}
      \sm\vphi_A}} =:~C \norm{v_{k,A}}_\infty^{\frac{2}{k\mu} \paren{1+\delta}} \; .
  \end{equation*} 
\end{lemma}
\begin{proof}
  It follows readily from \textde{Hölder's} inequalities that
  \begin{align*}
    \norm{v_{k,A}}_{\setE,S}^{\frac 2{k\mu}\paren{1+\delta}}
    &\leq \norm{v_{k,A}}_\infty^{\frac{2}{k\mu} \paren{1+\delta}}
    \int_S e^{-\delta \:\paren{\dep\bphi -\sm\bphi} -\paren{\phi_B
      -\sm\vphi_B}} \:d\vol_{S,\omega} \\
    &\leq
      \norm{v_{k,A}}_\infty^{\frac{2}{k\mu} \paren{1+\delta}}
      \paren{\int_S e^{-\delta \frac q{q-1} \paren{\dep\bphi
      -\sm\bphi}} \:d\vol_{S,\omega}}^{1 -\frac 1q}
      \paren{\int_S e^{-q \:\paren{\phi_B -\sm\vphi_B}}
      \:d\vol_{S,\omega}}^{\frac 1q} \; ,
  \end{align*}
  and $q > 1$ is chosen sufficiently close to $1$ such that the last
  integral on the right-hand-side converges.

  \cite{DHP}*{Lemma 5.5} is applied to assure that there exist
  constants $\delta > 0$ and $C' > 0$, which depend only on the
  cohomology class of $\frac{q}{q-1} \ibddbar \sm\bphi$, such that
  \begin{equation*}
    \int_S e^{-\delta \frac q{q-1} \paren{\dep\bphi -\sm\bphi}}
    \:d\vol_{S,\omega}
    \leq C' \; .
  \end{equation*}
  This completes the proof.
\end{proof}

From now on, the constant $\delta$ is chosen to be the one provided by
Lemma \ref{lem:uniform-delta}.

\emph{Assume that $\setE$ is non-empty.}
There exists, by Lemma \ref{lem:sup-norm-bdd-from-L2-with-denom} and
Montel's Theorem, an element $\Umin_{k, A} \in \setE$ whose
$\norm\cdot_{\setE}$-norm attains the minimum on $\setE$.
Define
\begin{equation} \label{eq:def-bphiU}
  \dep\bphiU :=\frac{1}{k\mu} \paren{\log\abs{\Umin_{k,
        A}}^2 - \sm\vphi_A}
\end{equation}
as a potential on $K_X+S+B$, which, of course, depends on the choice
of $v_{k,A}$.


\begin{lemma} \label{lem:potential_Umin_uniBdd}
  Suppose a sequence $\seq{v_{k,A}}_{k\in \Nnum}$ of sections as
  described above is given.
  If $\setE\paren{v_{k,A}}$ is non-empty and $\Umin_{k,A}$ is chosen
  as above for every $k\in\Nnum$, then one has
  \begin{equation*}
    \norm{\Umin_{k,A}}_{\setE}^{\frac{2}{k\mu}\paren{1+\delta}}
    \lesssim \norm{v_{k,A}}_{\setE,S}^{\frac 2{k\mu}\paren{1+\delta}}
  \end{equation*}
  for every $k \in \Nnum$, where the constant involved in $\lesssim$
  is independent of $k$.
\end{lemma}

\begin{proof}
  The strategy is
  to invoke Ohsawa--Takegoshi extension theorem using
  $\dep\bphiU$ as a potential to obtain an extension of
  $v_{k,A}$ with the required estimate, and then argue by
  minimality.
  
  Write
  \begin{equation*}
    k\mu \paren{K_X+S+B} + A = K_X + \underbrace{\paren{k\mu
        -1}\paren{K_X+S+B} + S + B + A}_{=:~L_k}
  \end{equation*}
  and endow the line bundle $L_k$ with the potential $\dep\vphi_{L_k}$
  defined such that
  \begin{equation} \label{eq:def-of-vphi_Lk}
    \dep\vphi_{L_k} +\dep\psi
    :=\paren{k\mu -1 -\delta} \dep\bphiU +\delta \dep\bphi
    +\phi_{S} +\phi_{B} +\sm\vphi_A \; ,
  \end{equation}
  where $\dep\psi$ is the function defined in \eqref{eq:psi}, and
  $\delta > 0$ is the one provided by Lemma \ref{lem:uniform-delta}
  (which can be replaced by a smaller one if $k\mu -1 -\delta$ is
  negative).
  Notice that both $\res{\dep\bphiU}_S$ and
  $\res{\dep\bphi}_S$ are well-defined potentials on $S$.
  It can be seen that the family $\set{\mtidlof{\dep\vphi_{L_k} + m
      \dep\psi}}_{m \in \fieldR}$ jumps along $S$ when $m=1$.
  When $m$ increases, the coefficient of $\dep\bphiU$ in
  $\dep\vphi_{L_k} +m\dep\psi$ is unchanged while that of $\dep\bphi$
  decreases, one can then see that
  $\Ann_{\holo_X} \paren{\frac{\mtidlof{\dep\vphi_{L_k} +m_0
        \dep\psi}} {\mtidlof{\dep\vphi_{L_k} + \dep\psi}} }$ defines
  exactly $S$ (with no other subvarieties in $D_2$ or
  $\divsr{\Umin_{k,A}}$) for some number $m_0 \in [0,1)$.

  It follows from the \textfr{Poincaré--Lelong} formula
  $\currInt E = \ibddbar \phi_E$ that, for $\beta = \numax \lambda$,
  \begin{align*}
    \ibddbar\dep[k]\vphi_{L_k} + \paren{1+ \beta} \ibddbar \dep\psi
    &=
      \begin{aligned}[t]
        &\paren{1-\frac{1+\delta}{k\mu}} \currInt{\Umin_{k,A}}
        +\currInt S +\currInt B
        +\frac{1+\delta}{k\mu} \ibddbar\sm\vphi_A \\
        &+\paren{\delta -\lambda} \ibddbar\dep\bphi +\lambda
        \currInt{D_2} +\lambda \currInt{\nu_S \cdot S}
      \end{aligned}
    \\
    &\geq \frac{1+\delta}{k\mu} \ibddbar\sm\vphi_A -\frac{\delta
      -\lambda}{k\mu} \ibddbar\sm\vphi_A +\lambda \currInt{D_2}
      +\lambda \currInt{\nu_S \cdot S}
    \geq 0
  \end{align*}
  for all $\lambda \in [0,\delta]$.\footnote{This is precisely the
    place where $\delta > 0$ is needed, and thus the use of
    \cite{DHP}*{Lemma 5.5} cannot be avoided.}
  According to Proposition \ref{prop:lc-measure}, as the minimal lc
  centres of $(X,S)$ are of codimension $1$ (as the irreducible
  components $S_j$ of $S$ are disjoint), the norm of $v_{k,A}$
  under the $1$-lc-measure $\lcV|1|<\dep\vphi_{L_k}>[\dep\psi]$ (on
  codimension-$1$ lc centres of $(X,S)$) is
  \begin{align*}
    \int_S \abs{v_{k,A}}_\omega^2 \lcV|1|<\dep\vphi_{L_k}>[\dep\psi]
    &=\sum_{j \in I_S} \frac{\pi \numax}{\nu_j} \int_{S_j}
    \abs{v_{k,A}}_\omega^2 e^{-\paren{k\mu -1-\delta} \dep\bphiU
      -\delta \dep\bphi -\phi_B -\sm\vphi_A} d\vol_{S_j, \omega} \\
    &=\sum_{j\in I_S} \frac{\pi \numax}{\nu_j} \int_{S_j}
    \abs{v_{k,A}}_\omega^{\frac 2{k\mu}\paren{1+\delta}} e^{-\delta
      \dep\bphi -\phi_B -\frac 1{k\mu}\paren{1+\delta} \sm\vphi_A}
      d\vol_{S_j, \omega} \\
    &\lesssim \norm{v_{k,A}}_{\setE,S}^{\frac 2{k\mu}\paren{1+\delta}}
      \qquad\text{(see \eqref{eq:norm-v_setE_S} for the definition)} \; ,
  \end{align*}
  which is therefore finite, and the constant involved in $\lesssim$
  is independent of $k$.

  Since $\delta$ is independent of $k$, by adding a suitable constant
  independent of $k$ to $\dep\psi$, the normalisation assumption
  (\ref{item:normalisation-cond}) in Theorem
  \ref{thm:dbar-eq-with-estimate_sigma=1} (with $\ell = \delta$) can
  also be fulfilled  (see Remark \ref{rem:normalisation-control}).
  
  By the Ohsawa--Takegoshi extension theorem with lc-measure in
  Theorem \ref{thm:extension-sigma=1} (or via Demailly's version in
  \cite{Demailly_extension}*{Thm.~2.8}), one obtains a
  \emph{holomorphic extension} $V_k \in \setE$ of $v_{k,A}$
  on $X$ with estimate
  \begin{equation*}
    \int_X \frac{\abs{V_k}^2 e^{-\dep\vphi_{L_k} -\dep\psi}}
    {\abs{\dep\psi} \paren{\paren{\log\abs{\delta \dep\psi}}^2 +1}}
    \leq
    \int_S \abs{v_{k,A}}_\omega^2 \lcV|1|<\dep\vphi_{L_k}>[\dep\psi]
    \lesssim \norm{v_{k,A}}_{\setE,S}^{\frac 2{k\mu}\paren{1+\delta}}
  \end{equation*}
  where the constant involved in the last $\lesssim$ is independent of
  $k$ (and $\delta$).
  As $\frac 1{\abs{\dep\psi}^2} \lesssim \frac
  1{\abs{\dep\psi} \paren{\paren{\log\abs{\delta \dep\psi}}^2 +1}}$
  with the constant involved in $\lesssim$ being independent of $k$
  via a use of \eqref{eq:xlogx-estimate}, it follows that
  \begin{equation*}
    \norm{V_k}_{\dep\vphi_{L_k}}^2
    :=\int_X \frac{\abs{V_k}^2}{\abs{\dep\psi}^2} e^{-\dep\vphi_{L_k} -\dep\psi}
    \lesssim
    \norm{v_{k,A}}_{\setE,S}^{\frac 2{k\mu}\paren{1+\delta}} \; ,
  \end{equation*}
  where the constant in $\lesssim$ is independent of $k$.
    
  Considering the definitions of the potentials $\dep\bphiU$ in
  \eqref{eq:def-bphiU} and $\dep\vphi_{L_k}$ in
  \eqref{eq:def-of-vphi_Lk}, as well as the definition of the norm
  $\norm\cdot_\setE$ in \eqref{eq:norm-U_setE}, \textde{Hölder}'s
  inequality, together with the minimality of $\Umin_{k,A}$ in the
  norm $\norm\cdot_\setE$, yields
  \begin{equation*}
    \norm{\Umin_{k,A}}_{\setE}^{\frac 2{k\mu}\paren{1+\delta}}
    \leq \norm{V_k}_{\setE}^{\frac 2{k\mu}\paren{1+\delta}}
    \leq \paren{\norm{V_k}_{\dep\vphi_{L_k}}^2}^{\frac 1{k\mu}\paren{1+\delta}}
    \paren{\norm{\Umin_{k,A}}_{\setE}^{\frac
        2{k\mu}\paren{1+\delta}}}^{1 - \frac{1}{k\mu}\paren{1+\delta}}
    \; . 
  \end{equation*}
  Therefore, it follows that
  \begin{equation*}
    \norm{\Umin_{k, A}}_\setE^{\frac{2}{k\mu}\paren{1+\delta}}
    \leq
    \norm{V_k}_{\dep\vphi_{L_k}}^2
    \lesssim
    \norm{v_{k,A}}_{\setE,S}^{\frac 2{k\mu}\paren{1+\delta}} \; ,
  \end{equation*}
  where the constant involved in $\lesssim$ is independent of $k$.
\end{proof}

Combining the Lemmata \ref{lem:sup-norm-bdd-from-L2-with-denom},
\ref{lem:potential_Umin_uniBdd} and \ref{lem:uniform-delta},
one obtains that, if $\setE$ is non-empty, there exists a minimal
element $\Umin_{k,A} \in \setE$ such that
\begin{equation} \label{eq:Umin-sup-norm-estimate}
  \abs{\Umin_{k,A}}^{\frac 2{k\mu}\paren{1+\delta}}
  e^{-\paren{1+\delta} \sm\bphi -\frac 1{k\mu} \paren{1+\delta} \sm\vphi_A}(p) 
  \lesssim
  \norm{v_{k,A}}_\infty^{\frac 2{k\mu}\paren{1+\delta}}
\end{equation}
for any $p \in X$, where the constant involved in $\lesssim$ is
independent of $k$ and $p$.

Now the following improvement to \cite{DHP}*{Thm.~6.1} can be proved.

\begin{thm}[\thmparen{cf.~\cite{DHP}*{Thm.~6.1}}] \label{thm:S_j-not-in-polar-bphi}
  Under the setup in Section \ref{sec:basic-setup-plt} (\emph{without}
  the assumptions $\supp D_2 \subset \supp B$ and the
  existence of $\orgu \in
  \cohgp0[S]{\holo_S\paren{\mu \paren{K_X+S+B}}}$), 
  the potential $\res{\bphi_\tmin}_S$ is well-defined on every
  component of $S$.
\end{thm}

\begin{proof}

  \newcommand{\setF}{\mathcal{F}}
  \newcommand{\phiPmin}{\boldsymbol{\phi}_\tmin}

  Since $S_j \not\subset \sBase[-]{K_X+S+B}$ for every $j \in I_S$,
  using the Ohsawa--Takegoshi extension theorem, one can find that,
  for every $j \in I_S$
  and for every $k\in\Nnum$, there exist $\ell \gg 0$ and $W_{\ell k,
    \ell A} \in \cohgp 0[X]{\ell k\mu\paren{K_X+S+B} + \ell A}$ such
  that  
  $\res{W_{\ell k, \ell A}}_{S_j} \not\equiv 0$ but
  $\res{W_{\ell k, \ell A}}_{S_{j'}} \equiv 0$ for all $j' \in I_S
  \setminus \set j$.
  Set $v_{\ell k, \ell A} := \res{W_{\ell k,\ell A}}_S$ and
  renormalise it such that
  \begin{equation*}
    \norm{v_{\ell k, \ell A}}_\infty^{\frac
      2{\ell k\mu}\paren{1+\delta}}
    :=\sup_{S} \paren{\abs{v_{\ell k,\ell A}}_\omega^{\frac 2{\ell k\mu}\paren{1+\delta}}
      e^{-\delta\sm\bphi -\sm\vphi_B -\frac 1{k\mu}\paren{1+\delta}
        \sm\vphi_A}}
    = 1 \; .
  \end{equation*}
  Recall that $\delta > 0$ is chosen as in Lemma
  \ref{lem:uniform-delta}, so it is independent of $k$ in particular.

  Define the set $\setE := \setE\paren{v_{\ell k, \ell A}}$ as in
  \eqref{eq:def_setE}.
  This set is obviously non-empty, which implies the existence of the
  minimal element $\Umin_{\ell k, \ell A} \in \setE$ such that the
  estimate \eqref{eq:Umin-sup-norm-estimate} holds. 

  Since $\dep\bphi =\dep[\ell,k]\bphi$ is a Bergman kernel potential
  constructed from holomorphic sections in $\cohgp 0[X]{\ell k\mu \paren{K_X+S+B}
    +\ell A}$, it follows that
  \begin{equation*}
    \frac{1}{\ell k\mu} \paren{\log\abs{\Umin_{\ell
          k,\ell A}}^2 -\ell \sm\vphi_A}
    \leq \dep[\ell,k]\bphi
    +\frac{1}{\ell k\mu} \log \norm{\Umin_{\ell k,\ell A}}_{\ell k\mu
      \sm\bphi +\ell\sm\vphi_A}^2 \; ,
  \end{equation*}
  where
  \begin{align*}
    \paren{\norm{\Umin_{\ell k,\ell A}}_{\ell k\mu \sm\bphi
      +\ell\sm\vphi_A}^2}^{\frac 1{\ell k \mu}}
    &=\paren{\int_X \abs{\Umin_{\ell k,\ell A}}^2
    e^{-\ell k\mu \sm\bphi -\ell\sm\vphi_A} d\vol_{X,\omega}}^{\frac
      1{\ell k\mu}} \\
    &\overset{\mathclap{\text{by
      \eqref{eq:Umin-sup-norm-estimate}}}}\lesssim \quad\;\;
    \norm{v_{\ell k, \ell A}}_\infty^{\frac 2{\ell k\mu}} =1 \; .
  \end{align*}
  Therefore,
  \begin{equation*}
    \frac{1}{\ell k\mu} \paren{\log\abs{\Umin_{\ell
          k,\ell A}}^2 -\ell \sm\vphi_A}
    \lesssim_\tlog \dep[\ell,k]\bphi
  \end{equation*}
  on $X$.
  Notice that $v_{\ell k,\ell A} = \res{\Umin_{\ell k,\ell A}}_{S}$
  and the sup-norm of $v_{\ell k,\ell A}$ on $S$ is the same as the
  one on $S_j$.
  It follows that, after restricting to $S_j$ and adding suitable
  potentials (which are uniformly bounded in $k$) on both sides, one
  has
  \begin{equation*}
    0 =\log \sup_{S_j}
    \paren{\abs{v_{\ell k,\ell A}}_\omega^{\frac 2{\ell k\mu}}
      e^{-\frac\delta{\paren{1+\delta}} \sm\bphi
        -\frac 1{\paren{1+\delta}} \sm\vphi_B -\frac 1{k\mu} 
        \sm\vphi_A}} \;
    \lesssim_\tlog \; \sup_{S_j} \dep\bphi \; ,
  \end{equation*}
  where the constant involved in $\lesssim_\tlog$ is independent of
  $k$.
  Taking the upper regularised limit yields
  \begin{equation*}
    0 \lesssim_\tlog \; \sup_{S_j} \dep[\infty]\bphi \: = \;
    \sup_{S_j}\bphi_\tmin \; .
  \end{equation*}
  This implies that $\bphi_\tmin$ is well-defined on $S_j$.

  Since this holds true for every $j \in I_S$, $\bphi_\tmin$ is
  therefore well-defined on every component of $S$. 
\end{proof}

Recall that $A$ is chosen to be globally generated and
$\set{s_{A,i}}_{i \in I_A}$ is a basis of $\cohgp 0[X]{A}$ such that
$\sm\vphi_A = \log\paren{\sum_{i\in I_A} \abs{s_{A,i}}^2}$.
The following is the key result of this section.

\begin{thm} \label{thm:bphi-bdd-below-by-u-on-S}
  Under the setup given in Section \ref{sec:basic-setup-plt} (\emph{without}
  the assumption $\supp D_2 \subset \supp B$),
  suppose that there is a section $\orgu \in
  \cohgp0[S]{\holo_S\paren{\mu\paren{K_X+S+B}}}$ and that the sets
  $\setE\paren{\orgu^k \otimes \res{s_{A,i}}_{S}}$ constructed
  as in \eqref{eq:def_setE} are non-empty for all $k \in \Nnum$, $i
  \in I_A$.
  Then, one has
  \begin{equation*}
    \log\abs\orgu^{\frac 2\mu} \lesssim_\tlog \res{\dep\bphi}_S
    \quad\text{on }S
  \end{equation*}
  for all $k$ with the constant involved in $\lesssim_\tlog$ being
  independent of $k$, and therefore
  \begin{equation*}
    \log\abs\orgu^{\frac 2\mu} \lesssim_\tlog \res{\bphi_\tmin}_S
    \quad\text{on }S \; .
  \end{equation*} 
\end{thm}

\begin{proof}
  The proof goes almost as in the proof of Theorem
  \ref{thm:S_j-not-in-polar-bphi}.

  For fixed $k \in \Nnum$ and $i \in I_A$, the set
  $\setE :=\setE\paren{\orgu^k \otimes \res{s_{A,i}}_S}$ being
  non-empty implies the existence of the minimal element
  $\Umin_{k,A,i} \in \setE$ such that the estimate
  \eqref{eq:Umin-sup-norm-estimate} holds.

  The potential $\dep\bphi =\dep[\ell,k]\bphi$ being a Bergman kernel
  potential implies that, for $\ell \gg 0$,
  \begin{equation*}
    \frac{1}{ k\mu} \paren{\log\abs{\Umin_{k,A,i}}^2 -\sm\vphi_A}
    \leq \dep\bphi
    +\frac{1}{\ell k\mu} \log \norm{\paren{\Umin_{k,A,i}}^\ell}_{\ell k\mu
      \sm\bphi +\ell \sm\vphi_A}^2 \; ,
  \end{equation*}
  where
  \begin{align*}
    \paren{\norm{\paren{\Umin_{k,A,i}}^\ell}_{\ell k\mu \sm\bphi
    +\ell \sm\vphi_A}^2}^{\frac 1{\ell k \mu}}
    &=\paren{\int_X \abs{\Umin_{k,A,i}}^{2\ell}
      e^{-\ell k\mu \sm\bphi -\ell\sm\vphi_A} d\vol_{X,\omega}}^{\frac
      1{\ell k\mu}} \\
    &\overset{\mathclap{\text{by
      \eqref{eq:Umin-sup-norm-estimate}}}}\lesssim \quad\;\;
      \norm{\orgu^k \otimes s_{A,i}}_\infty^{\frac 2{ k\mu}}
    \\
    &=\sup_S \paren{\abs{\orgu^k \otimes s_{A,i}}_\omega^{\frac 2{ k\mu}}
      e^{-\frac{\delta}{1+\delta}\sm\bphi -\frac 1{1+\delta}\sm\vphi_B
      -\frac{1}{k\mu} \sm\vphi_A}} \\
    &\lesssim 1 
  \end{align*}
  with the constants involved in both $\lesssim$'s being
  independent of $k$.
  Therefore,
  \begin{equation*}
    \begin{aligned}
      &&\frac{1}{k\mu} \paren{\log\abs{\Umin_{k,A,i}}^2 -\sm\vphi_A}
      &\lesssim_\tlog \dep\bphi &&\text{on }X \\
      \imply& \quad
      &\log\abs{\orgu}^{\frac 2\mu} +\frac{1}{k\mu}
      \log \parres{\abs{s_{A,i}}^2 e^{ -\sm\vphi_A}}_S
      &\lesssim_\tlog \res{\dep\bphi}_S &&\text{on }S \\
      \imply& \quad
      &\abs{\orgu}^{2k} \parres{\abs{s_{A,i}}^2 e^{
          -\sm\vphi_A}}_S
      &\lesssim \;\; \res{e^{k\mu\dep\bphi}}_S
      &&\text{on }S \; ,
    \end{aligned}
  \end{equation*}
  where the constants involved in $\lesssim_\tlog$'s and $\lesssim$
  are independent of $k$.
  Summing up the last inequality over $i\in I_A$ yields the
  inequality of the first claim.

  The second claim follows from taking the upper regularised limit of
  $\dep\bphi$ as $k \to \infty$ together with the fact that
  $\dep[\infty] \bphi = \bphi_\tmin$.
\end{proof}


\subsection{Proof of the main theorem}
\label{sec:proof-improved-plt}


Let $\pi \colon \rs X \to X$ be a log-resolution of $(X,S+B)$ such
that
\begin{equation*} 
  K_{\rs X} + \rs S + \rs B = \pi^*\paren{K_X + S + B} + \rs E \; ,
\end{equation*}
where $\rs S$, $\rs B$ and $\rs E$ are effective $\fieldQ$-divisors
with no common components such that $\rs S = \floor{\rs S +\rs B}$ and
$\pi(\rs S) = S$. 
Let $\Xi:= \paren{ N_\sigma\paren{\lVert K_{\rs X}+ \rs S+ \rs
    B \rVert_{\rs S}} \wedge \rs B|_{\rs S} }$ be the
\emph{extension obstruction divisor} introduced in \cite{DHP}
(see \cite{DHP}*{\S 2.1} for the definition of $N_\sigma\paren{\lVert
  K_{\rs X}+ \rs S+ \rs B \rVert_{\rs S}}$). 
The corresponding \emph{extension obstruction ideal sheaf} is defined
as
\begin{equation*}
  \obstidlSHM := \setd{w \in \holo_S}{
    \begin{aligned}
      \exists~&\text{log-resolution } \pi \colon \rs X \to X \text{
        s.t.~irred.~comp.~of $\rs B$ }
      \\
      &\text{are disjoint and }
      \divsr{\pi^*w} + \mu \widetilde E|_{\widetilde S} \geq \mu \Xi 
    \end{aligned}
  } \; .
\end{equation*}
Notice that, when $K_X+S+B$ is nef, one has $\Xi = 0$, and thus
$\obstidlSHM = \holo_S$.

\begin{thm} \label{thm:proof-improved-plt}
  Under the setup in Section \ref{sec:basic-setup-plt} (\emph{without}
  the assumption $\supp D_2 \subset \supp B$), every $\orgu
  \in \cohgp0[S]{\holo_S\paren{\mu\paren{K_X+S+B}} \otimes
    \obstidlSHM}$ extends to a holomorphic section in
  $\cohgp0[X]{\mu\paren{K_X+S+B}}$.

  In particular, when $K_X+S+B$ is nef, the restriction map
  $\cohgp0[X]{\mu\paren{K_X+S+B}} \to
  \cohgp0[S]{\holo_S\paren{\mu\paren{K_X+S+B}}}$ is surjective.
\end{thm}

\begin{proof}
  Let $\orgu \in \cohgp0[S]{\holo_S\paren{\mu\paren{K_X+S+B}} \otimes
    \obstidlSHM}$ and define $\setE_{k,i} := \setE\paren{u^k \otimes
    \res{s_{A,i}}_S}$ as in \eqref{eq:def_setE}.

  By the result of Hacon and M${}^{\text{c}}$Kernan in
  \cite{Hacon&Mckernan}*{Thm.~6.3} (see also
  \cite{Hacon&Mckernan_bddness_pluriK-map}*{Thm.~3.16}) on the
  extension of pluricanonical sections, in which the technique was
  originated in the work of Siu
  (\nocite{Siu_inv_plurigenera}\cite{Siu_inv_plurigenera2}), it
  follows that, on the plt pair $(X,S+B)$, the set $\setE_{k,i}=
  \setE\paren{u^k \otimes \res{s_{A,i}}_S}$ are non-empty for every
  $k\in\Nnum$ and $i \in I_A$ when $u$ is a section to the ideal
  $\obstidlSHM$.

  Then, given the choice of $\dep\psi$ and $\dep\vphi_L$ in
  \eqref{eq:psi} and \eqref{eq:wt_on_L}, Theorem
  \ref{thm:bphi-bdd-below-by-u-on-S} assures that
  \begin{equation*}
    \int_S \abs \orgu^2 \lcV|1|<\dep\vphi_L>[\dep\psi]
    =\sum_{j\in I_S} \frac{\pi\numax}{\nu_j} \int_{S_j}
    \abs\orgu_\omega^2 e^{-\paren{\mu-1} \dep\bphi -\phi_B}
    d\vol_{S_j,\omega}
  \end{equation*}
  and its limit as $k \to \infty$, which is
  \begin{equation*}
    \sum_{j\in I_S} \frac{\pi\numax}{\nu_j} \int_{S_j}
    \abs\orgu_\omega^2 e^{-\paren{\mu-1} \bphi_\tmin -\phi_B}
    d\vol_{S_j,\omega}
  \end{equation*}
  up to a constant multiple, are \emph{finite} (as $e^{-\phi_B}$ is
  integrable), which verifies the $L^2$ assumption in Theorem
  \ref{thm:extension-with-seq-of-potentials}.
  
  The inequality \eqref{eq:curvature-for-final-step} verifies the
  curvature assumption \ref{item:curv-cond-seq} in Theorem
  \ref{thm:extension-with-seq-of-potentials} and provides a $\delta_0
  := \paren{\mu-1}\numax$ which is independent of $k$.
  Considering the definition of $\dep\psi$ in \eqref{eq:psi} and the
  fact that $\dep[\infty]\bphi = \bphi_\tmin$, which is bounded from
  above, one sees that a uniform constant can be added to $\dep\psi$
  such that the normalisation assumption \ref{item:normalisation-seq}
  in Theorem \ref{thm:extension-with-seq-of-potentials} is satisfied.
  Therefore, it follows from Theorem
  \ref{thm:extension-with-seq-of-potentials} that there exists a
  holomorphic extension $U$ of $\orgu$ with the estimate
  \begin{equation*}
    \int_X \frac{\abs U^2 e^{-\paren{\mu -1}\bphi_\tmin -\phi_S
        -\phi_B }} {\abs{\dep[\infty]\psi} 
      \paren{\paren{\log\abs{\dep[\infty]\psi}}^2 +1}} 
    \leq
    \sum_{j\in I_S} \frac{\pi\numax}{\nu_j} \int_{S_j}
    \abs\orgu_\omega^2 e^{-\paren{\mu-1} \bphi_\tmin -\phi_B}
    d\vol_{S_j,\omega} \; .
  \end{equation*} 
  This completes the proof.
\end{proof}


\ifannalen

  \section*{Declarations}

  On behalf of all authors, the corresponding author states that there
  is no conflict of interest.
  
  \noindent
  Data sharing is not applicable to this article as no datasets were
  generated or analysed during the current study. 






\fi



\begin{bibdiv}
  \begin{biblist}
    \IfFileExists{references.ltb}{
      \bibselect{references}
    }{
      \bib{ASWY-gKingsFormula}{article}{
  author={Andersson, Mats},
  author={Samuelsson Kalm, H\aa kan},
  author={Wulcan, Elizabeth},
  author={Yger, Alain},
  title={Segre numbers, a generalized King formula, and local intersections},
  journal={J. Reine Angew. Math.},
  volume={728},
  date={2017},
  pages={105--136},
  issn={0075-4102},
  review={\MR {3668992}},
  doi={10.1515/crelle-2014-0109},
}

\bib{Berndtsson&Lempert}{article}{
  author={Berndtsson, Bo},
  author={Lempert, L\'{a}szl\'{o}},
  title={A proof of the Ohsawa-Takegoshi theorem with sharp estimates},
  journal={J. Math. Soc. Japan},
  volume={68},
  date={2016},
  number={4},
  pages={1461--1472},
  issn={0025-5645},
  review={\MR {3564439}},
  doi={10.2969/jmsj/06841461},
}

\bib{Bjork&Samuelsson}{article}{
  author={Bj{\"o}rk, Jan-Erik},
  author={Samuelsson, H{\aa }kan},
  title={Regularizations of residue currents},
  journal={J. Reine Angew. Math.},
  volume={649},
  date={2010},
  pages={33--54},
  issn={0075-4102},
  review={\MR {2746465}},
  doi={10.1515/CRELLE.2010.087},
}

\bib{Blocki_Suita-conj}{article}{
  author={B\l ocki, Zbigniew},
  title={Suita conjecture and the Ohsawa-Takegoshi extension theorem},
  journal={Invent. Math.},
  volume={193},
  date={2013},
  number={1},
  pages={149--158},
  issn={0020-9910},
  review={\MR {3069114}},
  doi={10.1007/s00222-012-0423-2},
}

\bib{Cao&Demailly&Matsumura}{article}{
  author={Cao, JunYan},
  author={Demailly, Jean-Pierre},
  author={Matsumura, Shinichi},
  title={A general extension theorem for cohomology classes on non reduced analytic subspaces},
  journal={Sci. China Math.},
  volume={60},
  date={2017},
  number={6},
  pages={949--962},
  issn={1674-7283},
  review={\MR {3647124}},
  doi={10.1007/s11425-017-9066-0},
}

\bib{Cao&Paun_OT-ext}{article}{
  author={Cao, Junyan},
  author={P\u aun, Mihai},
  title={On the Ohsawa--Takegoshi extension theorem},
  eprint={arXiv:2002.04968 [math.CV]},
  pages={36},
  date={2020},
  note={with an appendix by Bo Berndtsson},
}

\bib{Chi_extension}{article}{
  author={Chi, Chen-Yu},
  title={On the extension of holomorphic sections from reduced unions of strata of divisors},
  pages={28},
  eprints={arXiv:1905.06481v3 [math.AG]},
  date={2019},
}

\bib{Demailly_complete-Kahler}{article}{
  author={Demailly, Jean-Pierre},
  title={Estimations $L^{2}$\ pour l'op\'erateur $\bar \partial $\ d'un fibr\'e vectoriel holomorphe semi-positif au-dessus d'une vari\'et\'e k\"ahl\'erienne compl\`ete},
  language={French},
  journal={Ann.~Sci.~\'Ecole Norm.~Sup.~(4)},
  volume={15},
  date={1982},
  number={3},
  pages={457--511},
  issn={0012-9593},
  review={\MR {690650}},
}

\bib{Demailly_regularization}{article}{
  author={Demailly, Jean-Pierre},
  title={Regularization of closed positive currents and intersection theory},
  journal={J. Algebraic Geom.},
  volume={1},
  date={1992},
  number={3},
  pages={361--409},
  issn={1056-3911},
  review={\MR {1158622}},
}

\bib{Demailly_Analytic-methods}{book}{
  author={Demailly, Jean-Pierre},
  title={Analytic methods in algebraic geometry},
  series={Surveys of Modern Mathematics},
  volume={1},
  publisher={International Press, Somerville, MA; Higher Education Press, Beijing},
  date={2012},
  pages={viii+231},
  isbn={978-1-57146-234-3},
  review={\MR {2978333}},
}

\bib{Demailly}{webpage}{
  author={Demailly, Jean-Pierre},
  title={Complex analytic and differential geometry},
  note={OpenContent Book},
  url={http://www-fourier.ujf-grenoble.fr/~demailly/manuscripts/agbook.pdf},
  date={2012},
}

\bib{Demailly_extension}{article}{
  author={Demailly, Jean-Pierre},
  title={Extension of holomorphic functions defined on non reduced analytic subvarieties},
  conference={ title={The legacy of Bernhard Riemann after one hundred and fifty years. Vol. I}, },
  book={ series={Adv. Lect. Math. (ALM)}, volume={35}, publisher={Int. Press, Somerville, MA}, },
  date={2016},
  pages={191--222},
  review={\MR {3525916}},
  eprint={arXiv:1510.05230 [math.CV]},
}

\bib{DPS94}{article}{
  author={Demailly, Jean-Pierre},
  author={Peternell, Thomas},
  author={Schneider, Michael},
  title={Compact complex manifolds with numerically effective tangent bundles},
  journal={J. Algebraic Geom.},
  volume={3},
  date={1994},
  number={2},
  pages={295--345},
  issn={1056-3911},
  review={\MR {1257325}},
}

\bib{DPS01}{article}{
  author={Demailly, Jean-Pierre},
  author={Peternell, Thomas},
  author={Schneider, Michael},
  title={Pseudo-effective line bundles on compact K\"{a}hler manifolds},
  journal={Internat. J. Math.},
  volume={12},
  date={2001},
  number={6},
  pages={689--741},
  issn={0129-167X},
  review={\MR {1875649}},
  doi={10.1142/S0129167X01000861},
}

\bib{DHP}{article}{
  author={Demailly, Jean-Pierre},
  author={Hacon, Christopher D.},
  author={P\u {a}un, Mihai},
  title={Extension theorems, non-vanishing and the existence of good minimal models},
  journal={Acta Math.},
  volume={210},
  date={2013},
  number={2},
  pages={203--259},
  issn={0001-5962},
  review={\MR {3070567}},
  doi={10.1007/s11511-013-0094-x},
}

\bib{Fujino_abundance-3fold}{article}{
  author={Fujino, Osamu},
  title={On abundance theorem for semi log canonical threefolds},
  journal={Proc. Japan Acad. Ser. A Math. Sci.},
  volume={75},
  date={1999},
  number={6},
  pages={80--84},
  issn={0386-2194},
  review={\MR {1712649}},
}

\bib{Fujino&Gongyo_abundance}{article}{
  author={Fujino, Osamu},
  author={Gongyo, Yoshinori},
  title={Log pluricanonical representations and the abundance conjecture},
  journal={Compos.~Math.},
  volume={150},
  date={2014},
  number={4},
  pages={593--620},
  issn={0010-437X},
  review={\MR {3200670}},
  doi={10.1112/S0010437X13007495},
}

\bib{Gongyo&Matsumura}{article}{
  author={Gongyo, Yoshinori},
  author={Matsumura, Shinichi},
  title={Versions of injectivity and extension theorems},
  language={English, with English and French summaries},
  journal={Ann. Sci. \'{E}c. Norm. Sup\'{e}r. (4)},
  volume={50},
  date={2017},
  number={2},
  pages={479--502},
  issn={0012-9593},
  review={\MR {3621435}},
  doi={10.24033/asens.2325},
}

\bib{Grauert&Remmert-CAS}{book}{
  author={Grauert, Hans},
  author={Remmert, Reinhold},
  title={Coherent analytic sheaves},
  series={Grundlehren der Mathematischen Wissenschaften [Fundamental Principles of Mathematical Sciences]},
  volume={265},
  publisher={Springer-Verlag, Berlin},
  date={1984},
  pages={xviii+249},
  isbn={3-540-13178-7},
  review={\MR {755331}},
  doi={10.1007/978-3-642-69582-7},
}

\bib{Grauert&Remmert}{book}{
  author={Grauert, Hans},
  author={Remmert, Reinhold},
  title={Theory of Stein spaces},
  series={Classics in Mathematics},
  note={Translated from the German by Alan Huckleberry; Reprint of the 1979 translation},
  publisher={Springer-Verlag, Berlin},
  date={2004},
  pages={xxii+255},
  isbn={3-540-00373-8},
  review={\MR {2029201}},
  doi={10.1007/978-3-642-18921-0},
}

\bib{Guan&Zhou_optimal-L2-estimate}{article}{
  author={Guan, Qi'an},
  author={Zhou, Xiangyu},
  title={A solution of an $L^2$ extension problem with an optimal estimate and applications},
  journal={Ann.~of Math.~(2)},
  volume={181},
  date={2015},
  number={3},
  pages={1139--1208},
  issn={0003-486X},
  review={\MR {3296822}},
  doi={10.4007/annals.2015.181.3.6},
}

\bib{Guan&Zhou_effective_openness}{article}{
  author={Guan, Qi'an},
  author={Zhou, Xiangyu},
  title={Effectiveness of Demailly's strong openness conjecture and related problems},
  journal={Invent. Math.},
  volume={202},
  date={2015},
  number={2},
  pages={635--676},
  issn={0020-9910},
  review={\MR {3418242}},
  doi={10.1007/s00222-014-0575-3},
}

\bib{Guenancia}{article}{
  author={Guenancia, Henri},
  title={Toric plurisubharmonic functions and analytic adjoint ideal sheaves},
  journal={Math. Z.},
  volume={271},
  date={2012},
  number={3-4},
  pages={1011--1035},
  issn={0025-5874},
  review={\MR {2945594}},
  doi={10.1007/s00209-011-0900-0},
}

\bib{Hacon&Mckernan_bddness_pluriK-map}{article}{
  author={Hacon, Christopher D.},
  author={M${}^{\text {c}}$Kernan, James},
  title={Boundedness of pluricanonical maps of varieties of general type},
  journal={Invent. Math.},
  volume={166},
  date={2006},
  number={1},
  pages={1--25},
  issn={0020-9910},
  review={\MR {2242631}},
  doi={10.1007/s00222-006-0504-1},
}

\bib{Hacon&Mckernan}{article}{
  author={Hacon, Christopher D.},
  author={M${}^{\text {c}}$Kernan, James},
  title={Existence of minimal models for varieties of log general type. II},
  journal={J. Amer. Math. Soc.},
  volume={23},
  date={2010},
  number={2},
  pages={469--490},
  issn={0894-0347},
  review={\MR {2601040}},
  doi={10.1090/S0894-0347-09-00651-1},
}

\bib{Hiep_openness}{article}{
  author={Hiep, Pham Hoang},
  title={The weighted log canonical threshold},
  language={English, with English and French summaries},
  journal={C. R. Math. Acad. Sci. Paris},
  volume={352},
  date={2014},
  number={4},
  pages={283--288},
  issn={1631-073X},
  review={\MR {3186914}},
  doi={10.1016/j.crma.2014.02.010},
}

\bib{Hironaka}{article}{
  author={Hironaka, Heisuke},
  title={Resolution of singularities of an algebraic variety over a field of characteristic zero. I, II},
  journal={Ann. of Math. (2) {\bf 79} (1964), 109--203; ibid. (2)},
  volume={79},
  date={1964},
  pages={205--326},
  issn={0003-486X},
  review={\MR {0199184}},
}

\bib{Hormander}{book}{
  author={H\"{o}rmander, Lars},
  title={An introduction to complex analysis in several variables},
  edition={Second revised edition},
  note={North-Holland Mathematical Library, Vol. 7},
  publisher={North-Holland Publishing Co., Amsterdam-London; American Elsevier Publishing Co., Inc., New York},
  date={1973},
  pages={x+213},
  review={\MR {0344507}},
}

\bib{KimDano_Thesis}{book}{
  author={Kim, Dano},
  title={Extension of pluriadjoint sections from a log-canonical center},
  note={Thesis (Ph.D.)--Princeton University},
  publisher={ProQuest LLC, Ann Arbor, MI},
  date={2007},
  pages={62},
  isbn={978-0549-00791-3},
  review={\MR {2710171}},
}

\bib{KimDano_lc-extension}{article}{
  author={Kim, Dano},
  title={$L^2$ extension of adjoint line bundle sections},
  language={English, with English and French summaries},
  journal={Ann. Inst. Fourier (Grenoble)},
  volume={60},
  date={2010},
  number={4},
  pages={1435--1477},
  issn={0373-0956},
  review={\MR {2722247}},
}

\bib{KimDano-adjIdl}{article}{
  author={Kim, Dano},
  title={Themes on non-analytic singularities of plurisubharmonic functions},
  conference={ title={Complex analysis and geometry}, },
  book={ series={Springer Proc. Math. Stat.}, volume={144}, publisher={Springer, Tokyo}, },
  date={2015},
  pages={197--206},
  review={\MR {3446757}},
  doi={10.1007/978-4-431-55744-9\texttt {\_}14},
}

\bib{Koike15_min-sing-metric_nef}{article}{
  author={Koike, Takayuki},
  title={On the minimality of canonically attached singular Hermitian metrics on certain nef line bundles},
  journal={Kyoto J. Math.},
  volume={55},
  date={2015},
  number={3},
  pages={607--616},
  issn={2156-2261},
  review={\MR {3395981}},
  doi={10.1215/21562261-3089091},
}

\bib{Kollar_Resoln-of-sing}{book}{
  author={Koll\'{a}r, J\'{a}nos},
  title={Lectures on resolution of singularities},
  series={Annals of Mathematics Studies},
  volume={166},
  publisher={Princeton University Press, Princeton, NJ},
  date={2007},
  pages={vi+208},
  isbn={978-0-691-12923-5},
  isbn={0-691-12923-1},
  review={\MR {2289519}},
}

\bib{Kollar_Sing-of-MMP}{book}{
  author={Koll\'{a}r, J\'{a}nos},
  title={Singularities of the minimal model program},
  series={Cambridge Tracts in Mathematics},
  volume={200},
  note={With a collaboration of S\'{a}ndor Kov\'{a}cs},
  publisher={Cambridge University Press, Cambridge},
  date={2013},
  pages={x+370},
  isbn={978-1-107-03534-8},
  review={\MR {3057950}},
  doi={10.1017/CBO9781139547895},
}

\bib{Lempert_openness}{article}{
  author={Lempert, L\'{a}szl\'{o}},
  title={Modules of square integrable holomorphic germs},
  conference={ title={Analysis meets geometry}, },
  book={ series={Trends Math.}, publisher={Birkh\"{a}user/Springer, Cham}, },
  date={2017},
  pages={311--333},
  review={\MR {3773623}},
  eprint={arXiv:1404.0407 [math.CV]},
}

\bib{Matsumura_injectivity}{article}{
  author={Matsumura, Shinichi},
  title={An injectivity theorem with multiplier ideal sheaves of singular metrics with transcendental singularities},
  journal={J. Algebraic Geom.},
  volume={27},
  date={2018},
  number={2},
  pages={305--337},
  issn={1056-3911},
  review={\MR {3764278}},
  doi={10.1090/jag/687},
}

\bib{McNeal&Varolin_adjunction}{article}{
  author={McNeal, Jeffery D.},
  author={Varolin, Dror},
  title={Analytic inversion of adjunction: $L^2$ extension theorems with gain},
  language={English, with English and French summaries},
  journal={Ann. Inst. Fourier (Grenoble)},
  volume={57},
  date={2007},
  number={3},
  pages={703--718},
  issn={0373-0956},
  review={\MR {2336826}},
}

\bib{Mustata_resoln-sing}{article}{
  author={Musta\c {t}\u {a}, Mircea},
  title={Introduction to resolution of singularities},
  conference={ title={Analytic and algebraic geometry}, },
  book={ series={IAS/Park City Math. Ser.}, volume={17}, publisher={Amer. Math. Soc., Providence, RI}, },
  date={2010},
  pages={405--449},
  review={\MR {2743820}},
}

\bib{Ohsawa-V}{article}{
  author={Ohsawa, Takeo},
  title={On the extension of $L^2$ holomorphic functions. V. Effects of generalization},
  journal={Nagoya Math. J.},
  volume={161},
  date={2001},
  pages={1--21},
  issn={0027-7630},
  review={\MR {1820210}},
  doi={10.1017/S0027763000022108},
}

\bib{Ohsawa-example}{article}{
  author={Ohsawa, Takeo},
  title={On a curvature condition that implies a cohomology injectivity theorem of Koll\'{a}r-Skoda type},
  journal={Publ. Res. Inst. Math. Sci.},
  volume={41},
  date={2005},
  number={3},
  pages={565--577},
  issn={0034-5318},
  review={\MR {2153535}},
}

\bib{Ohsawa&Takegoshi-I}{article}{
  author={Ohsawa, Takeo},
  author={Takegoshi, Kensh\={o}},
  title={On the extension of $L^2$ holomorphic functions},
  journal={Math. Z.},
  volume={195},
  date={1987},
  number={2},
  pages={197--204},
  issn={0025-5874},
  review={\MR {892051}},
  doi={10.1007/BF01166457},
}

\bib{Samuelsson_residue}{article}{
  author={Samuelsson, H\aa kan},
  title={Analytic continuation of residue currents},
  journal={Ark. Mat.},
  volume={47},
  date={2009},
  number={1},
  pages={127--141},
  issn={0004-2080},
  review={\MR {2480918}},
  doi={10.1007/s11512-008-0086-9},
}

\bib{Siu}{article}{
  author={Siu, Yum Tong},
  title={Complex-analyticity of harmonic maps, vanishing and Lefschetz theorems},
  journal={J. Differential Geom.},
  volume={17},
  date={1982},
  number={1},
  pages={55--138},
  issn={0022-040X},
  review={\MR {658472}},
}

\bib{Siu_inv_plurigenera}{article}{
  author={Siu, Yum-Tong},
  title={Invariance of plurigenera},
  journal={Invent. Math.},
  volume={134},
  date={1998},
  number={3},
  pages={661--673},
  issn={0020-9910},
  review={\MR {1660941}},
  doi={10.1007/s002220050276},
}

\bib{Siu_inv_plurigenera2}{article}{
  author={Siu, Yum-Tong},
  title={Extension of twisted pluricanonical sections with plurisubharmonic weight and invariance of semipositively twisted plurigenera for manifolds not necessarily of general type},
  conference={ title={Complex geometry}, address={G\"{o}ttingen}, date={2000}, },
  book={ publisher={Springer, Berlin}, },
  date={2002},
  pages={223--277},
  review={\MR {1922108}},
}

\bib{Tian_KE-on-Fano}{article}{
  author={Tian, Gang},
  title={On K\"{a}hler-Einstein metrics on certain K\"{a}hler manifolds with $C_1(M)>0$},
  journal={Invent.~Math.},
  volume={89},
  date={1987},
  number={2},
  pages={225--246},
  issn={0020-9910},
  review={\MR {894378}},
  doi={10.1007/BF01389077},
}

\bib{Varolin_lecture-note-on-L2}{webpage}{
  author={Varolin, Dror},
  title={Minischool on $L^2$ extension},
  note={lecture note distributed in the ``Workshop on $L^2$ Extensions'' at the University of Tokyo from 13 to 17 February, 2016},
  pages={46},
  date={2016},
  url={https://www.math.stonybrook.edu/~dror/L2-minischool.pdf},
}


    }
  \end{biblist}
\end{bibdiv}
\end{document}